\documentclass[11pt]{amsart}
\usepackage{dynkin-diagrams}
\usepackage[utf8]{inputenc}
\usepackage{comment}
\usepackage{booktabs}
\usepackage{multirow}
\usepackage{multicol}

\usepackage[style=alphabetic,firstinits,backend=bibtex, uniquename=init,
url=true,hyperref]{biblatex}
\DeclareFieldFormat{labelalpha}{\thefield{entrykey}}
\DeclareFieldFormat{extraalpha}{}
\DeclareFieldFormat{postnote}{#1}
\DeclareFieldFormat{multipostnote}{#1}
\bibliography{grob-v12}
\usepackage{tcolorbox}
\tcbuselibrary{skins, breakable, listings}
\usepackage{algorithm}
\usepackage{algpseudocode}
\usepackage{caption}
\usepackage{bbold}
\usepackage[bbgreekl]{mathbbol}
\usepackage{bbm}
\usepackage{tikz}
\usetikzlibrary{external}
\usepackage{subcaption}
\usepackage{tkz-graph}
\usepackage{tikz-cd}
\usepackage{cancel}
\usetikzlibrary{decorations.text,calc,tikzmark,arrows.meta,positioning}
\usepackage[table]{xcolor}
\usepackage{stmaryrd}
\usepackage{array}
\usepackage{nicematrix}
\usepackage{varwidth}
\usepackage{filecontents}
\usepackage[titletoc]{appendix}
\usepackage{mathtools}
\usepackage{amsthm}
\usepackage{amsmath}
\usepackage{mathabx}
\usepackage{amsfonts}
\usepackage{amssymb}
\usepackage{mathrsfs}
\usepackage{multirow}
\usepackage{xypic}
\usepackage{color}
\usepackage{xcolor}
\definecolor{revisionpurple}{RGB}{112,48,160}
\definecolor{revisionbrown}{RGB}{122,73,44}
\definecolor{revisionsilver}{RGB}{128,128,128}
\definecolor{revisiongreen}{RGB}{0,100,60}
\definecolor{revisionvthirteen}{RGB}{112,48,160}
\definecolor{revisionvthirteenold}{RGB}{122,73,44}
\newcommand{\MINOREDITED}[1]{#1}
\newcommand{\MAJOREDIT}[1]{{\color{revisiongreen}#1}}
\usepackage{graphicx}
\usepackage{extarrows}
\usepackage{titletoc}
\usepackage[colorlinks]{hyperref}
\definecolor{cclr}{rgb}{25,25,112}
\hypersetup{
citecolor=[rgb]{0.15, 0.15, 0.68},
linkcolor=gray,
urlcolor=magenta}

\newtheorem{thm}{Theorem}[section]
\newtheorem{lem}[thm]{Lemma}
\newtheorem{sublem}{Sublemma}[thm]
\theoremstyle{definition}
\newtheorem{prop}[thm]{Proposition}
\newtheorem{cor}[thm]{Corollary}

\newtheorem{defn}[thm]{Definition}

\newtheorem{notation}[thm]{Notation}
\newtheorem{example}[thm]{Example}

\newtheorem{remark}[thm]{Remark}
\newtheorem{setup}[thm]{Setup}

\newcommand{\doublestroke}[1]{\pdfliteral{1 Tr .3 w}#1\pdfliteral{0 Tr 0 w}}
\newcommand{\BPhi}{\doublestroke{\Phi}}
\newcommand{\BPsi}{\doublestroke{\Psi}}
\newcommand{\BK}{\doublestroke{K}}
\newcommand{\Balpha}{\doublestroke{\alpha}}
\newcommand{\Bbeta}{\doublestroke{\beta}}
\newcommand{\Bdelta}{\doublestroke{\delta}}

\newcommand{\BDelta}{%
    \scalebox{.8}{$\Delta$}%
    \llap{%
            \scalebox{.9}{$\Delta$}%
    }%
}

\newcommand{\Z}{\mathbb{Z}}
\newcommand{\Q}{\mathbb{Q}}
\newcommand{\G}{\mathbb{G}}

\newcommand{\fG}{\mathfrak{G}}

\newcommand{\fB}{\mathfrak{B}}
\newcommand{\fC}{\mathfrak{C}}

\newcommand{\fD}{\mathfrak{D}}
\newcommand{\cU}{\mathcal{U}}

\newcommand{\cC}{\mathcal{C}}
\newcommand{\sC}{\mathscr{C}}

\newcommand{\F}{\mathbb{F}}
\newcommand{\bF}{\mathbb{F}}
\newcommand{\A}{\mathbb{A}}
\newcommand{\bfE}{\mathbf{E}}
\newcommand{\bfA}{\mathbf{A}}
\newcommand{\wt}[1]{\widetilde{#1}}

\newcommand{\lK}{\underset{K}{<}}
\renewcommand{\SS}{{\operatorname{ss}}}

\newcommand{\Ext}{{\operatorname{Ext}}}
\newcommand{\abs}{{\operatorname{abs}}}
\newcommand{\rot}{{\operatorname{rot}}}
\newcommand{\ad}{{\operatorname{ad}}}
\newcommand{\pcrys}{{\operatorname{pcrys}}}
\newcommand{\pst}{{\operatorname{pst}}}

\newcommand{\leK}{\underset{K}{\le}}

\newcommand{\bFp}{\bar{\mathbb{F}}_p}
\newcommand{\Fp}{\mathbb{F}_p}

\newcommand{\cO}{\mathcal{O}}
\newcommand{\cX}{\mathcal{X}}
\newcommand{\fX}{\mathfrak{X}}
\newcommand{\cY}{\mathcal{Y}}

\newcommand{\cW}{\mathcal{W}}
\newcommand{\sW}{\mathscr{W}}
\newcommand{\fW}{\mathfrak{W}}

\newcommand{\bG}{\breve{G}}
\newcommand{\bM}{\breve{M}}
\newcommand{\bB}{\breve{B}}
\newcommand{\bT}{\breve{T}}
\newcommand{\bZ}{\breve{Z}}

\newcommand{\univ}{{\operatorname{univ}}}
\newcommand{\LT}{{\operatorname{LT}}}

\newcommand{\cone}{{\operatorname{cone}}}
\newcommand{\CH}{{\operatorname{CH}}}
\newcommand{\Top}{{\operatorname{top}}}
\newcommand{\cyc}{{\operatorname{cyc}}}
\newcommand{\dkF}{[k_{F_{\cyc}}:\F_p]}
\newcommand{\dF}{[F:\Q_p]}

\DeclareMathOperator{\src}{src}
\DeclareMathOperator{\std}{std}
\DeclareMathOperator{\tar}{tar}
\DeclareMathOperator{\LHS}{LHS}
\DeclareMathOperator{\Span}{span}

\DeclareMathOperator{\PGL}{PGL}
\DeclareMathOperator{\Mat}{Mat}
\DeclareMathOperator{\Sp}{Sp}

\DeclareMathOperator{\Frob}{Frob}

\DeclareMathOperator{\EG}{EG}

\DeclareMathOperator{\triv}{triv}

\DeclareMathOperator{\un}{un}
\DeclareMathOperator{\Ht}{ht}
\DeclareMathOperator{\GL}{GL}
\DeclareMathOperator{\Lie}{Lie}
\DeclareMathOperator{\tr}{tr}
\DeclareMathOperator{\WD}{WD}
\DeclareMathOperator{\red}{red}
\DeclareMathOperator{\Spec}{Spec}
\DeclareMathOperator{\Spf}{Spf}

\DeclareMathOperator{\Tor}{Tor}

\DeclareMathOperator{\Hom}{Hom}

\DeclareMathOperator{\Aut}{Aut}
\DeclareMathOperator{\Gal}{Gal}
\DeclareMathOperator{\tame}{tame}

\DeclareMathOperator{\Res}{Res}

\DeclareMathOperator{\st}{st}

\newcommand{\Herr}{\mathrm{Herr}}

\begin{document}
\title{Algorithmic universal mod $p$ \'etale $(\varphi, \Gamma)$-modules}

\author{Zhongyipan Lin}
\date{\today}

\newcommand{\tikzGtwo}{
\begin{tikzpicture}[
    vertex/.style={circle, draw, minimum size=0.6cm, fill=white, inner sep=1pt, font=\sffamily}, scale=0.97,
    edge/.style={->, >=stealth, thick}
]
    \node[vertex] (v0) at (0.60, 1.50) {01};
    \node[vertex] (v1) at (-0.60, 1.50) {10};
    \node[vertex] (v2) at (0.00, 3.00) {11};
    \node[vertex] (v3) at (0.00, 4.50) {21};
    \node[vertex] (v4) at (0.00, 6.00) {31};
    \node[vertex] (v5) at (0.00, 7.50) {32};

    \draw[edge] (v0) -- (v2);
    \draw[edge] (v1) -- (v2);
    \draw[edge] (v2) -- (v3);
    \draw[edge] (v3) -- (v4);
    \draw[edge] (v4) -- (v5);
\end{tikzpicture}
}

\newcommand{\tikzBfour}{
\begin{tikzpicture}[
    vertex/.style={circle, draw, minimum size=0.6cm, fill=white, inner sep=1pt, font=\sffamily\tiny}, scale=0.65,
    edge/.style={->, >=stealth, thick}
]
    \node[vertex] (v0) at (1.80, 1.50) {0001};
    \node[vertex] (v1) at (0.60, 1.50) {0010};
    \node[vertex] (v2) at (1.20, 3.00) {0011};
    \node[vertex] (v3) at (1.20, 4.50) {0012};
    \node[vertex] (v4) at (-0.60, 1.50) {0100};
    \node[vertex] (v5) at (0.00, 3.00) {0110};
    \node[vertex] (v6) at (0.00, 4.50) {0111};
    \node[vertex] (v7) at (0.60, 6.00) {0112};
    \node[vertex] (v8) at (0.60, 7.50) {0122};
    \node[vertex] (v9) at (-1.80, 1.50) {1000};
    \node[vertex] (v10) at (-1.20, 3.00) {1100};
    \node[vertex] (v11) at (-1.20, 4.50) {1110};
    \node[vertex] (v12) at (-0.60, 6.00) {1111};
    \node[vertex] (v13) at (-0.60, 7.50) {1112};
    \node[vertex] (v14) at (0.00, 9.00) {1122};
    \node[vertex] (v15) at (0.00, 10.50) {1222};

    \draw[edge] (v0) -- (v2);
    \draw[edge] (v1) -- (v2);
    \draw[edge] (v1) -- (v5);
    \draw[edge] (v2) -- (v3);
    \draw[edge] (v2) -- (v6);
    \draw[edge] (v3) -- (v7);
    \draw[edge] (v4) -- (v5);
    \draw[edge] (v4) -- (v10);
    \draw[edge] (v5) -- (v6);
    \draw[edge] (v5) -- (v11);
    \draw[edge] (v6) -- (v7);
    \draw[edge] (v6) -- (v12);
    \draw[edge] (v7) -- (v8);
    \draw[edge] (v7) -- (v13);
    \draw[edge] (v8) -- (v14);
    \draw[edge] (v9) -- (v10);
    \draw[edge] (v10) -- (v11);
    \draw[edge] (v11) -- (v12);
    \draw[edge] (v12) -- (v13);
    \draw[edge] (v13) -- (v14);
    \draw[edge] (v14) -- (v15);
\end{tikzpicture}
}

\newcommand{\tikzDfour}{
\begin{tikzpicture}[
    vertex/.style={circle, draw, minimum size=0.6cm, inner sep=1pt, font=\sffamily},
    edge/.style={->, >=stealth, thick},
    box/.style={rounded corners, dashed, thick, draw=gray, fill=gray!10}
]
    \draw[box] (-1.85, 2.55) rectangle (1.85, 4.95);
    \node[vertex, fill=white] (v0) at (1.80, 1.50) {0001};
    \node[vertex, fill=white] (v1) at (0.60, 1.50) {0010};
    \node[vertex, fill=white] (v2) at (-0.60, 1.50) {0100};
    \node[vertex, fill=orange!40] (v3) at (0.00, 3.00) {0110};
    \node[vertex, fill=cyan!40] (v4) at (1.20, 3.00) {0101};
    \node[vertex, fill=orange!40] (v5) at (1.20, 4.50) {0111};
    \node[vertex, fill=white] (v6) at (-1.80, 1.50) {1000};
    \node[vertex, fill=white] (v7) at (-1.20, 3.00) {1100};
    \node[vertex, fill=white] (v8) at (0.00, 4.50) {1110};
    \node[vertex, fill=cyan!40] (v9) at (-1.20, 4.50) {1101};
    \node[vertex, fill=white] (v10) at (0.00, 6.00) {1111};
    \node[vertex, fill=white] (v11) at (0.00, 7.50) {1211};

    \draw[edge] (v0) -- (v4);
    \draw[edge] (v1) -- (v3);
    \draw[edge] (v2) -- (v3);
    \draw[edge] (v2) -- (v4);
    \draw[edge] (v2) -- (v7);
    \draw[edge] (v3) -- (v5);
    \draw[edge] (v3) -- (v8);
    \draw[edge] (v4) -- (v5);
    \draw[edge] (v4) -- (v9);
    \draw[edge] (v5) -- (v10);
    \draw[edge] (v6) -- (v7);
    \draw[edge] (v7) -- (v8);
    \draw[edge] (v7) -- (v9);
    \draw[edge] (v8) -- (v10);
    \draw[edge] (v9) -- (v10);
    \draw[edge] (v10) -- (v11);
\end{tikzpicture}
}

\newcommand{\tikzFfour}{
\begin{tikzpicture}[
    vertex/.style={circle, draw, minimum size=0.6cm, inner sep=1pt, font=\sffamily\tiny}, scale=0.6,
    edge/.style={->, >=stealth, thick},
    box/.style={rounded corners, dashed, thick, draw=gray, fill=gray!10}
]
    \draw[box] (-1.85, 3.95) rectangle (1.85, 6.55);
    \node[vertex, fill=white] (v0) at (-1.80, 1.50) {0001};
    \node[vertex, fill=white] (v1) at (1.80, 1.50) {0010};
    \node[vertex, fill=white] (v2) at (-1.20, 3.00) {0011};
    \node[vertex, fill=white] (v3) at (0.60, 1.50) {0100};
    \node[vertex, fill=white] (v4) at (1.20, 3.00) {0110};
    \node[vertex, fill=white] (v5) at (-1.20, 4.50) {0111};
    \node[vertex, fill=cyan!40] (v6) at (1.20, 4.50) {0120};
    \node[vertex, fill=cyan!40] (v7) at (-1.20, 6.00) {0121};
    \node[vertex, fill=white] (v8) at (1.20, 7.50) {0122};
    \node[vertex, fill=white] (v9) at (-0.60, 1.50) {1000};
    \node[vertex, fill=white] (v10) at (0.00, 3.00) {1100};
    \node[vertex, fill=orange!40] (v11) at (0.00, 4.50) {1110};
    \node[vertex, fill=white] (v12) at (0.00, 6.00) {1111};
    \node[vertex, fill=orange!40] (v13) at (1.20, 6.00) {1120};
    \node[vertex, fill=white] (v14) at (-1.20, 7.50) {1121};
    \node[vertex, fill=white] (v15) at (0.60, 9.00) {1122};
    \node[vertex, fill=white] (v16) at (0.00, 7.50) {1220};
    \node[vertex, fill=white] (v17) at (-0.60, 9.00) {1221};
    \node[vertex, fill=white] (v18) at (0.60, 10.50) {1222};
    \node[vertex, fill=white] (v19) at (-0.60, 10.50) {1231};
    \node[vertex, fill=white] (v20) at (0.00, 12.00) {1232};
    \node[vertex, fill=white] (v21) at (1.30, 12.50) {1242};
    \node[vertex, fill=white] (v22) at (2.60, 13.00) {1342};
    \node[vertex, fill=white] (v23) at (3.90, 13.50) {2342};

    \draw[edge] (v0) -- (v2);
    \draw[edge] (v1) -- (v2);
    \draw[edge] (v1) -- (v4);
    \draw[edge] (v2) -- (v5);
    \draw[edge] (v3) -- (v4);
    \draw[edge] (v3) -- (v10);
    \draw[edge] (v4) -- (v5);
    \draw[edge] (v4) -- (v6);
    \draw[edge] (v4) -- (v11);
    \draw[edge] (v5) -- (v7);
    \draw[edge] (v5) -- (v12);
    \draw[edge] (v6) -- (v7);
    \draw[edge] (v6) -- (v13);
    \draw[edge] (v7) -- (v8);
    \draw[edge] (v7) -- (v14);
    \draw[edge] (v8) -- (v15);
    \draw[edge] (v9) -- (v10);
    \draw[edge] (v10) -- (v11);
    \draw[edge] (v11) -- (v12);
    \draw[edge] (v11) -- (v13);
    \draw[edge] (v12) -- (v14);
    \draw[edge] (v13) -- (v14);
    \draw[edge] (v13) -- (v16);
    \draw[edge] (v14) -- (v15);
    \draw[edge] (v14) -- (v17);
    \draw[edge] (v15) -- (v18);
    \draw[edge] (v16) -- (v17);
    \draw[edge] (v17) -- (v18);
    \draw[edge] (v17) -- (v19);
    \draw[edge] (v18) -- (v20);
    \draw[edge] (v19) -- (v20);
    \draw[edge] (v20) -- (v21);
    \draw[edge] (v21) -- (v22);
    \draw[edge] (v22) -- (v23);
\end{tikzpicture}
}

\newcommand{\tikzESixTwo}{
\begin{tikzpicture}[
    vertex/.style={circle, draw, minimum size=0.6cm, inner sep=1pt, font=\sffamily\tiny},
    edge/.style={->, >=stealth, thick},
    box/.style={rounded corners, dashed, thick, draw=gray, fill=gray!10}
]
    \draw[box] (-0.45, 2.55) rectangle (2.85, 4.95);

    \node[vertex, fill=white] (v2) at (-2.40, 3.00) {000011};
    \node[vertex, fill=white] (v4) at (0.00, 3.00) {000110};
    \node[vertex, fill=white] (v5) at (-2.40, 4.50) {000111};
    \node[vertex, fill=orange!40] (v7) at (1.20, 3.00) {001100};
    \node[vertex, fill=white] (v8) at (1.20, 4.50) {001110};
    \node[vertex, fill=cyan!40] (v11) at (2.40, 3.00) {010100};
    \node[vertex, fill=orange!40] (v12) at (2.40, 4.50) {011100};
    \node[vertex, fill=cyan!40] (v13) at (0.00, 4.50) {010110};
    \node[vertex, fill=white] (v21) at (-1.20, 3.00) {101000};
    \node[vertex, fill=white] (v22) at (-1.20, 4.50) {101100};

    \draw[edge] (v2) -- (v5);
    \draw[edge] (v4) -- (v5);
    \draw[edge] (v4) -- (v8);
    \draw[edge] (v4) -- (v13);
    \draw[edge] (v7) -- (v8);
    \draw[edge] (v7) -- (v12);
    \draw[edge] (v7) -- (v22);
    \draw[edge] (v11) -- (v12);
    \draw[edge] (v11) -- (v13);
    \draw[edge] (v21) -- (v22);
\end{tikzpicture}
}

\newcommand{\tikzESixThree}{
\begin{tikzpicture}[
    vertex/.style={circle, draw, minimum size=0.6cm, inner sep=1pt, font=\sffamily\tiny},
    edge/.style={->, >=stealth, thick},
    box/.style={rounded corners, dashed, thick, draw=gray, fill=gray!10}
]

    \node[vertex, fill=red!40] (v5) at (-2.40, 4.50) {000111};
    \node[vertex, fill=green!40] (v12) at (1.20, 4.50) {011100};
    \node[vertex, fill=cyan!40] (v13) at (2.40, 4.50) {010110};
    \node[vertex, fill=teal!40] (v8) at (0.00, 4.50) {001110};
    \node[vertex, fill=red!40] (v9) at (0.00, 6.00) {001111};
    \node[vertex, fill=teal!40] (v14) at (2.40, 6.00) {011110};
    \node[vertex, fill=cyan!40] (v16) at (-2.40, 6.00) {010111};
    \node[vertex, fill=orange!40] (v22) at (-1.20, 4.50) {101100};
    \node[vertex, fill=green!40] (v23) at (-1.20, 6.00) {111100};
    \node[vertex, fill=orange!40] (v24) at (1.20, 6.00) {101110};

    \draw[edge] (v5) -- (v9);
    \draw[edge] (v5) -- (v16);
    \draw[edge] (v8) -- (v9);
    \draw[edge] (v8) -- (v14);
    \draw[edge] (v8) -- (v24);
    \draw[edge] (v12) -- (v14);
    \draw[edge] (v12) -- (v23);
    \draw[edge] (v13) -- (v14);
    \draw[edge] (v13) -- (v16);
    \draw[edge] (v22) -- (v23);
    \draw[edge] (v22) -- (v24);
\end{tikzpicture}
}

\newcommand{\tikzEEightFive}{
\begin{tikzpicture}[
    vertex/.style={circle, draw, minimum size=0.6cm, inner sep=1pt, font=\sffamily\tiny},
    edge/.style={->, >=stealth, thick}, scale=1.2,
    box/.style={rounded corners, dashed, thick, draw=gray, fill=gray!10}
]

    \node[vertex, fill=pink!40] (v28) at (3.60, 7.50) {01111100};
    \node[vertex, fill=green!40] (v31) at (2.40, 7.50) {01011110};
    \node[vertex, fill=teal!40] (v19) at (1.20, 7.50) {00111110};
    \node[vertex, fill=orange!40] (v14) at (0.00, 7.50) {00011111};
    \node[vertex, fill=orange!40] (v20) at (2.40, 9.00) {00111111};
    \node[vertex, fill=cyan!40] (v26) at (-2.40, 7.50) {01121000};
    \node[vertex, fill=pink!40] (v29) at (-1.20, 9.00) {01121100};
    \node[vertex, fill=teal!40] (v32) at (3.60, 9.00) {01111110};
    \node[vertex, fill=green!40] (v36) at (0.00, 9.00) {01011111};
    \node[vertex, fill=purple!40] (v47) at (-3.60, 7.50) {11111000};
    \node[vertex, fill=cyan!40] (v48) at (-3.60, 9.00) {11121000};
    \node[vertex, fill=red!40] (v50) at (-1.20, 7.50) {10111100};
    \node[vertex, fill=purple!40] (v51) at (-2.40, 9.00) {11111100};
    \node[vertex, fill=red!40] (v58) at (1.20, 9.00) {10111110};

    \draw[edge] (v14) -- (v20);
    \draw[edge] (v14) -- (v36);
    \draw[edge] (v19) -- (v20);
    \draw[edge] (v19) -- (v32);
    \draw[edge] (v19) -- (v58);
    \draw[edge] (v26) -- (v29);
    \draw[edge] (v26) -- (v48);
    \draw[edge] (v28) -- (v29);
    \draw[edge] (v28) -- (v32);
    \draw[edge] (v28) -- (v51);
    \draw[edge] (v31) -- (v32);
    \draw[edge] (v31) -- (v36);
    \draw[edge] (v47) -- (v48);
    \draw[edge] (v47) -- (v51);
    \draw[edge] (v50) -- (v51);
    \draw[edge] (v50) -- (v58);
\end{tikzpicture}
}

\newcommand{\FwPosetBCBC}{
\begin{tikzpicture}[
    >=stealth, scale=0.8,
    every node/.style={
        draw,
        rounded corners=2pt,
        inner sep=2pt,
        font=\small
    }
]
\node[fill=blue!20] (0001-0121) at (-6.00,-8.00) {0001-0121};
\node[fill=blue!20] (0011-0111) at (-3.00,-8.00) {0011-0111};
\node[fill=blue!20] (1000-1220) at (0.00,-8.00) {1000-1220};
\node[fill=blue!20] (1100-1120) at (3.00,-8.00) {1100-1120};
\node[fill=blue!20] (1110) at (6.00,-8.00) {1110};
\node[fill=blue!20] (0122) at (-3.00,-6.00) {0122};
\node[fill=blue!20] (1111-1231) at (0.00,-6.00) {1111-1231};
\node[fill=blue!20] (1121-1221) at (3.00,-6.00) {1121-1221};
\node[fill=blue!20] (1122-1342) at (-3.00,-4.00) {1122-1342};
\node[fill=blue!20] (1222-1242) at (0.00,-4.00) {1222-1242};
\node[fill=blue!20] (1232) at (3.00,-4.00) {1232};
\node[fill=blue!20] (2342) at (0.00,-2.00) {2342};

\draw[->] (0001-0121) -- (0122);
\draw[->] (0001-0121) -- (1111-1231);
\draw[->] (0001-0121) -- (1121-1221);
\draw[->] (0011-0111) -- (0122);
\draw[->] (0011-0111) -- (1111-1231);
\draw[->] (0011-0111) -- (1121-1221);
\draw[->] (0122) -- (1122-1342);
\draw[->] (0122) -- (1222-1242);
\draw[->] (0122) -- (1232);
\draw[->] (1000-1220) -- (1111-1231);
\draw[->] (1000-1220) -- (1121-1221);
\draw[->] (1100-1120) -- (1111-1231);
\draw[->] (1100-1120) -- (1121-1221);
\draw[->] (1110) -- (1111-1231);
\draw[->] (1110) -- (1121-1221);
\draw[->] (1111-1231) -- (1122-1342);
\draw[->] (1111-1231) -- (1222-1242);
\draw[->] (1111-1231) -- (1232);
\draw[->] (1121-1221) -- (1122-1342);
\draw[->] (1121-1221) -- (1222-1242);
\draw[->] (1121-1221) -- (1232);
\draw[->] (1122-1342) -- (2342);
\draw[->] (1222-1242) -- (2342);
\draw[->] (1232) -- (2342);
\end{tikzpicture}
}

\newcommand{\FwPosetBCD}{
\begin{tikzpicture}[
    >=stealth, scale=1.7,
    every node/.style={
        draw,
        rounded corners=2pt,
        inner sep=2pt,
        font=\small
    }
]
\node[fill=blue!20] (1000-1100-1122-1220-1242-1342) at (-3.00,-3.00) {1000-1100-1122-1220-1242-1342};
\node[fill=blue!20] (1110-1111-1121-1221-1231-1232) at (0.00,-3.00) {1110-1111-1121-1221-1231-1232};
\node[fill=blue!20] (1120-1222) at (3.00,-3.00) {1120-1222};
\node[fill=blue!20] (2342) at (0.00,-2.00) {2342};

\draw[->] (1000-1100-1122-1220-1242-1342) -- (2342);
\draw[->] (1110-1111-1121-1221-1231-1232) -- (2342);
\draw[->] (1120-1222) -- (2342);
\end{tikzpicture}
}
\newcommand{\FwPosetBCDB}{
\begin{tikzpicture}[
    >=stealth, scale=1,
    every node/.style={
        draw,
        rounded corners=2pt,
        inner sep=2pt,
        font=\small
    }
]
\node[fill=blue!20] (1000-1220-1222) at (-7.50,-4.00) {1000-1220-1222};
\node[fill=blue!20] (1100) at (-4.50,-4.00) {1100};
\node[fill=blue!20] (1110-1111-1221) at (-1.50,-4.00) {1110-1111-1221};
\node[fill=blue!20] (1120-1122-1342) at (1.50,-4.00) {1120-1122-1342};
\node[fill=blue!20] (1121-1231-1232) at (4.50,-4.00) {1121-1231-1232};
\node[fill=blue!20] (1242) at (7.50,-4.00) {1242};
\node[fill=blue!20] (2342) at (0.00,-2.00) {2342};

\draw[->] (1000-1220-1222) -- (2342);
\draw[->] (1100) -- (2342);
\draw[->] (1110-1111-1221) -- (2342);
\draw[->] (1120-1122-1342) -- (2342);
\draw[->] (1121-1231-1232) -- (2342);
\draw[->] (1242) -- (2342);
\end{tikzpicture}
}
\newcommand{\FwPosetBCDBCD}{
\begin{tikzpicture}[
    >=stealth, scale=0.8,
    every node/.style={
        draw,
        rounded corners=2pt,
        inner sep=2pt,
        font=\small
    }
]
\node[fill=blue!20] (1000-1342) at (-9.00,-4.00) {1000-1342};
\node[fill=blue!20] (1100-1242) at (-6.00,-4.00) {1100-1242};
\node[fill=blue!20] (1110-1232) at (-3.00,-4.00) {1110-1232};
\node[fill=blue!20] (1111-1231) at (0.00,-4.00) {1111-1231};
\node[fill=blue!20] (1120-1222) at (3.00,-4.00) {1120-1222};
\node[fill=blue!20] (1121-1221) at (6.00,-4.00) {1121-1221};
\node[fill=blue!20] (1122-1220) at (9.00,-4.00) {1122-1220};
\node[fill=blue!20] (2342) at (0.00,-2.00) {2342};

\draw[->] (1000-1342) -- (2342);
\draw[->] (1100-1242) -- (2342);
\draw[->] (1110-1232) -- (2342);
\draw[->] (1111-1231) -- (2342);
\draw[->] (1120-1222) -- (2342);
\draw[->] (1121-1221) -- (2342);
\draw[->] (1122-1220) -- (2342);
\end{tikzpicture}
}

\newcommand{\FwPosetABC}{
\begin{tikzpicture}[
    >=stealth, scale=1.2,
    every node/.style={
        draw,
        rounded corners=2pt,
        inner sep=2pt,
        font=\small
    }
]
\node[fill=blue!20] (0001-0011-0121-1111-1221-1231) at (-1.50,-4.00) {0001-0011-0121-1111-1221-1231};
\node[fill=blue!20] (0111-1121) at (1.50,-4.00) {0111-1121};
\node[fill=blue!20] (0122-1122-1222-1242-1342-2342) at (-1.50,-2.00) {0122-1122-1222-1242-1342-2342};
\node[fill=blue!20] (1232) at (1.50,-2.00) {1232};

\draw[->] (0001-0011-0121-1111-1221-1231) -- (0122-1122-1222-1242-1342-2342);
\draw[->] (0001-0011-0121-1111-1221-1231) -- (1232);
\draw[->] (0111-1121) -- (0122-1122-1222-1242-1342-2342);
\draw[->] (0111-1121) -- (1232);
\end{tikzpicture}
}
\newcommand{\FwPosetABCA}{
\begin{tikzpicture}[
    >=stealth, scale=1.2,
    every node/.style={
        draw,
        rounded corners=2pt,
        inner sep=2pt,
        font=\small
    }
]
\node[fill=blue!20] (0001-0011-1111-1121) at (-1.50,-4.00) {0001-0011-1111-1121};
\node[fill=blue!20] (0111-0121-1221-1231) at (1.50,-4.00) {0111-0121-1221-1231};
\node[fill=blue!20] (0122-1222-1242-2342) at (-4.50,-2.00) {0122-1222-1242-2342};
\node[fill=blue!20] (1122) at (-1.50,-2.00) {1122};
\node[fill=blue!20] (1232) at (1.50,-2.00) {1232};
\node[fill=blue!20] (1342) at (4.50,-2.00) {1342};

\draw[->] (0001-0011-1111-1121) -- (0122-1222-1242-2342);
\draw[->] (0001-0011-1111-1121) -- (1122);
\draw[->] (0001-0011-1111-1121) -- (1232);
\draw[->] (0111-0121-1221-1231) -- (0122-1222-1242-2342);
\draw[->] (0111-0121-1221-1231) -- (1232);
\draw[->] (0111-0121-1221-1231) -- (1342);
\end{tikzpicture}
}
\newcommand{\FwPosetABCABC}{
\begin{tikzpicture}[
    >=stealth, scale=0.8,
    every node/.style={
        draw,
        rounded corners=2pt,
        inner sep=2pt,
        font=\small
    }
]
\node[fill=blue!20] (0001-1231) at (-4.50,-4.00) {0001-1231};
\node[fill=blue!20] (0011-1221) at (-1.50,-4.00) {0011-1221};
\node[fill=blue!20] (0111-1121) at (1.50,-4.00) {0111-1121};
\node[fill=blue!20] (0121-1111) at (4.50,-4.00) {0121-1111};
\node[fill=blue!20] (0122-2342) at (-4.50,-2.00) {0122-2342};
\node[fill=blue!20] (1122-1342) at (-1.50,-2.00) {1122-1342};
\node[fill=blue!20] (1222-1242) at (1.50,-2.00) {1222-1242};
\node[fill=blue!20] (1232) at (4.50,-2.00) {1232};

\draw[->] (0001-1231) -- (0122-2342);
\draw[->] (0001-1231) -- (1122-1342);
\draw[->] (0001-1231) -- (1222-1242);
\draw[->] (0001-1231) -- (1232);
\draw[->] (0011-1221) -- (0122-2342);
\draw[->] (0011-1221) -- (1122-1342);
\draw[->] (0011-1221) -- (1222-1242);
\draw[->] (0011-1221) -- (1232);
\draw[->] (0111-1121) -- (0122-2342);
\draw[->] (0111-1121) -- (1122-1342);
\draw[->] (0111-1121) -- (1222-1242);
\draw[->] (0111-1121) -- (1232);
\draw[->] (0121-1111) -- (0122-2342);
\draw[->] (0121-1111) -- (1122-1342);
\draw[->] (0121-1111) -- (1222-1242);
\draw[->] (0121-1111) -- (1232);
\end{tikzpicture}
}

\newcommand{\StratumA}{

\begin{tikzpicture}[
    >=stealth, scale=0.8,
    every node/.style={draw,circle,minimum size=8mm,font=\small}
]
\node[fill=blue!20] (0001) at (-4.50,-4.00) {0001};
\node[fill=red!20] (0010) at (-1.50,-4.00) {0010};
\node[fill=green!20] (0100) at (1.50,-4.00) {0100};
\node[fill=red!20] (1000) at (4.50,-4.00) {1000};
\node[fill=blue!20] (0011) at (-3.00,-2.00) {0011};
\node[fill=red!20] (0110) at (0.00,-2.00) {0110};
\node[fill=red!20] (1100) at (3.00,-2.00) {1100};

\draw[->] (0100) -- (0110);
\draw[->] (0100) -- (1100);
\draw[->] (0010) -- (0110);
\draw[->] (0010) -- (0011);
\draw[->] (0001) -- (0011);
\draw[->] (1000) -- (1100);
\end{tikzpicture}
}

\newcommand{\StratumB}{
\begin{tikzpicture}[
    >=stealth, scale=0.8,
    every node/.style={draw,circle,minimum size=8mm,font=\small}
]
\node[fill=red!20] (0001) at (-4.50,-4.00) {0001};
\node[fill=red!20] (0010) at (-1.50,-4.00) {0010};
\node[fill=green!20] (0100) at (1.50,-4.00) {0100};
\node[fill=red!20] (1000) at (4.50,-4.00) {1000};
\node[fill=blue!20] (0011) at (-3.00,-2.00) {0011};
\node[fill=red!20] (0110) at (0.00,-2.00) {0110};
\node[fill=red!20] (1100) at (3.00,-2.00) {1100};

\draw[->] (0100) -- (0110);
\draw[->] (0100) -- (1100);
\draw[->] (0010) -- (0110);
\draw[->] (0010) -- (0011);
\draw[->] (0001) -- (0011);
\draw[->] (1000) -- (1100);
\end{tikzpicture}
}

\newcommand{\StratumG}{
\begin{tikzpicture}[
    >=stealth, scale=0.8,
    every node/.style={draw,circle,minimum size=8mm,font=\small}
]
\node[fill=blue!20] (0001) at (-4.50,-6.00) {0001};
\node[fill=green!20] (0010) at (-1.50,-6.00) {0010};
\node[fill=red!20] (0100) at (1.50,-6.00) {0100};
\node[fill=red!20] (1000) at (4.50,-6.00) {1000};
\node[fill=blue!20] (0011) at (-3.00,-4.00) {0011};
\node[fill=red!20] (0110) at (0.00,-4.00) {0110};
\node[fill=blue!20] (1100) at (3.00,-4.00) {1100};
\node[fill=blue!20] (0111) at (-3.00,-2.00) {0111};
\node[fill=red!20] (0120) at (0.00,-2.00) {0120};
\node[fill=blue!20] (1110) at (3.00,-2.00) {1110};

\draw[->] (0110) -- (0120);
\draw[->] (0110) -- (0111);
\draw[->] (0110) -- (1110);
\draw[->] (0100) -- (0110);
\draw[->] (0100) -- (1100);
\draw[->] (0011) -- (0111);
\draw[->] (1100) -- (1110);
\draw[->] (0010) -- (0110);
\draw[->] (0010) -- (0011);
\draw[->] (0001) -- (0011);
\draw[->] (1000) -- (1100);
\end{tikzpicture}
}

\newcommand{\StratumL}{
\begin{tikzpicture}[
    >=stealth, scale=0.8,
    every node/.style={draw,circle,minimum size=8mm,font=\small}
]
\node[fill=blue!20] (0001) at (-4.50,-6.00) {0001};
\node[fill=green!20] (0010) at (-1.50,-6.00) {0010};
\node[fill=red!20] (0100) at (1.50,-6.00) {0100};
\node[fill=blue!20] (1000) at (4.50,-6.00) {1000};
\node[fill=blue!20] (0011) at (-3.00,-4.00) {0011};
\node[fill=red!20] (0110) at (0.00,-4.00) {0110};
\node[fill=blue!20] (1100) at (3.00,-4.00) {1100};
\node[fill=blue!20] (0111) at (-3.00,-2.00) {0111};
\node[fill=red!20] (0120) at (0.00,-2.00) {0120};
\node[fill=blue!20] (1110) at (3.00,-2.00) {1110};

\draw[->] (0100) -- (1100);
\draw[->] (0100) -- (0110);
\draw[->] (0001) -- (0011);
\draw[->] (0110) -- (0111);
\draw[->] (0110) -- (1110);
\draw[->] (0110) -- (0120);
\draw[->] (0011) -- (0111);
\draw[->] (1000) -- (1100);
\draw[->] (0010) -- (0110);
\draw[->] (0010) -- (0011);
\draw[->] (1100) -- (1110);
\end{tikzpicture}
}

\newcommand{\StratumP}{
\begin{tikzpicture}[
    >=stealth, scale=0.8,
    every node/.style={draw,circle,minimum size=8mm,font=\small}
]
\node[fill=blue!20] (0001) at (-4.50,-8.00) {0001};
\node[fill=green!20] (0010) at (-1.50,-8.00) {0010};
\node[fill=red!20] (0100) at (1.50,-8.00) {0100};
\node[fill=green!20] (1000) at (4.50,-8.00) {1000};
\node[fill=blue!20] (0011) at (-3.00,-6.00) {0011};
\node[fill=red!20] (0110) at (0.00,-6.00) {0110};
\node[fill=red!20] (1100) at (3.00,-6.00) {1100};
\node[fill=blue!20] (0111) at (-3.00,-4.00) {0111};
\node[fill=red!20] (0120) at (0.00,-4.00) {0120};
\node[fill=red!20] (1110) at (3.00,-4.00) {1110};
\node[fill=blue!20] (0121) at (-3.00,-2.00) {0121};
\node[fill=blue!20] (1111) at (0.00,-2.00) {1111};
\node[fill=red!20] (1120) at (3.00,-2.00) {1120};

\draw[->] (0100) -- (1100);
\draw[->] (0100) -- (0110);
\draw[->] (1110) -- (1111);
\draw[->] (1110) -- (1120);
\draw[->] (0001) -- (0011);
\draw[->] (0110) -- (0111);
\draw[->] (0110) -- (1110);
\draw[->] (0110) -- (0120);
\draw[->] (0011) -- (0111);
\draw[->] (0111) -- (1111);
\draw[->] (0111) -- (0121);
\draw[->] (1000) -- (1100);
\draw[->] (0010) -- (0110);
\draw[->] (0010) -- (0011);
\draw[->] (1100) -- (1110);
\draw[->] (0120) -- (0121);
\draw[->] (0120) -- (1120);
\end{tikzpicture}
}

\newcommand{\StratumN}{
\begin{tikzpicture}[
    >=stealth, scale=0.8,
    every node/.style={draw,circle,minimum size=8mm,font=\small}
]
\node[fill=red!20] (0001) at (-4.50,-8.00) {0001};
\node[fill=green!20] (0010) at (-1.50,-8.00) {0010};
\node[fill=blue!20] (0100) at (1.50,-8.00) {0100};
\node[fill=blue!20] (1000) at (4.50,-8.00) {1000};
\node[fill=red!20] (0011) at (-3.00,-6.00) {0011};
\node[fill=blue!20] (0110) at (0.00,-6.00) {0110};
\node[fill=red!20] (1100) at (3.00,-6.00) {1100};
\node[fill=blue!20] (0111) at (-3.00,-4.00) {0111};
\node[fill=blue!20] (0120) at (0.00,-4.00) {0120};
\node[fill=red!20] (1110) at (3.00,-4.00) {1110};
\node[fill=blue!20] (0121) at (-3.00,-2.00) {0121};
\node[fill=blue!20] (1111) at (0.00,-2.00) {1111};
\node[fill=red!20] (1120) at (3.00,-2.00) {1120};

\draw[->] (0100) -- (1100);
\draw[->] (0100) -- (0110);
\draw[->] (1110) -- (1111);
\draw[->] (1110) -- (1120);
\draw[->] (0001) -- (0011);
\draw[->] (0110) -- (0111);
\draw[->] (0110) -- (1110);
\draw[->] (0110) -- (0120);
\draw[->] (0011) -- (0111);
\draw[->] (0111) -- (1111);
\draw[->] (0111) -- (0121);
\draw[->] (1000) -- (1100);
\draw[->] (0010) -- (0110);
\draw[->] (0010) -- (0011);
\draw[->] (1100) -- (1110);
\draw[->] (0120) -- (0121);
\draw[->] (0120) -- (1120);
\end{tikzpicture}
}

\newcommand{\StratumQ}{
\begin{tikzpicture}[
    >=stealth, scale=0.8,
    every node/.style={draw,circle,minimum size=8mm,font=\small}
]
\node[fill=red!20] (0001) at (-4.50,-6.00) {0001};
\node[fill=green!20] (0010) at (-1.50,-6.00) {0010};
\node[fill=red!20] (0100) at (1.50,-6.00) {0100};
\node[fill=blue!20] (1000) at (4.50,-6.00) {1000};
\node[fill=red!20] (0011) at (-3.00,-4.00) {0011};
\node[fill=red!20] (0110) at (0.00,-4.00) {0110};
\node[fill=blue!20] (1100) at (3.00,-4.00) {1100};
\node[fill=blue!20] (0111) at (-3.00,-2.00) {0111};
\node[fill=red!20] (0120) at (0.00,-2.00) {0120};
\node[fill=blue!20] (1110) at (3.00,-2.00) {1110};

\draw[->] (0100) -- (1100);
\draw[->] (0100) -- (0110);
\draw[->] (0001) -- (0011);
\draw[->] (0110) -- (0111);
\draw[->] (0110) -- (1110);
\draw[->] (0110) -- (0120);
\draw[->] (0011) -- (0111);
\draw[->] (1000) -- (1100);
\draw[->] (0010) -- (0110);
\draw[->] (0010) -- (0011);
\draw[->] (1100) -- (1110);
\end{tikzpicture}
}

\newcommand{\StratumK}{
    \begin{tikzpicture}[
    >=stealth, scale=0.8,
    every node/.style={draw,circle,minimum size=8mm,font=\small}
]
\node[fill=red!20] (0001) at (-4.50,-6.00) {0001};
\node[fill=green!20] (0010) at (-1.50,-6.00) {0010};
\node[fill=red!20] (0100) at (1.50,-6.00) {0100};
\node[fill=blue!20] (1000) at (4.50,-6.00) {1000};
\node[fill=red!20] (0011) at (-3.00,-4.00) {0011};
\node[fill=red!20] (0110) at (0.00,-4.00) {0110};
\node[fill=blue!20] (1100) at (3.00,-4.00) {1100};
\node[fill=blue!20] (0111) at (-3.00,-2.00) {0111};
\node[fill=red!20] (0120) at (0.00,-2.00) {0120};
\node[fill=blue!20] (1110) at (3.00,-2.00) {1110};

\draw[->] (0001) -- (0011);
\draw[->] (0110) -- (0111);
\draw[->] (0110) -- (1110);
\draw[->] (0110) -- (0120);
\draw[->] (1000) -- (1100);
\draw[->] (0010) -- (0011);
\draw[->] (0010) -- (0110);
\draw[->] (0100) -- (1100);
\draw[->] (0100) -- (0110);
\draw[->] (1100) -- (1110);
\draw[->] (0011) -- (0111);
\end{tikzpicture}
}

\newcommand{\StratumKGrob}{
\left\{ \begin{aligned} &\mu^2 - 2\mu + \frac{-2x_{0001}s_{0100}s_{0120} - x_{0001}s_{0110}^2 + \cdots + x_{0110}s_{0001}s_{0110} + x_{0110}s_{0011}s_{0100}}{2x_{0001}s_{0100}s_{0120} - x_{0001}s_{0110}^2 - \cdots + x_{0110}s_{0001}s_{0110} - x_{0110}s_{0011}s_{0100}}, \\ &b_{0010r} + \frac{-4x_{0001}s_{0001}s_{0100}s_{0120}^2 + 2x_{0001}s_{0001}s_{0110}^2s_{0120} + \cdots + x_{0110}s_{0001}s_{0011}s_{0110}^2 - x_{0110}s_{0011}^2s_{0100}s_{0110}}{8x_{0001}^2s_{0100}s_{0110}s_{0120} - 8x_{0001}x_{0100}s_{0001}s_{0110}s_{0120} + \cdots - 4x_{0100}x_{0110}s_{0001}s_{0011}s_{0110} + 4x_{0110}^2s_{0001}s_{0011}s_{0100}}\mu \\ &\quad + \frac{4x_{0001}s_{0001}s_{0100}s_{0120}^2 - 2x_{0001}s_{0001}s_{0110}^2s_{0120} - \cdots - x_{0110}s_{0001}s_{0011}s_{0110}^2 + x_{0110}s_{0011}^2s_{0100}s_{0110}}{8x_{0001}^2s_{0100}s_{0110}s_{0120} - 8x_{0001}x_{0100}s_{0001}s_{0110}s_{0120} + \cdots - 4x_{0100}x_{0110}s_{0001}s_{0011}s_{0110} + 4x_{0110}^2s_{0001}s_{0011}s_{0100}}, \\ &b_{0010u} + \frac{4x_{0001}^2s_{0100}s_{0120}^2 - 2x_{0001}^2s_{0110}^2s_{0120} - \cdots - x_{0110}^2s_{0001}s_{0011}s_{0110} + x_{0110}^2s_{0011}^2s_{0100}}{8x_{0001}^2s_{0100}s_{0110}s_{0120} - 8x_{0001}x_{0100}s_{0001}s_{0110}s_{0120} + \cdots - 4x_{0100}x_{0110}s_{0001}s_{0011}s_{0110} + 4x_{0110}^2s_{0001}s_{0011}s_{0100}}\mu 
\\ &\quad + \frac{-4x_{0001}^2s_{0100}s_{0120}^2 + 2x_{0001}^2s_{0110}^2s_{0120} + \cdots + x_{0110}^2s_{0001}s_{0011}s_{0110} - x_{0110}^2s_{0011}^2s_{0100}}{8x_{0001}^2s_{0100}s_{0110}s_{0120} - 8x_{0001}x_{0100}s_{0001}s_{0110}s_{0120} + \cdots - 4x_{0100}x_{0110}s_{0001}s_{0011}s_{0110} + 4x_{0110}^2s_{0001}s_{0011}s_{0100}}, \\ &\lambda_{0010} - 1 \end{aligned} \right\}
}

\newcommand{\StratumKEqns}{
\left\{ \begin{aligned} 
(-2x_{0001})b_{0010r}\mu + (-2s_{0001})b_{0010u}\mu + (2x_{0001})b_{0010r} + (2s_{0001})b_{0010u} + (-s_{0011})
\\
(2x_{0100})b_{0010r}\mu^2 + (2s_{0100})b_{0010u}\mu^2 + (-4x_{0100})b_{0010r}\mu + (-4s_{0100})b_{0010u}\mu + (-2x_{0100})b_{0010r} + (-2s_{0100})b_{0010u} + s_{0110}\mu + (-s_{0110})
\\
(-2x_{0110})b_{0010r}\mu + (-2s_{0110})b_{0010u}\mu + (2x_{0110})b_{0010r} + (2s_{0110})b_{0010u} + (-2s_{0120})
\\
\lambda_{0010}-1
\end{aligned}
\right\}
}

\newcommand{\StratumR}{
    \begin{tikzpicture}[
    >=stealth, scale=0.8,
    every node/.style={draw,circle,minimum size=8mm,font=\small}
]
\node[fill=green!20] (0001) at (-4.50,-12.00) {0001};
\node[fill=green!20] (0010) at (-1.50,-12.00) {0010};
\node[fill=red!20] (0100) at (1.50,-12.00) {0100};
\node[fill=green!20] (1000) at (4.50,-12.00) {1000};
\node[fill=green!20] (0011) at (-3.00,-10.00) {0011};
\node[fill=red!20] (0110) at (0.00,-10.00) {0110};
\node[fill=red!20] (1100) at (3.00,-10.00) {1100};
\node[fill=red!20] (0111) at (-3.00,-8.00) {0111};
\node[fill=red!20] (0120) at (0.00,-8.00) {0120};
\node[fill=red!20] (1110) at (3.00,-8.00) {1110};
\node[fill=red!20] (0121) at (-3.00,-6.00) {0121};
\node[fill=red!20] (1111) at (0.00,-6.00) {1111};
\node[fill=red!20] (1120) at (3.00,-6.00) {1120};
\node[fill=red!20] (0122) at (-3.00,-4.00) {0122};
\node[fill=red!20] (1121) at (0.00,-4.00) {1121};
\node[fill=blue!20] (1220) at (3.00,-4.00) {1220};
\node[fill=red!20] (1122) at (-1.50,-2.00) {1122};
\node[fill=blue!20] (1221) at (1.50,-2.00) {1221};

\draw[->] (1100) -- (1110);
\draw[->] (1111) -- (1121);
\draw[->] (0121) -- (1121);
\draw[->] (0121) -- (0122);
\draw[->] (0111) -- (0121);
\draw[->] (0111) -- (1111);
\draw[->] (0100) -- (0110);
\draw[->] (0100) -- (1100);
\draw[->] (1220) -- (1221);
\draw[->] (1110) -- (1120);
\draw[->] (1110) -- (1111);
\draw[->] (0122) -- (1122);
\draw[->] (0011) -- (0111);
\draw[->] (0010) -- (0110);
\draw[->] (0010) -- (0011);
\draw[->] (0110) -- (0120);
\draw[->] (0110) -- (1110);
\draw[->] (0110) -- (0111);
\draw[->] (1000) -- (1100);
\draw[->] (1121) -- (1221);
\draw[->] (1121) -- (1122);
\draw[->] (1120) -- (1220);
\draw[->] (1120) -- (1121);
\draw[->] (0120) -- (1120);
\draw[->] (0120) -- (0121);
\draw[->] (0001) -- (0011);
\end{tikzpicture}

}

\newcommand{\FwPosetA}{
\begin{tikzpicture}[
    >=stealth, scale=0.65,
    every node/.style={
        draw,
        rounded corners=2pt,
        inner sep=2pt,
        font=\small
    }
]
\node[fill=blue!20] (0001) at (-3.00,-18.00) {0001};
\node[fill=blue!20] (0010) at (0.00,-18.00) {0010};
\node[fill=gray] (0100-1100) at (3.00,-18.00) {0100-1100};
\node[fill=blue!20] (0011) at (-1.50,-16.00) {0011};
\node[fill=gray] (0110-1110) at (1.50,-16.00) {0110-1110};
\node[fill=gray] (0111-1111) at (-1.50,-14.00) {0111-1111};
\node[fill=gray] (0120-1120) at (1.50,-14.00) {0120-1120};
\node[fill=gray] (0121-1121) at (-1.50,-12.00) {0121-1121};
\node[fill=blue!20] (1220) at (1.50,-12.00) {1220};
\node[fill=gray] (0122-1122) at (-1.50,-10.00) {0122-1122};
\node[fill=blue!20] (1221) at (1.50,-10.00) {1221};
\node[fill=blue!20] (1222) at (-1.50,-8.00) {1222};
\node[fill=blue!20] (1231) at (1.50,-8.00) {1231};
\node[fill=blue!20] (1232) at (0.00,-6.00) {1232};
\node[fill=blue!20] (1242) at (0.00,-4.00) {1242};
\node[fill=gray] (1342-2342) at (0.00,-2.00) {1342-2342};

\draw[->] (0001) -- (0011);
\draw[->] (0010) -- (0011);
\draw[->] (0010) -- (0110-1110);
\draw[->] (0011) -- (0111-1111);
\draw[->] (0100-1100) -- (0110-1110);
\draw[->] (0110-1110) -- (0111-1111);
\draw[->] (0110-1110) -- (0120-1120);
\draw[->] (0111-1111) -- (0121-1121);
\draw[->] (0120-1120) -- (0121-1121);
\draw[->] (0120-1120) -- (1220);
\draw[->] (0121-1121) -- (0122-1122);
\draw[->] (0121-1121) -- (1221);
\draw[->] (0122-1122) -- (1222);
\draw[->] (1220) -- (1221);
\draw[->] (1221) -- (1222);
\draw[->] (1221) -- (1231);
\draw[->] (1222) -- (1232);
\draw[->] (1231) -- (1232);
\draw[->] (1232) -- (1242);
\draw[->] (1242) -- (1342-2342);
\end{tikzpicture}
}
\newcommand{\FwPosetB}{
\begin{tikzpicture}[
    >=stealth, scale=0.65,
    every node/.style={
        draw,
        rounded corners=2pt,
        inner sep=2pt,
        font=\small
    }
]
\node[fill=blue!20] (0001) at (-3.00,-16.00) {0001};
\node[fill=gray] (0010-0110) at (0.00,-16.00) {0010-0110};
\node[fill=gray] (1000-1100) at (3.00,-16.00) {1000-1100};
\node[fill=gray] (0011-0111) at (-3.00,-14.00) {0011-0111};
\node[fill=blue!20] (0120) at (0.00,-14.00) {0120};
\node[fill=blue!20] (1110) at (3.00,-14.00) {1110};
\node[fill=blue!20] (0121) at (-3.00,-12.00) {0121};
\node[fill=blue!20] (1111) at (0.00,-12.00) {1111};
\node[fill=gray] (1120-1220) at (3.00,-12.00) {1120-1220};
\node[fill=blue!20] (0122) at (-1.50,-10.00) {0122};
\node[fill=gray] (1121-1221) at (1.50,-10.00) {1121-1221};
\node[fill=gray] (1122-1222) at (-1.50,-8.00) {1122-1222};
\node[fill=blue!20] (1231) at (1.50,-8.00) {1231};
\node[fill=blue!20] (1232) at (0.00,-6.00) {1232};
\node[fill=gray] (1242-1342) at (0.00,-4.00) {1242-1342};
\node[fill=blue!20] (2342) at (0.00,-2.00) {2342};

\draw[->] (0001) -- (0011-0111);
\draw[->] (0010-0110) -- (0011-0111);
\draw[->] (0010-0110) -- (0120);
\draw[->] (0010-0110) -- (1110);
\draw[->] (0011-0111) -- (0121);
\draw[->] (0011-0111) -- (1111);
\draw[->] (0120) -- (0121);
\draw[->] (0120) -- (1120-1220);
\draw[->] (0121) -- (0122);
\draw[->] (0121) -- (1121-1221);
\draw[->] (0122) -- (1122-1222);
\draw[->] (1000-1100) -- (1110);
\draw[->] (1110) -- (1111);
\draw[->] (1110) -- (1120-1220);
\draw[->] (1111) -- (1121-1221);
\draw[->] (1120-1220) -- (1121-1221);
\draw[->] (1121-1221) -- (1122-1222);
\draw[->] (1121-1221) -- (1231);
\draw[->] (1122-1222) -- (1232);
\draw[->] (1231) -- (1232);
\draw[->] (1232) -- (1242-1342);
\draw[->] (1242-1342) -- (2342);
\end{tikzpicture}
}
\newcommand{\FwPosetC}{
\begin{tikzpicture}[
    >=stealth, scale=0.65,
    every node/.style={
        draw,
        rounded corners=2pt,
        inner sep=2pt,
        font=\small
    }
]
\node[fill=gray] (0001-0011) at (-4.50,-14.00) {0001-0011};
\node[fill=blue!20] (0100-0120) at (-1.50,-14.00) {0100-0120};
\node[fill=blue!20] (0110) at (1.50,-14.00) {0110};
\node[fill=blue!20] (1000) at (3.00,-14.00) {1000};
\node[fill=gray] (0111-0121) at (-3.00,-12.00) {0111-0121};
\node[fill=blue!20] (1100-1120) at (0.00,-12.00) {1100-1120};
\node[fill=blue!20] (1110) at (3.00,-12.00) {1110};
\node[fill=blue!20] (0122) at (-3.00,-10.00) {0122};
\node[fill=gray] (1111-1121) at (0.00,-10.00) {1111-1121};
\node[fill=blue!20] (1220) at (3.00,-10.00) {1220};
\node[fill=blue!20] (1122) at (-1.50,-8.00) {1122};
\node[fill=gray] (1221-1231) at (1.50,-8.00) {1221-1231};
\node[fill=blue!20] (1222-1242) at (-1.50,-6.00) {1222-1242};
\node[fill=blue!20] (1232) at (1.50,-6.00) {1232};
\node[fill=blue!20] (1342) at (0.00,-4.00) {1342};
\node[fill=blue!20] (2342) at (0.00,-2.00) {2342};

\draw[->] (0001-0011) -- (0111-0121);
\draw[->] (0001-0011) -- (0111-0121);
\draw[->] (0100-0120) -- (0111-0121);
\draw[->] (0100-0120) -- (1100-1120);
\draw[->] (0110) -- (0111-0121);
\draw[->] (0110) -- (1110);
\draw[->] (0111-0121) -- (0122);
\draw[->] (0111-0121) -- (1111-1121);
\draw[->] (0122) -- (1122);
\draw[->] (1000) -- (1100-1120);
\draw[->] (1000) -- (1110);
\draw[->] (1100-1120) -- (1111-1121);
\draw[->] (1100-1120) -- (1220);
\draw[->] (1110) -- (1111-1121);
\draw[->] (1110) -- (1220);
\draw[->] (1111-1121) -- (1122);
\draw[->] (1111-1121) -- (1221-1231);
\draw[->] (1111-1121) -- (1221-1231);
\draw[->] (1122) -- (1222-1242);
\draw[->] (1122) -- (1232);
\draw[->] (1220) -- (1221-1231);
\draw[->] (1221-1231) -- (1222-1242);
\draw[->] (1221-1231) -- (1232);
\draw[->] (1222-1242) -- (1342);
\draw[->] (1232) -- (1342);
\draw[->] (1342) -- (2342);
\end{tikzpicture}
}
\newcommand{\FwPosetD}{
\begin{tikzpicture}[
    >=stealth, scale=0.65,
    every node/.style={
        draw,
        rounded corners=2pt,
        inner sep=2pt,
        font=\small
    }
]
\node[fill=gray] (0010-0011) at (-3.00,-18.00) {0010-0011};
\node[fill=blue!20] (0100) at (0.00,-18.00) {0100};
\node[fill=blue!20] (1000) at (3.00,-18.00) {1000};
\node[fill=gray] (0110-0111) at (-1.50,-16.00) {0110-0111};
\node[fill=blue!20] (1100) at (1.50,-16.00) {1100};
\node[fill=blue!20] (0120-0122) at (-3.00,-14.00) {0120-0122};
\node[fill=blue!20] (0121) at (0.00,-14.00) {0121};
\node[fill=gray] (1110-1111) at (3.00,-14.00) {1110-1111};
\node[fill=blue!20] (1120-1122) at (-1.50,-12.00) {1120-1122};
\node[fill=blue!20] (1121) at (1.50,-12.00) {1121};
\node[fill=blue!20] (1220-1222) at (-1.50,-10.00) {1220-1222};
\node[fill=blue!20] (1221) at (1.50,-10.00) {1221};
\node[fill=gray] (1231-1232) at (0.00,-8.00) {1231-1232};
\node[fill=blue!20] (1242) at (0.00,-6.00) {1242};
\node[fill=blue!20] (1342) at (0.00,-4.00) {1342};
\node[fill=blue!20] (2342) at (0.00,-2.00) {2342};

\draw[->] (0010-0011) -- (0110-0111);
\draw[->] (0100) -- (0110-0111);
\draw[->] (0100) -- (1100);
\draw[->] (0110-0111) -- (0120-0122);
\draw[->] (0110-0111) -- (0121);
\draw[->] (0110-0111) -- (1110-1111);
\draw[->] (0120-0122) -- (1120-1122);
\draw[->] (0121) -- (1121);
\draw[->] (1000) -- (1100);
\draw[->] (1100) -- (1110-1111);
\draw[->] (1110-1111) -- (1120-1122);
\draw[->] (1110-1111) -- (1121);
\draw[->] (1120-1122) -- (1220-1222);
\draw[->] (1121) -- (1221);
\draw[->] (1220-1222) -- (1231-1232);
\draw[->] (1221) -- (1231-1232);
\draw[->] (1231-1232) -- (1242);
\draw[->] (1242) -- (1342);
\draw[->] (1342) -- (2342);
\end{tikzpicture}
}
\newcommand{\FwPosetAB}{
\begin{tikzpicture}[
    >=stealth, scale=0.65,
    every node/.style={
        draw,
        rounded corners=2pt,
        inner sep=2pt,
        font=\small
    }
]
\node[fill=blue!20] (0001) at (-1.50,-12.00) {0001};
\node[fill=gray] (0010-1110-0110) at (1.50,-12.00) {0010-1110-0110};
\node[fill=gray] (0011-1111-0111) at (-1.50,-10.00) {0011-1111-0111};
\node[fill=gray] (0120-1120-1220) at (1.50,-10.00) {0120-1120-1220};
\node[fill=gray] (0121-1121-1221) at (0.00,-8.00) {0121-1121-1221};
\node[fill=gray] (0122-1122-1222) at (-1.50,-6.00) {0122-1122-1222};
\node[fill=blue!20] (1231) at (1.50,-6.00) {1231};
\node[fill=blue!20] (1232) at (0.00,-4.00) {1232};
\node[fill=gray] (1242-2342-1342) at (0.00,-2.00) {1242-2342-1342};

\draw[->] (0001) -- (0011-1111-0111);
\draw[->] (0010-1110-0110) -- (0011-1111-0111);
\draw[->] (0010-1110-0110) -- (0120-1120-1220);
\draw[->] (0011-1111-0111) -- (0121-1121-1221);
\draw[->] (0120-1120-1220) -- (0121-1121-1221);
\draw[->] (0121-1121-1221) -- (0122-1122-1222);
\draw[->] (0121-1121-1221) -- (1231);
\draw[->] (0122-1122-1222) -- (1232);
\draw[->] (1231) -- (1232);
\draw[->] (1232) -- (1242-2342-1342);
\end{tikzpicture}
}
\newcommand{\FwPosetAC}{
\begin{tikzpicture}[
    >=stealth, scale=0.65,
    every node/.style={
        draw,
        rounded corners=2pt,
        inner sep=2pt,
        font=\small
    }
]
\node[fill=gray] (0001-0011) at (-4.50,-10.00) {0001-0011};
\node[fill=blue!20] (0100-1120) at (-1.50,-10.00) {0100-1120};
\node[fill=gray] (0110-1110) at (1.50,-10.00) {0110-1110};
\node[fill=blue!20] (0120-1100) at (4.50,-10.00) {0120-1100};
\node[fill=blue!20] (0111-1121) at (-3.00,-8.00) {0111-1121};
\node[fill=blue!20] (0121-1111) at (0.00,-8.00) {0121-1111};
\node[fill=blue!20] (1220) at (3.00,-8.00) {1220};
\node[fill=gray] (0122-1122) at (-1.50,-6.00) {0122-1122};
\node[fill=gray] (1221-1231) at (1.50,-6.00) {1221-1231};
\node[fill=blue!20] (1222-1242) at (-1.50,-4.00) {1222-1242};
\node[fill=blue!20] (1232) at (1.50,-4.00) {1232};
\node[fill=gray] (1342-2342) at (0.00,-2.00) {1342-2342};

\draw[->] (0001-0011) -- (0111-1121);
\draw[->] (0001-0011) -- (0111-1121);
\draw[->] (0001-0011) -- (0121-1111);
\draw[->] (0001-0011) -- (0121-1111);
\draw[->] (0100-1120) -- (0111-1121);
\draw[->] (0100-1120) -- (1220);
\draw[->] (0110-1110) -- (0111-1121);
\draw[->] (0110-1110) -- (0121-1111);
\draw[->] (0110-1110) -- (1220);
\draw[->] (0111-1121) -- (0122-1122);
\draw[->] (0111-1121) -- (1221-1231);
\draw[->] (0111-1121) -- (1221-1231);
\draw[->] (0120-1100) -- (0121-1111);
\draw[->] (0120-1100) -- (1220);
\draw[->] (0121-1111) -- (0122-1122);
\draw[->] (0121-1111) -- (1221-1231);
\draw[->] (0121-1111) -- (1221-1231);
\draw[->] (0122-1122) -- (1222-1242);
\draw[->] (0122-1122) -- (1232);
\draw[->] (0122-1122) -- (1222-1242);
\draw[->] (1220) -- (1221-1231);
\draw[->] (1221-1231) -- (1222-1242);
\draw[->] (1221-1231) -- (1232);
\draw[->] (1222-1242) -- (1342-2342);
\draw[->] (1222-1242) -- (1342-2342);
\draw[->] (1232) -- (1342-2342);
\end{tikzpicture}
}
\newcommand{\FwPosetAD}{
\begin{tikzpicture}[
    >=stealth, scale=0.65,
    every node/.style={
        draw,
        rounded corners=2pt,
        inner sep=2pt,
        font=\small
    }
]
\node[fill=gray] (0010-0011) at (-1.50,-14.00) {0010-0011};
\node[fill=gray] (0100-1100) at (1.50,-14.00) {0100-1100};
\node[fill=blue!20] (0110-1111) at (-1.50,-12.00) {0110-1111};
\node[fill=blue!20] (0111-1110) at (1.50,-12.00) {0111-1110};
\node[fill=blue!20] (0120-1122) at (-3.00,-10.00) {0120-1122};
\node[fill=gray] (0121-1121) at (0.00,-10.00) {0121-1121};
\node[fill=blue!20] (0122-1120) at (3.00,-10.00) {0122-1120};
\node[fill=blue!20] (1220-1222) at (-1.50,-8.00) {1220-1222};
\node[fill=blue!20] (1221) at (1.50,-8.00) {1221};
\node[fill=gray] (1231-1232) at (0.00,-6.00) {1231-1232};
\node[fill=blue!20] (1242) at (0.00,-4.00) {1242};
\node[fill=gray] (1342-2342) at (0.00,-2.00) {1342-2342};

\draw[->] (0010-0011) -- (0110-1111);
\draw[->] (0010-0011) -- (0111-1110);
\draw[->] (0100-1100) -- (0110-1111);
\draw[->] (0100-1100) -- (0111-1110);
\draw[->] (0110-1111) -- (0120-1122);
\draw[->] (0110-1111) -- (0121-1121);
\draw[->] (0111-1110) -- (0121-1121);
\draw[->] (0111-1110) -- (0122-1120);
\draw[->] (0120-1122) -- (1220-1222);
\draw[->] (0121-1121) -- (1221);
\draw[->] (0122-1120) -- (1220-1222);
\draw[->] (1220-1222) -- (1231-1232);
\draw[->] (1221) -- (1231-1232);
\draw[->] (1231-1232) -- (1242);
\draw[->] (1242) -- (1342-2342);
\end{tikzpicture}
}
\newcommand{\FwPosetBC}{
\begin{tikzpicture}[
    >=stealth, scale=0.65,
    every node/.style={
        draw,
        rounded corners=2pt,
        inner sep=2pt,
        font=\small
    }
]
\node[fill=gray] (0001-0111-0121-0011) at (-3.00,-8.00) {0001-0111-0121-0011};
\node[fill=gray] (1000-1100-1220-1120) at (0.00,-8.00) {1000-1100-1220-1120};
\node[fill=blue!20] (1110) at (3.00,-8.00) {1110};
\node[fill=blue!20] (0122) at (-1.50,-6.00) {0122};
\node[fill=gray] (1111-1221-1231-1121) at (1.50,-6.00) {1111-1221-1231-1121};
\node[fill=gray] (1122-1222-1342-1242) at (-1.50,-4.00) {1122-1222-1342-1242};
\node[fill=blue!20] (1232) at (1.50,-4.00) {1232};
\node[fill=blue!20] (2342) at (0.00,-2.00) {2342};

\draw[->] (0001-0111-0121-0011) -- (0122);
\draw[->] (0001-0111-0121-0011) -- (1111-1221-1231-1121);
\draw[->] (0001-0111-0121-0011) -- (1111-1221-1231-1121);
\draw[->] (0122) -- (1122-1222-1342-1242);
\draw[->] (0122) -- (1232);
\draw[->] (1000-1100-1220-1120) -- (1111-1221-1231-1121);
\draw[->] (1110) -- (1111-1221-1231-1121);
\draw[->] (1111-1221-1231-1121) -- (1122-1222-1342-1242);
\draw[->] (1111-1221-1231-1121) -- (1232);
\draw[->] (1122-1222-1342-1242) -- (2342);
\draw[->] (1232) -- (2342);
\end{tikzpicture}
}
\newcommand{\FwPosetCBCB}{
\begin{tikzpicture}[
    >=stealth, scale=0.65,
    every node/.style={
        draw,
        rounded corners=2pt,
        inner sep=2pt,
        font=\small
    }
]
\node[fill=gray] (0001-0121) at (-6.00,-8.00) {0001-0121};
\node[fill=gray] (0011-0111) at (-3.00,-8.00) {0011-0111};
\node[fill=blue!20] (1000-1220) at (0.00,-8.00) {1000-1220};
\node[fill=blue!20] (1100-1120) at (3.00,-8.00) {1100-1120};
\node[fill=blue!20] (1110) at (6.00,-8.00) {1110};
\node[fill=blue!20] (0122) at (-3.00,-6.00) {0122};
\node[fill=gray] (1111-1231) at (0.00,-6.00) {1111-1231};
\node[fill=gray] (1121-1221) at (3.00,-6.00) {1121-1221};
\node[fill=blue!20] (1122-1342) at (-3.00,-4.00) {1122-1342};
\node[fill=blue!20] (1222-1242) at (0.00,-4.00) {1222-1242};
\node[fill=blue!20] (1232) at (3.00,-4.00) {1232};
\node[fill=blue!20] (2342) at (0.00,-2.00) {2342};

\draw[->] (0001-0121) -- (0122);
\draw[->] (0001-0121) -- (1121-1221);
\draw[->] (0001-0121) -- (1121-1221);
\draw[->] (0001-0121) -- (1111-1231);
\draw[->] (0011-0111) -- (0122);
\draw[->] (0011-0111) -- (1111-1231);
\draw[->] (0011-0111) -- (1111-1231);
\draw[->] (0011-0111) -- (1121-1221);
\draw[->] (0122) -- (1122-1342);
\draw[->] (0122) -- (1222-1242);
\draw[->] (0122) -- (1232);
\draw[->] (1000-1220) -- (1121-1221);
\draw[->] (1000-1220) -- (1111-1231);
\draw[->] (1100-1120) -- (1121-1221);
\draw[->] (1100-1120) -- (1111-1231);
\draw[->] (1110) -- (1111-1231);
\draw[->] (1110) -- (1121-1221);
\draw[->] (1111-1231) -- (1232);
\draw[->] (1111-1231) -- (1122-1342);
\draw[->] (1111-1231) -- (1222-1242);
\draw[->] (1121-1221) -- (1122-1342);
\draw[->] (1121-1221) -- (1222-1242);
\draw[->] (1121-1221) -- (1232);
\draw[->] (1122-1342) -- (2342);
\draw[->] (1222-1242) -- (2342);
\draw[->] (1232) -- (2342);
\end{tikzpicture}
}
\newcommand{\FwPosetBD}{
\begin{tikzpicture}[
    >=stealth, scale=0.65,
    every node/.style={
        draw,
        rounded corners=2pt,
        inner sep=2pt,
        font=\small
    }
]
\node[fill=blue!20] (0010-0111) at (-3.00,-12.00) {0010-0111};
\node[fill=blue!20] (0011-0110) at (0.00,-12.00) {0011-0110};
\node[fill=gray] (1000-1100) at (3.00,-12.00) {1000-1100};
\node[fill=blue!20] (0120-0122) at (-3.00,-10.00) {0120-0122};
\node[fill=blue!20] (0121) at (0.00,-10.00) {0121};
\node[fill=gray] (1110-1111) at (3.00,-10.00) {1110-1111};
\node[fill=blue!20] (1120-1222) at (-3.00,-8.00) {1120-1222};
\node[fill=gray] (1121-1221) at (0.00,-8.00) {1121-1221};
\node[fill=blue!20] (1122-1220) at (3.00,-8.00) {1122-1220};
\node[fill=gray] (1231-1232) at (0.00,-6.00) {1231-1232};
\node[fill=gray] (1242-1342) at (0.00,-4.00) {1242-1342};
\node[fill=blue!20] (2342) at (0.00,-2.00) {2342};

\draw[->] (0010-0111) -- (0121);
\draw[->] (0010-0111) -- (0120-0122);
\draw[->] (0010-0111) -- (1110-1111);
\draw[->] (0011-0110) -- (0120-0122);
\draw[->] (0011-0110) -- (0121);
\draw[->] (0011-0110) -- (1110-1111);
\draw[->] (0120-0122) -- (1122-1220);
\draw[->] (0120-0122) -- (1120-1222);
\draw[->] (0121) -- (1121-1221);
\draw[->] (1000-1100) -- (1110-1111);
\draw[->] (1000-1100) -- (1110-1111);
\draw[->] (1110-1111) -- (1120-1222);
\draw[->] (1110-1111) -- (1121-1221);
\draw[->] (1110-1111) -- (1122-1220);
\draw[->] (1110-1111) -- (1121-1221);
\draw[->] (1120-1222) -- (1231-1232);
\draw[->] (1121-1221) -- (1231-1232);
\draw[->] (1121-1221) -- (1231-1232);
\draw[->] (1122-1220) -- (1231-1232);
\draw[->] (1231-1232) -- (1242-1342);
\draw[->] (1231-1232) -- (1242-1342);
\draw[->] (1242-1342) -- (2342);
\end{tikzpicture}
}
\newcommand{\FwPosetCD}{
\begin{tikzpicture}[
    >=stealth, scale=0.65,
    every node/.style={
        draw,
        rounded corners=2pt,
        inner sep=2pt,
        font=\small
    }
]
\node[fill=blue!20] (0100-0120-0122) at (-3.00,-10.00) {0100-0120-0122};
\node[fill=gray] (0110-0121-0111) at (0.00,-10.00) {0110-0121-0111};
\node[fill=blue!20] (1000) at (3.00,-10.00) {1000};
\node[fill=blue!20] (1100-1120-1122) at (-1.50,-8.00) {1100-1120-1122};
\node[fill=gray] (1110-1121-1111) at (1.50,-8.00) {1110-1121-1111};
\node[fill=blue!20] (1220-1242-1222) at (-1.50,-6.00) {1220-1242-1222};
\node[fill=gray] (1221-1231-1232) at (1.50,-6.00) {1221-1231-1232};
\node[fill=blue!20] (1342) at (0.00,-4.00) {1342};
\node[fill=blue!20] (2342) at (0.00,-2.00) {2342};

\draw[->] (0100-0120-0122) -- (1100-1120-1122);
\draw[->] (0110-0121-0111) -- (1110-1121-1111);
\draw[->] (1000) -- (1100-1120-1122);
\draw[->] (1000) -- (1110-1121-1111);
\draw[->] (1100-1120-1122) -- (1220-1242-1222);
\draw[->] (1100-1120-1122) -- (1221-1231-1232);
\draw[->] (1110-1121-1111) -- (1221-1231-1232);
\draw[->] (1110-1121-1111) -- (1220-1242-1222);
\draw[->] (1110-1121-1111) -- (1221-1231-1232);
\draw[->] (1220-1242-1222) -- (1342);
\draw[->] (1221-1231-1232) -- (1342);
\draw[->] (1342) -- (2342);
\end{tikzpicture}
}
\newcommand{\FwPosetBCABCBA}{
\begin{tikzpicture}[
    >=stealth, scale=0.65,
    every node/.style={
        draw,
        rounded corners=2pt,
        inner sep=2pt,
        font=\small
    }
]
\node[fill=gray] (0001-1221-0121-1121) at (-1.50,-4.00) {0001-1221-0121-1121};
\node[fill=gray] (0011-1111-0111-1231) at (1.50,-4.00) {0011-1111-0111-1231};
\node[fill=blue!20] (0122-2342) at (-3.00,-2.00) {0122-2342};
\node[fill=gray] (1122-1222-1342-1242) at (0.00,-2.00) {1122-1222-1342-1242};
\node[fill=blue!20] (1232) at (3.00,-2.00) {1232};

\draw[->] (0001-1221-0121-1121) -- (0122-2342);
\draw[->] (0001-1221-0121-1121) -- (1122-1222-1342-1242);
\draw[->] (0001-1221-0121-1121) -- (1232);
\draw[->] (0011-1111-0111-1231) -- (1232);
\draw[->] (0011-1111-0111-1231) -- (0122-2342);
\draw[->] (0011-1111-0111-1231) -- (1122-1222-1342-1242);
\end{tikzpicture}
}
\newcommand{\FwPosetCBCABCB}{
\begin{tikzpicture}[
    >=stealth, scale=0.65,
    every node/.style={
        draw,
        rounded corners=2pt,
        inner sep=2pt,
        font=\small
    }
]
\node[fill=blue!20] (0001-1231) at (-1.50,-4.00) {0001-1231};
\node[fill=gray] (0011-1121-1111-1221-0111-0121) at (1.50,-4.00) {0011-1121-1111-1221-0111-0121};
\node[fill=gray] (0122-1242-1122-2342-1222-1342) at (-1.50,-2.00) {0122-1242-1122-2342-1222-1342};
\node[fill=blue!20] (1232) at (1.50,-2.00) {1232};

\draw[->] (0001-1231) -- (1232);
\draw[->] (0001-1231) -- (0122-1242-1122-2342-1222-1342);
\draw[->] (0011-1121-1111-1221-0111-0121) -- (0122-1242-1122-2342-1222-1342);
\draw[->] (0011-1121-1111-1221-0111-0121) -- (0122-1242-1122-2342-1222-1342);
\draw[->] (0011-1121-1111-1221-0111-0121) -- (1232);
\end{tikzpicture}
}
\newcommand{\FwPosetCBCABCABA}{
\begin{tikzpicture}[
    >=stealth, scale=0.65,
    every node/.style={
        draw,
        rounded corners=2pt,
        inner sep=2pt,
        font=\small
    }
]
\node[fill=blue!20] (0001-1231) at (-4.50,-4.00) {0001-1231};
\node[fill=blue!20] (0011-1221) at (-1.50,-4.00) {0011-1221};
\node[fill=blue!20] (0111-1121) at (1.50,-4.00) {0111-1121};
\node[fill=blue!20] (0121-1111) at (4.50,-4.00) {0121-1111};
\node[fill=blue!20] (0122-2342) at (-4.50,-2.00) {0122-2342};
\node[fill=blue!20] (1122-1342) at (-1.50,-2.00) {1122-1342};
\node[fill=blue!20] (1222-1242) at (1.50,-2.00) {1222-1242};
\node[fill=blue!20] (1232) at (4.50,-2.00) {1232};

\draw[->] (0001-1231) -- (1232);
\draw[->] (0001-1231) -- (1222-1242);
\draw[->] (0001-1231) -- (1122-1342);
\draw[->] (0001-1231) -- (0122-2342);
\draw[->] (0011-1221) -- (1222-1242);
\draw[->] (0011-1221) -- (1232);
\draw[->] (0011-1221) -- (0122-2342);
\draw[->] (0011-1221) -- (1122-1342);
\draw[->] (0111-1121) -- (1122-1342);
\draw[->] (0111-1121) -- (0122-2342);
\draw[->] (0111-1121) -- (1232);
\draw[->] (0111-1121) -- (1222-1242);
\draw[->] (0121-1111) -- (0122-2342);
\draw[->] (0121-1111) -- (1122-1342);
\draw[->] (0121-1111) -- (1222-1242);
\draw[->] (0121-1111) -- (1232);
\end{tikzpicture}
}
\newcommand{\FwPosetABD}{
\begin{tikzpicture}[
    >=stealth, scale=0.65,
    every node/.style={
        draw,
        rounded corners=2pt,
        inner sep=2pt,
        font=\small
    }
]
\node[fill=gray] (0010-1111-0110-0011-1110-0111) at (0.00,-8.00) {0010-1111-0110-0011-1110-0111};
\node[fill=gray] (0120-1122-1220-0122-1120-1222) at (-1.50,-6.00) {0120-1122-1220-0122-1120-1222};
\node[fill=gray] (0121-1121-1221) at (1.50,-6.00) {0121-1121-1221};
\node[fill=gray] (1231-1232) at (0.00,-4.00) {1231-1232};
\node[fill=gray] (1242-2342-1342) at (0.00,-2.00) {1242-2342-1342};

\draw[->] (0010-1111-0110-0011-1110-0111) -- (0120-1122-1220-0122-1120-1222);
\draw[->] (0010-1111-0110-0011-1110-0111) -- (0121-1121-1221);
\draw[->] (0120-1122-1220-0122-1120-1222) -- (1231-1232);
\draw[->] (0121-1121-1221) -- (1231-1232);
\draw[->] (1231-1232) -- (1242-2342-1342);
\end{tikzpicture}
}
\newcommand{\FwPosetACD}{
\begin{tikzpicture}[
    >=stealth, scale=0.65,
    every node/.style={
        draw,
        rounded corners=2pt,
        inner sep=2pt,
        font=\small
    }
]
\node[fill=gray] (0100-1120-0122-1100-0120-1122) at (-1.50,-6.00) {0100-1120-0122-1100-0120-1122};
\node[fill=gray] (0110-1121-0111-1110-0121-1111) at (1.50,-6.00) {0110-1121-0111-1110-0121-1111};
\node[fill=blue!20] (1220-1242-1222) at (-1.50,-4.00) {1220-1242-1222};
\node[fill=gray] (1221-1231-1232) at (1.50,-4.00) {1221-1231-1232};
\node[fill=gray] (1342-2342) at (0.00,-2.00) {1342-2342};

\draw[->] (0100-1120-0122-1100-0120-1122) -- (1220-1242-1222);
\draw[->] (0100-1120-0122-1100-0120-1122) -- (1221-1231-1232);
\draw[->] (0110-1121-0111-1110-0121-1111) -- (1221-1231-1232);
\draw[->] (0110-1121-0111-1110-0121-1111) -- (1220-1242-1222);
\draw[->] (0110-1121-0111-1110-0121-1111) -- (1221-1231-1232);
\draw[->] (1220-1242-1222) -- (1342-2342);
\draw[->] (1221-1231-1232) -- (1342-2342);
\end{tikzpicture}
}
\newcommand{\FwPosetCBDCBCD}{
\begin{tikzpicture}[
    >=stealth, scale=0.65,
    every node/.style={
        draw,
        rounded corners=2pt,
        inner sep=2pt,
        font=\small
    }
]
\node[fill=blue!20] (1000-1242-1220-1222) at (-4.50,-4.00) {1000-1242-1220-1222};
\node[fill=blue!20] (1100-1122-1120-1342) at (-1.50,-4.00) {1100-1122-1120-1342};
\node[fill=gray] (1110-1232) at (1.50,-4.00) {1110-1232};
\node[fill=gray] (1111-1121-1231-1221) at (4.50,-4.00) {1111-1121-1231-1221};
\node[fill=blue!20] (2342) at (0.00,-2.00) {2342};

\draw[->] (1000-1242-1220-1222) -- (2342);
\draw[->] (1100-1122-1120-1342) -- (2342);
\draw[->] (1110-1232) -- (2342);
\draw[->] (1111-1121-1231-1221) -- (2342);
\end{tikzpicture}
}
\newcommand{\FwPosetBCBDCBC}{
\begin{tikzpicture}[
    >=stealth, scale=0.65,
    every node/.style={
        draw,
        rounded corners=2pt,
        inner sep=2pt,
        font=\small
    }
]
\node[fill=blue!20] (1000-1342) at (-3.00,-4.00) {1000-1342};
\node[fill=gray] (1100-1222-1122-1242-1120-1220) at (0.00,-4.00) {1100-1222-1122-1242-1120-1220};
\node[fill=gray] (1110-1221-1111-1232-1121-1231) at (3.00,-4.00) {1110-1221-1111-1232-1121-1231};
\node[fill=blue!20] (2342) at (0.00,-2.00) {2342};

\draw[->] (1000-1342) -- (2342);
\draw[->] (1100-1222-1122-1242-1120-1220) -- (2342);
\draw[->] (1110-1221-1111-1232-1121-1231) -- (2342);
\end{tikzpicture}
}
\newcommand{\FwPosetBCBDCBDCD}{
\begin{tikzpicture}[
    >=stealth, scale=0.65,
    every node/.style={
        draw,
        rounded corners=2pt,
        inner sep=2pt,
        font=\small
    }
]
\node[fill=blue!20] (1000-1342) at (-9.00,-4.00) {1000-1342};
\node[fill=blue!20] (1100-1242) at (-6.00,-4.00) {1100-1242};
\node[fill=gray] (1110-1232) at (-3.00,-4.00) {1110-1232};
\node[fill=gray] (1111-1231) at (0.00,-4.00) {1111-1231};
\node[fill=blue!20] (1120-1222) at (3.00,-4.00) {1120-1222};
\node[fill=gray] (1121-1221) at (6.00,-4.00) {1121-1221};
\node[fill=blue!20] (1122-1220) at (9.00,-4.00) {1122-1220};
\node[fill=blue!20] (2342) at (0.00,-2.00) {2342};

\draw[->] (1000-1342) -- (2342);
\draw[->] (1100-1242) -- (2342);
\draw[->] (1110-1232) -- (2342);
\draw[->] (1111-1231) -- (2342);
\draw[->] (1120-1222) -- (2342);
\draw[->] (1121-1221) -- (2342);
\draw[->] (1122-1220) -- (2342);
\end{tikzpicture}
}

\newcommand{\blockFive}{

\begin{tikzpicture}[
    x=8mm,y=8mm,
    >=Stealth,
    every node/.style={font=\scriptsize}
]

\draw[step=1] (0,0) grid (5,-5);

\fill[gray!55] (0,0) rectangle (2,-2);
\fill[gray!55] (2,-2) rectangle (5,-5);

\foreach \i in {1,...,5}{
  \node[left]  at (0,-\i+0.5) {$\i$};
  \node[above] at (\i-0.5,0) {$\i$};
}

\node[text=red]  (r13) at (2.5,-0.5) {1100};
\node[text=red]  (r24) at (3.5,-1.5) {0110};
\node[text=red]  (r15) at (4.5,-0.5) {1111};
\node[text=red] (b14) at (3.5,-0.5) {1110};
\node[text=red] (b25) at (4.5,-1.5) {0111};
\node[text=red] (b23) at (2.5,-1.5) {0100};

\draw[->,red,thick,bend left=18] (r13) to (r24);
\draw[->,red,thick,bend left=18] (r24) to (r15);
\draw[->,red,thick,bend left=18] (r15) to (b23);
\draw[->,red,thick,bend left=18] (b23) to (b14);
\draw[->,red,thick,bend left=18] (b14) to (b25);
\draw[->,red,thick,bend left=18] (b25) to (r13);

\end{tikzpicture}
}

\newcommand{\GraphCSix}{
\begin{tikzpicture}[
    >={Stealth[scale=1.2]}, scale=.5,
    vertex/.style={circle, draw=black, thick, minimum size=10mm, inner sep=0pt, font=\Large},
    edge/.style={->, thick}
]

    \def\Rbeta{4.0}  
    \def\Ralpha{2.6} 

    \node[vertex] (b1) at (90:\Rbeta)  {$\beta_1$};
    \node[vertex] (b2) at (210:\Rbeta) {$\beta_2$};
    \node[vertex] (b3) at (330:\Rbeta) {$\beta_3$};

    \node[vertex] (a3) at (150:\Ralpha) {$\alpha_3$};

    \node[vertex] (a1) at (270:\Ralpha) {$\alpha_1$};

    \node[vertex] (a2) at (30:\Ralpha)  {$\alpha_2$};

    \draw[edge] (b1) -- (a3);
    \draw[edge] (b2) -- (a3);

    \draw[edge] (b2) -- (a1);
    \draw[edge] (b3) -- (a1);

    \draw[edge] (b1) -- (a2);
    \draw[edge] (b3) -- (a2);
\end{tikzpicture}
}

\newcommand{\FKfull}{
\begin{tikzpicture}[
    >=stealth, scale=0.6,
    every node/.style={draw,circle,minimum size=8mm,font=\small}
]
\node[fill=blue!20] (0001) at (-10.50,-10.00) {0001};
\node[fill=blue!20] (0010) at (-7.50,-10.00) {0010};
\node[fill=blue!20] (0100) at (-4.50,-10.00) {0100};
\node[fill=blue!20] (1000) at (-1.50,-10.00) {1000};
\node[fill=blue!20] (1100) at (1.50,-10.00) {1100};
\node[fill=blue!20] (1220) at (4.50,-10.00) {1220};
\node[fill=blue!20] (1342) at (7.50,-10.00) {1342};
\node[fill=blue!20] (2342) at (10.50,-10.00) {2342};
\node[fill=blue!20] (0011) at (-4.50,-8.00) {0011};
\node[fill=blue!20] (0110) at (-1.50,-8.00) {0110};
\node[fill=blue!20] (1110) at (1.50,-8.00) {1110};
\node[fill=blue!20] (1221) at (4.50,-8.00) {1221};
\node[fill=blue!20] (0111) at (-7.50,-6.00) {0111};
\node[fill=blue!20] (0120) at (-4.50,-6.00) {0120};
\node[fill=blue!20] (1111) at (-1.50,-6.00) {1111};
\node[fill=blue!20] (1120) at (1.50,-6.00) {1120};
\node[fill=blue!20] (1222) at (4.50,-6.00) {1222};
\node[fill=blue!20] (1231) at (7.50,-6.00) {1231};
\node[fill=blue!20] (0121) at (-3.00,-4.00) {0121};
\node[fill=blue!20] (1121) at (0.00,-4.00) {1121};
\node[fill=blue!20] (1232) at (3.00,-4.00) {1232};
\node[fill=blue!20] (0122) at (-3.00,-2.00) {0122};
\node[fill=blue!20] (1122) at (0.00,-2.00) {1122};
\node[fill=blue!20] (1242) at (3.00,-2.00) {1242};

\draw[->] (1222) -- (1232);
\draw[->] (1111) -- (1121);
\draw[->] (1111) -- (1122);
\draw[->] (0120) -- (0121);
\draw[->] (1121) -- (1122);
\draw[->] (1110) -- (1120);
\draw[->] (1110) -- (1111);
\draw[->] (1110) -- (1121);
\draw[->] (1232) -- (1242);
\draw[->] (1120) -- (1121);
\draw[->] (0111) -- (0121);
\draw[->] (0111) -- (0122);
\draw[->] (1231) -- (1232);
\draw[->] (1231) -- (1242);
\draw[->] (0010) -- (0011);
\draw[->] (1100) -- (1110);
\draw[->] (1100) -- (1111);
\draw[->] (0121) -- (0122);
\draw[->] (0100) -- (0110);
\draw[->] (0100) -- (0111);
\draw[->] (1220) -- (1221);
\draw[->] (1220) -- (1231);
\draw[->] (0001) -- (0011);
\draw[->] (0110) -- (0120);
\draw[->] (0110) -- (0111);
\draw[->] (0110) -- (0121);
\draw[->] (1221) -- (1231);
\draw[->] (1221) -- (1222);
\draw[->] (1221) -- (1232);
\end{tikzpicture}
}

\newcommand{\FKfour}{
\begin{tikzpicture}[
    >=stealth, scale=0.7,
    every node/.style={draw,circle,minimum size=8mm,font=\small}
]
\node[fill=blue!20] (0121) at (-3.00,-4.00) {0121};
\node[fill=blue!20] (1121) at (0.00,-4.00) {1121};
\node[fill=blue!20] (1232) at (3.00,-4.00) {1232};
\node[fill=blue!20] (0122) at (-3.00,-2.00) {0122};
\node[fill=blue!20] (1122) at (0.00,-2.00) {1122};
\node[fill=blue!20] (1242) at (3.00,-2.00) {1242};

\draw[->] (0121) -- (0122);
\draw[->] (1232) -- (1242);
\draw[->] (1121) -- (1122);
\end{tikzpicture}
}

\newcommand{\FKthree}{
\begin{tikzpicture}[
    >=stealth, scale=0.7,
    every node/.style={draw,circle,minimum size=8mm,font=\small}
]
\node[fill=blue!20] (0111) at (-7.50,-6.00) {0111};
\node[fill=blue!20] (0120) at (-4.50,-6.00) {0120};
\node[fill=blue!20] (1111) at (-1.50,-6.00) {1111};
\node[fill=blue!20] (1120) at (1.50,-6.00) {1120};
\node[fill=blue!20] (1222) at (4.50,-6.00) {1222};
\node[fill=blue!20] (1231) at (7.50,-6.00) {1231};
\node[fill=blue!20] (0121) at (-3.00,-4.00) {0121};
\node[fill=blue!20] (1121) at (0.00,-4.00) {1121};
\node[fill=blue!20] (1232) at (3.00,-4.00) {1232};

\draw[->] (0111) -- (0121);
\draw[->] (1222) -- (1232);
\draw[->] (1120) -- (1121);
\draw[->] (1231) -- (1232);
\draw[->] (1111) -- (1121);
\draw[->] (0120) -- (0121);
\end{tikzpicture}
}

\newcommand{\FKtwo}{
\begin{tikzpicture}[
    >=stealth, scale=0.7,
    every node/.style={draw,circle,minimum size=8mm,font=\small}
]
\node[fill=blue!20] (0011) at (-4.50,-8.00) {0011};
\node[fill=blue!20] (0110) at (-1.50,-8.00) {0110};
\node[fill=blue!20] (1110) at (1.50,-8.00) {1110};
\node[fill=blue!20] (1221) at (4.50,-8.00) {1221};
\node[fill=blue!20] (0111) at (-7.50,-6.00) {0111};
\node[fill=blue!20] (0120) at (-4.50,-6.00) {0120};
\node[fill=blue!20] (1111) at (-1.50,-6.00) {1111};
\node[fill=blue!20] (1120) at (1.50,-6.00) {1120};
\node[fill=blue!20] (1222) at (4.50,-6.00) {1222};
\node[fill=blue!20] (1231) at (7.50,-6.00) {1231};

\draw[->] (0110) -- (0120);
\draw[->] (0110) -- (0111);
\draw[->] (1110) -- (1120);
\draw[->] (1110) -- (1111);
\draw[->] (1221) -- (1231);
\draw[->] (1221) -- (1222);
\end{tikzpicture}
}

\newcommand{\FKone}{
\begin{tikzpicture}[
    >=stealth, scale=.7,
    every node/.style={draw,circle,minimum size=8mm,font=\small}
]
\node[fill=blue!20] (0011) at (-4.50,-8.00) {0011};
\node[fill=blue!20] (0110) at (-1.50,-8.00) {0110};
\node[fill=blue!20] (1110) at (1.50,-8.00) {1110};
\node[fill=blue!20] (1221) at (4.50,-8.00) {1221};
\node[fill=blue!20] (0001) at (-10.50,-10.00) {0001};
\node[fill=blue!20] (0010) at (-7.50,-10.00) {0010};
\node[fill=blue!20] (0100) at (-4.50,-10.00) {0100};
\node[fill=blue!20] (1000) at (-1.50,-10.00) {1000};
\node[fill=blue!20] (1100) at (1.50,-10.00) {1100};
\node[fill=blue!20] (1220) at (4.50,-10.00) {1220};
\node[fill=blue!20] (1342) at (7.50,-10.00) {1342};
\node[fill=blue!20] (2342) at (10.50,-10.00) {2342};

\draw[->] (1100) -- (1110);
\draw[->] (1220) -- (1221);
\draw[->] (0001) -- (0011);
\draw[->] (0010) -- (0011);
\draw[->] (0100) -- (0110);
\end{tikzpicture}
}

  \newcommand{\orbitSix}[7][]{
    \begin{scope}[#1]
      \node (n1) at (0, 1.2) {$#2$};
      \node (n2) at (1.8, 1.2) {$#3$};
      \node (n3) at (3.6, 1.2) {$#4$};
      \node (n4) at (3.6, 0) {$#5$};
      \node (n5) at (1.8, 0) {$#6$};
      \node (n6) at (0, 0) {$#7$};
      
      \draw[thick, -] (n1) to[bend left=15] (n2);
      \draw[thick, -] (n2) to[bend left=15] (n3);
      \draw[thick, -] (n3) to[bend left=15] (n4);
      \draw[thick, -] (n4) to[bend left=15] (n5);
      \draw[thick, -] (n5) to[bend left=15] (n6);
      \draw[thick, -] (n6) to[bend left=15] (n1);
    \end{scope}
  }

\newcommand{\BranchingA}{
\begin{tikzpicture}

  \orbitSix[shift={(0,0)}]{0001}{0011}{0121}{1111}{1221}{1231}

  \draw[double, ->, thick] (4.4, 0.6) -- (6.0, 0.6);

  \orbitSix[shift={(4.57, 0.95)}, scale=0.35, every node/.append style={transform shape}]{0001}{0011}{0121}{1111}{1221}{1231}

  \node at (6.8, 0.6) {$1232$};
  \node at (7.8, 0.6) {$\oplus$};
  \orbitSix[shift={(8.7,0)}]{0122}{1122}{1222}{1242}{1342}{2342}

\end{tikzpicture}
}

\newcommand{\orbitThree}[7][]{
    \begin{scope}[#1]
      \node (n1) at (0, 1.2) {$#2$};
      \node (n2) at (1.8, 1.2) {$#3$};
      \node (n3) at (0.9, 0) {$#4$};
      \draw[thick, -] (n1) to[bend left=15] (n2);
      \draw[thick, -] (n2) to[bend left=15] (n3);
      \draw[thick, -] (n3) to[bend left=15] (n1);
    \end{scope}
}

\newcommand{\BranchingB}{
\begin{tikzpicture}
  \orbitSix[shift={(0,0)}]{0010}{0011}{0110}{0111}{1110}{1111}

  \draw[double, ->, thick] (4.4, 0.6) -- (6.0, 0.6);

  \orbitSix[shift={(4.57, 0.95)}, scale=0.35, every node/.append style={transform shape}]{0010}{0011}{0110}{0111}{1110}{1111}

  \node at (9.5, 0.6) {$\oplus$};
  \orbitThree[shift={(7.2, 0)}]{0121}{1121}{1221}{}{}{}
  \orbitSix[shift={(10.2,0)}]{0120}{0122}{1120}{1122}{1220}{1222}
\end{tikzpicture}
}

\newcommand{\BranchingC}{
\begin{tikzpicture}
  \orbitSix[shift={(0,0)}]{0110}{0111}{0121}{1110}{1111}{1121}

  \draw[double, ->, thick] (4.4, 0.6) -- (6.0, 0.6);


  \orbitSix[shift={(4.57, 0.95)}, scale=0.35, every node/.append style={transform shape}]{0110}{0111}{0121}{1110}{1111}{1121}

    \orbitThree[shift={(7.2,0)}]{1221}{1231}{1232}{}{}{}

  \node at (9.5,0.6) {$\oplus$};

    \orbitThree[shift={(10.2,0)}]{1220}{1222}{1242}{}{}{}
\end{tikzpicture}
}

\newcommand{\CaseN}{
\begin{tikzpicture}[
    >=stealth, scale=0.8,
    every node/.style={draw,circle,minimum size=8mm,font=\small}
]
\node[fill=red!20] (0001) at (-4.50,-22.00) {0001};
\node[fill=gray] (0010) at (-1.50,-22.00) {0010};
\node[fill=cyan!20] (0100) at (1.50,-22.00) {0100};
\node[fill=gray] (1000) at (4.50,-22.00) {1000};
\node[fill=red!20] (0011) at (-3.00,-20.00) {0011};
\node[fill=cyan!20] (0110) at (0.00,-20.00) {0110};
\node[fill=red!20] (1100) at (3.00,-20.00) {1100};
\node[fill=blue!20] (0111) at (-3.00,-18.00) {0111};
\node[fill=cyan!20] (0120) at (0.00,-18.00) {0120};
\node[fill=red!20] (1110) at (3.00,-18.00) {1110};
\node[fill=blue!20] (0121) at (-3.00,-16.00) {0121};
\node[fill=blue!20] (1111) at (0.00,-16.00) {1111};
\node[fill=red!20] (1120) at (3.00,-16.00) {1120};
\node[fill=blue!20] (0122) at (-3.00,-14.00) {0122};
\node[fill=blue!20] (1121) at (0.00,-14.00) {1121};
\node[fill=blue!20] (1220) at (3.00,-14.00) {1220};
\node[fill=blue!20] (1122) at (-1.50,-12.00) {1122};
\node[fill=blue!20] (1221) at (1.50,-12.00) {1221};
\node[fill=blue!20] (1222) at (-1.50,-10.00) {1222};
\node[fill=blue!20] (1231) at (1.50,-10.00) {1231};
\node[fill=blue!20] (1232) at (0.00,-8.00) {1232};
\node[fill=blue!20] (1242) at (0.00,-6.00) {1242};
\node[fill=blue!20] (1342) at (0.00,-4.00) {1342};
\node[fill=blue!20] (2342) at (0.00,-2.00) {2342};

\draw[->] (1110) -- (1120);
\draw[->] (1110) -- (1111);
\draw[->] (1122) -- (1222);
\draw[->] (1231) -- (1232);
\draw[->] (1342) -- (2342);
\draw[->] (1121) -- (1221);
\draw[->] (1121) -- (1122);
\draw[->] (1242) -- (1342);
\draw[->] (1221) -- (1231);
\draw[->] (1221) -- (1222);
\draw[->] (0121) -- (1121);
\draw[->] (0121) -- (0122);
\draw[->] (1000) -- (1100);
\draw[->] (0100) -- (1100);
\draw[->] (0100) -- (0110);
\draw[->] (1111) -- (1121);
\draw[->] (1120) -- (1220);
\draw[->] (1120) -- (1121);
\draw[->] (0010) -- (0110);
\draw[->] (0010) -- (0011);
\draw[->] (0011) -- (0111);
\draw[->] (1220) -- (1221);
\draw[->] (1232) -- (1242);
\draw[->] (0122) -- (1122);
\draw[->] (0111) -- (1111);
\draw[->] (0111) -- (0121);
\draw[->] (0001) -- (0011);
\draw[->] (0120) -- (1120);
\draw[->] (0120) -- (0121);
\draw[->] (0110) -- (1110);
\draw[->] (0110) -- (0120);
\draw[->] (0110) -- (0111);
\draw[->] (1222) -- (1232);
\draw[->] (1100) -- (1110);
\end{tikzpicture}
}

\newcommand{\CaseNF}{
\begin{tikzpicture}[
    >=stealth, scale=0.8,
    every node/.style={draw,circle,minimum size=8mm,font=\small}
]
\node[fill=red!20] (0001) at (-6.50,-10.00) {0001};
\node[fill=red!20] (0011) at (-5.00,-10.00) {0011};
\node[fill=cyan!20] (0100) at (-3.50,-10.00) {0100};
\node[fill=cyan!20] (0110) at (-1.50,-10.00) {0110};
\node[fill=cyan!20] (0120) at (1.50,-10.00) {0120};
\node[fill=red!20] (1100) at (3.50,-10.00) {1100};
\node[fill=red!20] (1110) at (5.00,-10.00) {1110};
\node[fill=red!20] (1120) at (6.50,-10.00) {1120};
\node[fill=blue!20] (0111) at (-6.00,-8.00) {0111};
\node[fill=blue!20] (0121) at (-3.00,-8.00) {0121};
\node[fill=blue!20] (1111) at (0.00,-8.00) {1111};
\node[fill=blue!20] (1121) at (3.00,-8.00) {1121};
\node[fill=blue!20] (1220) at (6.00,-8.00) {1220};
\node[fill=blue!20] (0122) at (-4.50,-6.00) {0122};
\node[fill=blue!20] (1122) at (-1.50,-6.00) {1122};
\node[fill=blue!20] (1221) at (1.50,-6.00) {1221};
\node[fill=blue!20] (1231) at (4.50,-6.00) {1231};
\node[fill=blue!20] (1222) at (-3.00,-4.00) {1222};
\node[fill=blue!20] (1232) at (0.00,-4.00) {1232};
\node[fill=blue!20] (1242) at (3.00,-4.00) {1242};
\node[fill=blue!20] (1342) at (-1.50,-2.00) {1342};
\node[fill=blue!20] (2342) at (1.50,-2.00) {2342};

\draw[->] (0011) -- (1121);
\draw[->] (0011) -- (0111);
\draw[->] (0011) -- (0121);
\draw[->] (0011) -- (1111);
\draw[->] (1110) -- (1121);
\draw[->] (1110) -- (1111);
\draw[->] (1110) -- (1220);
\draw[->] (1122) -- (1222);
\draw[->] (1122) -- (1242);
\draw[->] (1122) -- (1232);
\draw[->] (1231) -- (1242);
\draw[->] (1231) -- (1232);
\draw[->] (1220) -- (1231);
\draw[->] (1220) -- (1221);
\draw[->] (1232) -- (2342);
\draw[->] (1232) -- (1342);
\draw[->] (0122) -- (1232);
\draw[->] (0122) -- (1242);
\draw[->] (0122) -- (1222);
\draw[->] (1121) -- (1221);
\draw[->] (1121) -- (1122);
\draw[->] (1121) -- (1231);
\draw[->] (1242) -- (1342);
\draw[->] (1242) -- (2342);
\draw[->] (1221) -- (1232);
\draw[->] (1221) -- (1222);
\draw[->] (0121) -- (1231);
\draw[->] (0121) -- (0122);
\draw[->] (0121) -- (1221);
\draw[->] (0111) -- (0122);
\draw[->] (0111) -- (1221);
\draw[->] (0111) -- (1231);
\draw[->] (0100) -- (0111);
\draw[->] (0100) -- (1220);
\draw[->] (0001) -- (1111);
\draw[->] (0001) -- (0121);
\draw[->] (0001) -- (0111);
\draw[->] (0001) -- (1121);
\draw[->] (0120) -- (0121);
\draw[->] (0120) -- (1220);
\draw[->] (0110) -- (0121);
\draw[->] (0110) -- (1220);
\draw[->] (0110) -- (0111);
\draw[->] (1111) -- (1122);
\draw[->] (1111) -- (1231);
\draw[->] (1111) -- (1221);
\draw[->] (1120) -- (1220);
\draw[->] (1120) -- (1121);
\draw[->] (1222) -- (1342);
\draw[->] (1222) -- (2342);
\draw[->] (1100) -- (1111);
\draw[->] (1100) -- (1220);
\end{tikzpicture}
}

\newcommand{\KRPOIOO}{
    \begin{tikzpicture}[
    >=stealth, scale=0.8,
    every node/.style={draw,circle,minimum size=8mm,font=\small}
]
\node[fill=blue!20] (0001) at (-6.00,-16.00) {0001};
\node[fill=blue!20] (0010) at (-3.00,-16.00) {0010};
\node[fill=blue!20] (0110) at (0.00,-16.00) {0110};
\node[fill=blue!20] (1000) at (3.00,-16.00) {1000};
\node[fill=blue!20] (1100) at (6.00,-16.00) {1100};
\node[fill=blue!20] (0011) at (-4.50,-14.00) {0011};
\node[fill=blue!20] (0111) at (-1.50,-14.00) {0111};
\node[fill=blue!20] (0120) at (1.50,-14.00) {0120};
\node[fill=blue!20] (1110) at (4.50,-14.00) {1110};
\node[fill=blue!20] (0121) at (-4.50,-12.00) {0121};
\node[fill=blue!20] (1111) at (-1.50,-12.00) {1111};
\node[fill=blue!20] (1120) at (1.50,-12.00) {1120};
\node[fill=blue!20] (1220) at (4.50,-12.00) {1220};
\node[fill=blue!20] (0122) at (-3.00,-10.00) {0122};
\node[fill=blue!20] (1121) at (0.00,-10.00) {1121};
\node[fill=blue!20] (1221) at (3.00,-10.00) {1221};
\node[fill=blue!20] (1122) at (-3.00,-8.00) {1122};
\node[fill=blue!20] (1222) at (0.00,-8.00) {1222};
\node[fill=blue!20] (1231) at (3.00,-8.00) {1231};
\node[fill=blue!20] (1232) at (0.00,-6.00) {1232};
\node[fill=blue!20] (1242) at (-1.50,-4.00) {1242};
\node[fill=blue!20] (1342) at (1.50,-4.00) {1342};
\node[fill=blue!20] (2342) at (0.00,-2.00) {2342};

\draw[->] (1342) -- (2342);
\draw[->] (1121) -- (1231);
\draw[->] (1121) -- (1122);
\draw[->] (0121) -- (1221);
\draw[->] (0121) -- (0122);
\draw[->] (0121) -- (1121);
\draw[->] (0120) -- (1220);
\draw[->] (0120) -- (0121);
\draw[->] (0120) -- (1120);
\draw[->] (0110) -- (0120);
\draw[->] (0110) -- (0111);
\draw[->] (0110) -- (1110);
\draw[->] (0011) -- (0121);
\draw[->] (0011) -- (1111);
\draw[->] (1220) -- (1221);
\draw[->] (1111) -- (1221);
\draw[->] (1111) -- (1121);
\draw[->] (1221) -- (1231);
\draw[->] (1221) -- (1222);
\draw[->] (1231) -- (1232);
\draw[->] (1120) -- (1121);
\draw[->] (1242) -- (2342);
\draw[->] (1222) -- (1232);
\draw[->] (1110) -- (1220);
\draw[->] (1110) -- (1120);
\draw[->] (1110) -- (1111);
\draw[->] (0111) -- (0121);
\draw[->] (0111) -- (1111);
\draw[->] (0122) -- (1222);
\draw[->] (0122) -- (1122);
\draw[->] (1122) -- (1232);
\draw[->] (1100) -- (1110);
\draw[->] (0010) -- (0120);
\draw[->] (0010) -- (1110);
\draw[->] (0010) -- (0011);
\draw[->] (1232) -- (1342);
\draw[->] (1232) -- (1242);
\draw[->] (0001) -- (0111);
\draw[->] (0001) -- (0011);
\draw[->] (1000) -- (1110);
\end{tikzpicture}
}

\newcommand{\KRPOOIO}{
\begin{tikzpicture}[
    >=stealth, scale=0.8,
    every node/.style={draw,circle,minimum size=8mm,font=\small}
]
\node[fill=blue!20] (0001) at (-7.50,-14.00) {0001};
\node[fill=blue!20] (0011) at (-4.50,-14.00) {0011};
\node[fill=blue!20] (0100) at (-1.50,-14.00) {0100};
\node[fill=blue!20] (0110) at (1.50,-14.00) {0110};
\node[fill=blue!20] (0120) at (4.50,-14.00) {0120};
\node[fill=blue!20] (1000) at (7.50,-14.00) {1000};
\node[fill=blue!20] (0111) at (-6.00,-12.00) {0111};
\node[fill=blue!20] (0121) at (-3.00,-12.00) {0121};
\node[fill=blue!20] (1100) at (0.00,-12.00) {1100};
\node[fill=blue!20] (1110) at (3.00,-12.00) {1110};
\node[fill=blue!20] (1120) at (6.00,-12.00) {1120};
\node[fill=blue!20] (0122) at (-4.50,-10.00) {0122};
\node[fill=blue!20] (1111) at (-1.50,-10.00) {1111};
\node[fill=blue!20] (1121) at (1.50,-10.00) {1121};
\node[fill=blue!20] (1220) at (4.50,-10.00) {1220};
\node[fill=blue!20] (1122) at (-3.00,-8.00) {1122};
\node[fill=blue!20] (1221) at (0.00,-8.00) {1221};
\node[fill=blue!20] (1231) at (3.00,-8.00) {1231};
\node[fill=blue!20] (1222) at (-3.00,-6.00) {1222};
\node[fill=blue!20] (1232) at (0.00,-6.00) {1232};
\node[fill=blue!20] (1242) at (3.00,-6.00) {1242};
\node[fill=blue!20] (1342) at (0.00,-4.00) {1342};
\node[fill=blue!20] (2342) at (0.00,-2.00) {2342};

\draw[->] (1342) -- (2342);
\draw[->] (1121) -- (1231);
\draw[->] (1121) -- (1221);
\draw[->] (1121) -- (1122);
\draw[->] (0121) -- (0122);
\draw[->] (0121) -- (1121);
\draw[->] (0120) -- (0121);
\draw[->] (0120) -- (1120);
\draw[->] (0110) -- (0121);
\draw[->] (0110) -- (0111);
\draw[->] (0110) -- (1110);
\draw[->] (0011) -- (0121);
\draw[->] (0011) -- (0111);
\draw[->] (1220) -- (1231);
\draw[->] (1220) -- (1221);
\draw[->] (1111) -- (1231);
\draw[->] (1111) -- (1221);
\draw[->] (1111) -- (1122);
\draw[->] (1221) -- (1232);
\draw[->] (1221) -- (1222);
\draw[->] (1231) -- (1242);
\draw[->] (1231) -- (1232);
\draw[->] (1120) -- (1220);
\draw[->] (1120) -- (1121);
\draw[->] (1242) -- (1342);
\draw[->] (1222) -- (1342);
\draw[->] (1110) -- (1220);
\draw[->] (1110) -- (1121);
\draw[->] (1110) -- (1111);
\draw[->] (0111) -- (0122);
\draw[->] (0111) -- (1111);
\draw[->] (0122) -- (1122);
\draw[->] (0100) -- (0111);
\draw[->] (0100) -- (1100);
\draw[->] (1122) -- (1242);
\draw[->] (1122) -- (1232);
\draw[->] (1122) -- (1222);
\draw[->] (1100) -- (1220);
\draw[->] (1100) -- (1111);
\draw[->] (1232) -- (1342);
\draw[->] (0001) -- (0121);
\draw[->] (0001) -- (0111);
\draw[->] (1000) -- (1120);
\draw[->] (1000) -- (1110);
\draw[->] (1000) -- (1100);
\end{tikzpicture}
}

\newcommand{\KRPIOIO}{
    \begin{tikzpicture}[
    >=stealth, scale=0.8,
    every node/.style={draw,circle,minimum size=8mm,font=\small}
]
\node[fill=blue!20] (0001) at (-10.50,-10.00) {0001};
\node[fill=blue!20] (0011) at (-7.50,-10.00) {0011};
\node[fill=blue!20] (0100) at (-4.50,-10.00) {0100};
\node[fill=blue!20] (0110) at (-1.50,-10.00) {0110};
\node[fill=blue!20] (0120) at (1.50,-10.00) {0120};
\node[fill=blue!20] (1100) at (4.50,-10.00) {1100};
\node[fill=blue!20] (1110) at (7.50,-10.00) {1110};
\node[fill=blue!20] (1120) at (10.50,-10.00) {1120};
\node[fill=blue!20] (0111) at (-6.00,-8.00) {0111};
\node[fill=blue!20] (0121) at (-3.00,-8.00) {0121};
\node[fill=blue!20] (1111) at (0.00,-8.00) {1111};
\node[fill=blue!20] (1121) at (3.00,-8.00) {1121};
\node[fill=blue!20] (1220) at (6.00,-8.00) {1220};
\node[fill=blue!20] (0122) at (-4.50,-6.00) {0122};
\node[fill=blue!20] (1122) at (-1.50,-6.00) {1122};
\node[fill=blue!20] (1221) at (1.50,-6.00) {1221};
\node[fill=blue!20] (1231) at (4.50,-6.00) {1231};
\node[fill=blue!20] (1222) at (-3.00,-4.00) {1222};
\node[fill=blue!20] (1232) at (0.00,-4.00) {1232};
\node[fill=blue!20] (1242) at (3.00,-4.00) {1242};
\node[fill=blue!20] (1342) at (-1.50,-2.00) {1342};
\node[fill=blue!20] (2342) at (1.50,-2.00) {2342};

\draw[->] (1121) -- (1231);
\draw[->] (1121) -- (1221);
\draw[->] (1121) -- (1122);
\draw[->] (0121) -- (1231);
\draw[->] (0121) -- (1221);
\draw[->] (0121) -- (0122);
\draw[->] (0120) -- (1220);
\draw[->] (0120) -- (0121);
\draw[->] (0110) -- (0121);
\draw[->] (0110) -- (1220);
\draw[->] (0110) -- (0111);
\draw[->] (0011) -- (0121);
\draw[->] (0011) -- (1121);
\draw[->] (0011) -- (0111);
\draw[->] (0011) -- (1111);
\draw[->] (1220) -- (1231);
\draw[->] (1220) -- (1221);
\draw[->] (1111) -- (1231);
\draw[->] (1111) -- (1221);
\draw[->] (1111) -- (1122);
\draw[->] (1221) -- (1232);
\draw[->] (1221) -- (1222);
\draw[->] (1231) -- (1242);
\draw[->] (1231) -- (1232);
\draw[->] (1120) -- (1220);
\draw[->] (1120) -- (1121);
\draw[->] (1242) -- (1342);
\draw[->] (1242) -- (2342);
\draw[->] (1222) -- (1342);
\draw[->] (1222) -- (2342);
\draw[->] (1110) -- (1220);
\draw[->] (1110) -- (1121);
\draw[->] (1110) -- (1111);
\draw[->] (0111) -- (0122);
\draw[->] (0111) -- (1231);
\draw[->] (0111) -- (1221);
\draw[->] (0122) -- (1242);
\draw[->] (0122) -- (1232);
\draw[->] (0122) -- (1222);
\draw[->] (0100) -- (0111);
\draw[->] (0100) -- (1220);
\draw[->] (1122) -- (1242);
\draw[->] (1122) -- (1232);
\draw[->] (1122) -- (1222);
\draw[->] (1100) -- (1220);
\draw[->] (1100) -- (1111);
\draw[->] (1232) -- (1342);
\draw[->] (1232) -- (2342);
\draw[->] (0001) -- (0121);
\draw[->] (0001) -- (0111);
\draw[->] (0001) -- (1121);
\draw[->] (0001) -- (1111);
\end{tikzpicture}
}

\newcommand{\KRPIIII}{
    \begin{tikzpicture}[
    >=stealth, scale=0.6,
    every node/.style={draw,circle,minimum size=8mm,font=\small}
]
\node[fill=blue!20] (0100) at (-16.50,-6.00) {0100};
\node[fill=blue!20] (0110) at (-13.50,-6.00) {0110};
\node[fill=blue!20] (0111) at (-10.50,-6.00) {0111};
\node[fill=blue!20] (0120) at (-7.50,-6.00) {0120};
\node[fill=blue!20] (0121) at (-4.50,-6.00) {0121};
\node[fill=blue!20] (0122) at (-1.50,-6.00) {0122};
\node[fill=blue!20] (1100) at (1.50,-6.00) {1100};
\node[fill=blue!20] (1110) at (4.50,-6.00) {1110};
\node[fill=blue!20] (1111) at (7.50,-6.00) {1111};
\node[fill=blue!20] (1120) at (10.50,-6.00) {1120};
\node[fill=blue!20] (1121) at (13.50,-6.00) {1121};
\node[fill=blue!20] (1122) at (16.50,-6.00) {1122};
\node[fill=blue!20] (1220) at (-7.50,-4.00) {1220};
\node[fill=blue!20] (1221) at (-4.50,-4.00) {1221};
\node[fill=blue!20] (1222) at (-1.50,-4.00) {1222};
\node[fill=blue!20] (1231) at (1.50,-4.00) {1231};
\node[fill=blue!20] (1232) at (4.50,-4.00) {1232};
\node[fill=blue!20] (1242) at (7.50,-4.00) {1242};
\node[fill=blue!20] (1342) at (-1.50,-2.00) {1342};
\node[fill=blue!20] (2342) at (1.50,-2.00) {2342};

\draw[->] (1121) -- (1242);
\draw[->] (1121) -- (1231);
\draw[->] (1121) -- (1232);
\draw[->] (1121) -- (1221);
\draw[->] (0121) -- (1242);
\draw[->] (0121) -- (1232);
\draw[->] (0121) -- (1231);
\draw[->] (0121) -- (1221);
\draw[->] (0120) -- (1231);
\draw[->] (0120) -- (1242);
\draw[->] (0120) -- (1220);
\draw[->] (0110) -- (1231);
\draw[->] (0110) -- (1221);
\draw[->] (0110) -- (1220);
\draw[->] (0110) -- (1232);
\draw[->] (1220) -- (1342);
\draw[->] (1220) -- (2342);
\draw[->] (1111) -- (1232);
\draw[->] (1111) -- (1231);
\draw[->] (1111) -- (1221);
\draw[->] (1111) -- (1222);
\draw[->] (1221) -- (2342);
\draw[->] (1221) -- (1342);
\draw[->] (1231) -- (2342);
\draw[->] (1231) -- (1342);
\draw[->] (1120) -- (1231);
\draw[->] (1120) -- (1242);
\draw[->] (1120) -- (1220);
\draw[->] (1242) -- (1342);
\draw[->] (1242) -- (2342);
\draw[->] (1222) -- (1342);
\draw[->] (1222) -- (2342);
\draw[->] (1110) -- (1231);
\draw[->] (1110) -- (1220);
\draw[->] (1110) -- (1221);
\draw[->] (1110) -- (1232);
\draw[->] (0111) -- (1232);
\draw[->] (0111) -- (1222);
\draw[->] (0111) -- (1231);
\draw[->] (0111) -- (1221);
\draw[->] (0122) -- (1242);
\draw[->] (0122) -- (1232);
\draw[->] (0122) -- (1222);
\draw[->] (0100) -- (1221);
\draw[->] (0100) -- (1220);
\draw[->] (0100) -- (1222);
\draw[->] (1122) -- (1242);
\draw[->] (1122) -- (1232);
\draw[->] (1122) -- (1222);
\draw[->] (1100) -- (1221);
\draw[->] (1100) -- (1220);
\draw[->] (1100) -- (1222);
\draw[->] (1232) -- (1342);
\draw[->] (1232) -- (2342);
\end{tikzpicture}
}

\begin{abstract}
Let $F/\Q_p$ be a finite extension and
let $\bG$ be a connected reductive group
with connected center
satisfying $p>h_{\bG}$.
We write down explicit polynomial equations
for the reduced Emerton-Gee stacks
(the Borel and the twisted Borel versions
that jointly cover the entire $\cX_{\bG,\red}^{\EG}$),
by introducing the
{\it mod $p$ Weil-Deligne stacks} capturing
derived structures of
$(\varphi, \Gamma)$-modules in reduced families.

This allows us to algorithmically
compute the set of irreducible components
of $\cX_{\bG,\red}^{\EG}$,
thereby establishing its equidimensionality,
and the existence of crystalline lifts
by showing the total number of irreducible components
equals the number of (mod $p$) crystalline (or
potentially semistable) components.
The last step of the arguments is to interpolate the
{\it rigid analytic Weil-Deligne stacks}
and the
mod $p$ Weil-Deligne stacks to deduce 
the number of potentially semistable components.

An initial implementation of the algorithms
based on Gr\"obner basis
establishes the existence of crystalline lifts
in the $\bG=\mathrm{F}_4$ case
for all $F/\Q_p$,
and
in the remaining  $\bG=\mathrm{E}_6, \mathrm{E}_7$,
and $\mathrm{E}_8$ cases
for all but finitely many $F/\Q_p$.
\end{abstract}

\maketitle

\tableofcontents

\section{Introduction}

Let $F/\Q_p$ be a finite extension and
let $\bG$ be a connected reductive group
with connected center.
The assumption that $p>h_{\bG}$ is in force
throughout the paper
to enable the Baker-Campbell-Hausdorff formula
and the truncated $\log_{\le h_{\bG}}$
and $\exp_{\le h_{\bG}}$.

We make major progress
towards the open problem of
the existence of crystalline lifts:

\begin{thm} (Theorem \ref{thm:maindup})
\label{thm:main}
If $\bG=\mathrm{F}_4$, then all $\Gal_F\to \bG(\bFp)$
have a crystalline lift.

If $\bG$ is an arbitrary reductive group
and $[F:\Q_p]>\frac{\dim \bG}{2}$, then all $\Gal_F\to \bG(\bFp)$
have a crystalline lift.
\end{thm}

An algorithm
(c.f. \url{github.com/mocham/AlgEG} for the implementation)
is presented
in Section \ref{sec:rotate}
to resolve the remaining small field
cases.
Experiments show that a computationally efficient
cone model suffices for our purposes,
and thus the bottleneck is the
very large number of parallelizable tasks
rather than the computational cost of a single task.
A complete resolution of this project seems
realistic given abundant computing resources.

Our method is based on explicit polynomial equations
for the reduced Emerton-Gee stacks.
The structure of the polynomial equations yields the following
purity result that is worth recording:

\begin{thm} (Corollary \ref{cor:XB-dim})
\label{thm:purity}
Write $\bB$ for the Borel of $\bG$.
If $\dim \cX_{\bB,\red}^{\EG}\le [F:\Q_p]\dim \bG/\bB$,
then
$\cX_{\bB,\red}^{\EG}$
is equidimensional of dimension $[F:\Q_p]\dim \bG/\bB$.
\end{thm}

As opposed to our earlier expectation,
$\cX_{\bB,\red}^{\EG}$
can have 
\MINOREDITED{more irreducible components than}
$\cX_{\bG,\red}^{\EG}$:

\begin{thm} (Proposition \ref{prop:TBM}, Corollary \ref{cor:F4-equidim})
\label{thm:comp}
Let $\bG=\mathrm{F}_4$.
\begin{itemize}
\item 
If $F\neq \Q_p$, then
$\cX_{\bB,\red}^{\EG}\to \cX_{\bG,\red}^{\EG}$
induces a bijection of irreducible components.

\item
If $F= \Q_p$, then
$\cX_{\bB,\red}^{\EG}$
is still equidimensional but
has $4$ extra irreducible components
compared to that of
$\cX_{\bG,\red}^{\EG}$.
\end{itemize}

On the other hand,
if $\bG$ is an arbitrary reductive group
and $[F:\Q_p]>\frac{\dim \bG}{2}$, then 
$\cX_{\bB,\red}^{\EG}\to \cX_{\bG,\red}^{\EG}$
induces a bijection of irreducible components.
\end{thm}

Analysis of the moduli of Weil-Deligne representations
in characteristic $0$
(that dates back to Kisin \cite{Kis08})
shows that
the (reduced) special fiber of a crystalline stack
of regular Hodge type
is necessarily equidimensional of dimension
$[F:\Q_p]\dim \bG/\bB$.
Combining this with
the purity result (c.f. Theorem \ref{thm:purity}),
proving and disproving the existence
of crystalline lifts
both
amount to counting the number
of irreducible components.

The starting point of this paper
is to mimic Kisin's approach
of analyzing the generic fiber
of the potentially semistable deformation rings
using the Weil-Deligne moduli space,
and to study the geometry of the
reduced Emerton-Gee stacks
by constructing a mod $p$
version of Weil-Deligne stacks
that parametrize
pairs $(\bar r, \bar N)$
where $\bar r$ is a semisimple Galois representation
and $\bar N$ is a monodromy operator.
These ideas are implicit in our previous work
\cite{Lin25},
where we geometrize the Serre weights
(a mod $p$ notion)
as irreducible components of the rigid analytic
Borel-valued Weil-Deligne stacks.

Before we explain our methods,
we briefly review the past works.

\subsection{Past works}

\subsubsection{}

\MINOREDITED{\cite{EG23} initiated the use}
of the reduced moduli stack
of \'etale $(\varphi, \Gamma)$-modules
to study the existence of crystalline lifts.
They solved the
$\GL_n$-case by induction on $n$.
Let $\bar\rho:\Gal_F\to \GL_n(\bFp)$
be a mod $p$ Galois representation
and let $\bar\alpha:\Gal_F\to \GL_m(\bFp)$
be an irreducible mod $p$ Galois representation
together with a fixed lift $\alpha$,
they proved the following criterion
via ingenious Galois cohomology techniques:

\begin{quotation}
($*$)
If the locus where
$\dim \Ext^2(\bar\alpha,-)\ge n$
has codimension at least $n$ in
$\cX^{\EG}_{\GL_n,\red}$,
then each extension class
$[\bar c]\in \Ext^1(\bar\alpha,\bar\rho)$
is in the image of 
$\Ext^1(\alpha,\rho)\to \Ext^1(\bar\alpha,\bar\rho)$
for a suitable crystalline lift $\rho$ of $\bar\rho$.
\end{quotation}

\noindent
The existence of crystalline lifts for $\GL_n$
is thus reduced to dimension counting:
cut $\cX^{\EG}_{\GL_n, \red}$
into many strata, and show
the obstruction dimension
is smaller than the codimension
over each stratum.

For general reductive groups,
we can try to adapt the strategy
to an induction on Levi subgroups:
For a parabolic subgroup $P=U^{\bM}\rtimes\bM$
where $U^{\bM}$ is the unipotent radical,
we can try to lift extension classes
$H^1(\Gal_F, U^{\bM}(\bFp))$.
The issue is that $U^{\bM}$ is a non-abelian group
for other reductive groups,
and $H^1(\Gal_F, U^{\bM}(\bFp))$
lacks meaningful algebraic or geometric structures
for a similar analysis.
The even more serious issue
is that this approach
implicitly assumes
that a $P$-valued mod $p$ Galois representation
always admits a $P$-valued crystalline lift.
Even though we don't know of a counterexample yet,
we see
Theorem \ref{thm:comp}
as \MINOREDITED{strong negative evidence for this} ---
a generic $\bFp$-point on the $4$ extra
irreducible components likely does not
have a Borel-valued lift.

\begin{remark}
We did manage to extend the methods
of \cite{EG23} to the case of classical groups (c.f. \cite{Lin25B}).
The key is that $U^{\bM}$ is very close to being
an abelian group
when $\bM$ is a maximal proper Levi.
The dimension counting process is much more complicated
compared to the $\GL_n$-case
as \MINOREDITED{each stratum is} no longer an affine space
but is rather an affine cone.
\end{remark}

\subsubsection{Reflections on \cite{EG23}}
The main application
of the existence of crystalline lifts in their work
is the so-called {\it qualitative Breuil-M\'ezard}
conjecture,
\MINOREDITED{which asserts that}
$\CH_{\Top}(\cX_{n,\red}^{\EG})$
is \MINOREDITED{a} free abelian group generated by the Serre weights.

In their work,
$H^1(\Gal_F, U^{\bM}(\bFp))$
is very structured while
$\CH_{\Top}(\cX_{n,\red}^{\EG})$
seems to be a structure-less object.
This explains why they 
exploit $H^1(\Gal_F, U^{\bM}(\bFp))$
to calculate 
$\CH_{\Top}(\cX_{n,\red}^{\EG})$.

The situation changes for general reductive groups:
$H^1(\Gal_F, U^{\bM}(\bFp))$
is now completely structure-less.
However, 
$\CH_{\Top}(\cX_{\bG,\red}^{\EG})$
is a structured object,
through its connection to the (rigid analytic)
Weil-Deligne stacks.

\subsubsection{}
In his PhD thesis \cite{Mu13},
Muller constructs crystalline lifts for
$\GL_2$ and $\GL_3$
using explicit unramified twists.

\MINOREDITED{His method} relies on strict genericity conditions
on the relative positions of the Galois cohomology classes,
which are not satisfied in full generality for $\GL_4$.
Nevertheless, it
is powerful enough to show {\it sufficiently generic}
Galois representations have a crystalline lift,
which is, surprisingly, 
sufficient from the perspective of
Theorem \ref{thm:purity}
---
if $\cX_{\bB, \red}^{\EG}$
is already known to be equidimensional
and generic points of
$\cX_{\bB, \red}^{\EG}$
have crystalline lifts,
then all points of
$\cX_{\bB, \red}^{\EG}$
have crystalline lifts.
Indeed, in an early draft of this paper,
we didn't know that
an explicit presentation of the reduced Emerton-Gee stacks
would
exist,
and we still managed to prove
Theorem \ref{thm:purity}
using very indirect
methods,
and proceeded
to establish the existence of crystalline lifts for
$\mathrm{F}_4$ using Muller's methods.

\subsection{The strategy}
\newcommand{\lsup}[2]{{^{#1}\!#2}}

\subsubsection{Insights from rigid analytic geometry}

The story begins with our previous work
on the rigid analytic Weil-Deligne stacks.
We recall the following theorem:

\begin{thm} (\cite[Theorem 6]{Lin25})
Let $\rho: \Gal_{F}\to \bM(\bar \Q_p)$ be a
potentially semistable
representation whose Hodge type is $P^-$-dominant.
Endow $\Lie U^{\bM}(\bar\Q_p)$ with $\Gal_F$-action
via $\ad\rho$.

Any $\rho':\Gal_F\to P(\bar\Q_p)$
such that $\rho'\times^P\bM= \rho$
is automatically potentially semistable
of the same inertial type,
with its possible Weil-Deligne
types parameterized by $H^2(\Gal_F, U^{\bM}(\bar\Q_p))$.
\end{thm}

This theorem suggests that
if we write
$\fX_{\bM}=\fX_{\bM}^{\pcrys,\underline{\lambda}, \tau}$
for the rigid analytic moduli stack
of potentially crystalline
Galois representations of Hodge type
$\underline{\lambda}$
and inertial type $\tau$,
and write
$\fW_{P}\to \fX_{\bM}$
for the morphism
relatively representing
the groupoid
\[
(\Sp A \to \fX_{\bM})\mapsto
\{\text{Weil-Deligne representations~}(r_A, N_A)
\text{~whose $\bM$-semisimplification is $D_{\text{pst}}(V_A)$}\}
\]
where $D_{\text{pst}}$ is
Fontaine's functor 
that sends a potentially semistable Galois representation
(or a filtered $(\varphi, N)$-module)
to
a Weil-Deligne representation
and $V_A$ is the universal family over $\Sp A$,
then fibers of
$\fW_{P}\to \fX_{\bM}$
are precisely the $2$-cohomology 
$
H^2(\Gal_F, U^{\bM}).
$

Moreover, if we write
$\fX_{P}=\fX_{P}^{\pst,\underline{\lambda}, \tau}$
for the moduli of potentially semistable
$P$-valued Galois representations,
then there is a canonical commutative diagram
\[
\xymatrix{
\fX_{P}\ar[rr]^{\WD}\ar[rd] & & \fW_P \ar[ld]
\\
& \fX_{\bM}
}
\]
where $\WD$ is the geometrization of the 
$D_{\text{pst}}(V_A)$ functor.
The morphism $\WD$ induces a bijection between irreducible components of $\fX_P$ and $\fW_P$.

\begin{example}
We consider the case where $\bG=\PGL_2$
and $F=\Q_p$.
Then $\bM=\bT=\G_m$.
Suppose $\tau$ is the trivial inertial type
and that $\underline{\lambda}=-1$
(the Barsotti-Tate case).
Then
$\fX_{\bM}=[\fG_m/\fG_m]$
where $\fG_m
=\Sp \Q_p\langle T^{\pm1}\rangle$.

The morphism $\fW_P\to \cX_{\bM}$
is relatively representable by the $\Tor$ presheaf
\[
\Sp A\mapsto
\Tor^{\Q_p\langle T^{\pm1}\rangle}_1
(H^2(\Gal_F, U(\Q_p\langle T^{\pm1}\rangle)), A)
\]
where $H^2$ should really be interpreted as Herr complex cohomology.
We have
\[
H^2(\Gal_F, U(\Q_p\langle T^{\pm1}\rangle))
=\Q_p\langle T^{\pm1}\rangle/(T-1)
\]
and thus
$\fW_P\to\fX_{\bM}$ is relatively representable by 
\[
[\Sp \Q_p\langle T^{\pm1}, Y\rangle/(Y(T-1))/\fG_m]
\to
[\Sp \Q_p\langle T^{\pm1}\rangle/\fG_m].
\]
Note that $\fW_P$ has two irreducible components:
\begin{itemize}
\item (the Steinberg component)
$[\Sp \Q_p\langle Y\rangle/\fG_m]$, and
\item (the non-Steinberg component)
$[\Sp \Q_p\langle T^{\pm1}\rangle/\fG_m]$.
\end{itemize}
On the other hand,
the morphism
\[
\fX_P\to \fW_P
\]
relatively represents the connecting homomorphism
$\delta$
of the long exact sequence:
\[
H^1(\Gal_F, U(\Q_p\langle T^{\pm1}\rangle))
\otimes_{\Q_p\langle T^{\pm1}\rangle}A
\to
H^1(\Gal_F, U(A))
\xrightarrow{\delta} \Tor_1^{\Q_p\langle T^{\pm1}\rangle}(H^2(M,\Q_p\langle T^{\pm1}\rangle),A).
\]
It is not hard to see
$\delta=D_{\pst}$
coincides with \MINOREDITED{Fontaine's} functor.
\end{example}

The example above suggests that
the Weil-Deligne construction
has an alternative interpretation
as inherent derived structures.
This suggests that
a mod $p$ Weil-Deligne functor
is feasible,
at least for Borel-valued
mod $p$ Galois representations.

\begin{notation}
Write
$\cX_{\bG}^{\EG}$ for the Emerton-Gee stacks.

Write
\begin{align*}
\cX_{\bT} & := \cX_{\bT, \red, \bFp}^{\EG}
\end{align*}
for the reduced torus-valued
Emerton-Gee stack.
Also write
\begin{align*}
\cX_{\bB} & := \cX_{\bB,\bFp}^{\EG}\underset{\cX_{\bT,\bFp}^{\EG}}{\times}
\cX_{\bT}
\end{align*}
for the partially reduced Borel-valued Emerton-Gee stack.
\end{notation}

\subsubsection{Gauge cochains}
Once we have obtained a mod $p$ version
of the Weil-Deligne stacks
$\cW_{\bB}$,
we are able to 
exploit the morphism
$\cX_{\bB}\to \cW_{\bB}$
to write down explicit polynomial
presentations of
$\cX_{\bB}$ for the small-rank groups $\PGL_2$
and $\PGL_3$.

It is crucial that we are able to write down
global equations for the small-rank groups,
rather than just local equations
--- we do need to understand how
various strata are glued in a precise manner.

After analysis of the explicit universal families
for $\PGL_2$ and $\PGL_3$,
we are able to recursively define a minimal set
of $1$-cochains in Herr complexes
called the {\it gauge cochains}
for $\PGL_n$
such that
the universal family for $\cX_{\bB_{\PGL_n}}$
can be expressed
as a linear combination of the gauge cochains
where the coefficients
are valued in a multivariable Laurent polynomial ring
$\bFp[t_1^{\pm1}, \dots]$.
Such a free linear combination is not automatically
a valid \'etale $(\varphi, \Gamma)$-module.
Nevertheless, the structure of the gauge cochains
makes it easy to write down the polynomial constraints
imposed on the coefficients.

The purity theorem (c.f. Theorem \ref{thm:purity})
follows from
Krull's Hauptidealsatz,
once the polynomial equations are written down.

These $\PGL_n$-constructions carry over \MINOREDITED{to} general reductive groups
through the Baker-Campbell-Hausdorff formula
and truncated logarithm maps.

\subsubsection{The cone models}
Directly using the polynomial equations to
obtain the irreducible components of $\cX_{\bB}$
is not computationally viable.
We make two optimizations:
\begin{itemize}
\item 
First,
we can now divide $\cX_{\bB}$
into many small strata
and only compute the irreducible components
of these strata ---
thanks to the purity theorem.
\item
Second, over each stratum,
we use the theory of Gr\"obner \MINOREDITED{bases}
to prove that the dimension of the stratum
will not drop after removing the cubic and higher degree terms.
\end{itemize}

\noindent
It is a bit surprising that Gr\"obner basis,
which is mostly used for concrete computations,
can be utilized to prove abstract theorems.

The algebraic varieties (or stacks) defined by the truncated
polynomial equations are called the {\it cone models}.

It turns out that the cone model
of each stratum is
of dimension strictly smaller than
the expected dimension of $\cX_{\bB}$,
unless the cone model is smooth
and
is {\it strictly isomorphic to the stratum itself},
for the cases covered by Theorem \ref{thm:main}.
We are thus able to compute the irreducible components
of $\cX_{\bB}$.

\subsubsection{Unramified descent and the non-Borel valued
mod $p$ Galois representations}

We reduce the non-Borel-valued case
to the Borel-valued case by unramified descent.

Indeed, if we want to write down global equations
for open charts near a non-Borel-valued \MINOREDITED{point},
then tame descent is required.
Fortunately,
we only need equations for (non-open)
strata covering the non-Borel-valued points,
for which the unramified descent suffices.

\subsection{Final thoughts}

Many technical innovations
are left out in the previous subsection,
mainly because they are too specialized
to be placed in this introduction.

During the preparation of this paper,
we feel the notion of ``mod $p$ Weil-Deligne stacks''
should have broader applications
to the study of Galois deformation theory,
integral $p$-adic Hodge theory,
and mod $p$ (derived) Hecke algebra.

Our ad hoc constructions can probably be explained and
vastly generalized by the
theory of prismatic $F$-crystals/gauges.
We hope to come back to this in a follow-up project.

This paper is divided into three parts:

\begin{itemize}
\item {\bf Part I} is the study of the geometry of mod $p$ 
Borel-valued Emerton-Gee stacks,

\item {\bf Part II} is the study of the geometry of mod $p$ 
twisted
Borel-valued Emerton-Gee stacks, and

\item {\bf Part III} is the study of an integral version
of Weil-Deligne stacks and the existence of crystalline lifts.
\end{itemize}

\subsection{\MINOREDITED{Acknowledgements}}
We thank Yu Min and Yupeng Wang for helpful discussions,
especially Yu for explaining derived algebraic geometry
to the author.

\subsection{Use of AI}
The original manuscript and the original 
\url{github.com/mocham/AlgEG} implementation
were written entirely by human.

Later, the author used 
DeepSeek Pro V4 and
ChatGPT 5.6 Sol to fix typos
and refractor the source code;
the refractoring is done mainly by a DeepSeek agent
and the verification/testing is done by a ChatGPT agent;
the appendices are generated by a ChatGPT agent from the refractored
source code.

We have instructed the AI to highlight all major changes colored as green,
and we retain the coloring in the arXiv version.

\newpage
\phantomsection
\addcontentsline{toc}{part}{Part I: 
The mod $p$ Weil-Deligne stacks}

\noindent
{\large Part I: The mod $p$ Weil-Deligne stacks}

\section{Geometrization of the mod $p$ local Euler characteristic}

We have $\Gal(F_{\mu_{p^\infty}}/F)\cong \Delta_F
\times \Gamma$, where $\Delta_F$ is the torsion factor
and $\Gamma\cong \Z_p$.
Set $F_{\cyc} = F(\mu_{p^\infty})^{\Delta_F}$.
We have
$\mathbf{E}_F = k_{F_{\cyc}}(\!(T_F)\!)$,
where $k_{F_{\cyc}}$ is the residue field of $F_{\cyc}$.

\subsection{\'Etale $(\varphi, \Gamma)$-modules with coefficients}
Let $R$ be a commutative ring of finite type over $\bFp$
and let $F/\Q_p$ be a finite extension. Following the construction of Emerton and Gee, we consider the period ring 
\[ \bfE_{F,R} \coloneqq \left( \varprojlim_n (\bfE_F^+/T_F^n) \otimes_{\Fp} R \right) [1/T_F]. \]
An \'etale $(\varphi, \Gamma)$-module over $\bfE_{F,R}$ is defined as a finite projective $\bfE_{F,R}$-module $M$ equipped with a semilinear $\Gamma$-action and a semilinear $\varphi$-action such that the linearized map $\varphi^* M \to M$ is an isomorphism.

\subsection{Herr complexes}
For an \'etale $(\varphi, \Gamma)$-module $M$ over $\bfE_{F,R}$, the Herr complex ${C}_{\Herr}^\bullet(M)$
is $0 \to M \xrightarrow{d^0_\Herr} M \oplus M \xrightarrow{d^1_\Herr} M \to 0$, where the maps are given by:
\begin{align*}
d^0_\Herr(x) &= ((\varphi_M-1)x, (\gamma_M-1)x) \\
d^1_\Herr(x, y) &= (\gamma_M-1)x - (\varphi_M-1)y
\end{align*}

For any algebra $S$ of finite type over $R$, let 
$M_{S} = M \otimes_{\bfE_{F,R}} \bfE_{F,S}$. We define $Z^i$, $B^i$, and $H^i$ as presheaves on the category of finite type $R$-algebras. Emphasizing the dependence on $M$, we write
\begin{align*}
Z^i(M, S) \coloneqq Z^i({C}_{\Herr}^\bullet(M_S))
=:Z^i_{\Herr}(M_S),
\\
\quad B^i(M, S) \coloneqq B^i({C}_{\Herr}^\bullet(M_S))
=:B^i_{\Herr}(M_S),
\\
\quad H^i(M, S) \coloneqq H^i({C}_{\Herr}^\bullet(M_S))
=:H^i_{\Herr}(M_S).
\end{align*}
We note that while $Z^i(M,-)$, $B^i(M,-)$, and $H^i(M,-)$ are generally only 
presheaves of $\cO_{\Spec R}$-modules, with the exception of $H^2(M, -)$ being a coherent sheaf over $\Spec R$.

\begin{remark}
For the technical developments of this paper,
all presheaves are understood as presheaves
over {\it the big affine site}.
We warn the reader that the quasi-coherent sheaves
over the big affine site
over $\Spec R$ do {\it not} form an abelian category,
despite that they do form an abelian category over
the small Zariski site over $\Spec R$.
Kernels, cokernels and images are always
taken over the big affine site.
\end{remark}

We have the following observation:

\begin{lem}
\rm
Suppose that $R$ is a PID,
and that $M$ is an \'etale $(\varphi, \Gamma)$-module
over $\Spec R$.
There is a {\it surjective} homomorphism
of presheaves of $\cO_{\Spec R}$-modules:
\[
H^1(M,S)\to \Tor_1^R(H^2(M,R), S).
\]
\label{lem:PID-Tor}
\end{lem}

\begin{proof}
\MAJOREDIT{By \cite[Theorem 5.1.22]{EG23}, Herr cohomology is represented by a
perfect complex $P^\bullet$ of finite projective
$R$-modules representing Herr cohomology and satisfying, functorially in
$R\to S$,
\[
P^\bullet\otimes_R^{\mathbf L}S\simeq C^\bullet_{\Herr}(M_S).
\]
Since $R$ is a PID, the universal-coefficient spectral sequence degenerates
to the short exact sequence
\[
0\to H^1(P^\bullet)\otimes_R S\to H^1(P^\bullet\otimes_R S)
\to \Tor_1^R(H^2(P^\bullet),S)\to 0.
\]
The base-change identification gives
$H^i(P^\bullet)=H^i(M,R)$ and
$H^i(P^\bullet\otimes_R S)=H^i(M,S)$, and hence the claimed surjection.}
\end{proof}

The presheaf $\Tor^R_1(H^2(M,R),-)$
is often representable by a finite type
affine scheme over $\Spec R$:

\begin{lem}
\rm
Let $X=\Spec R$ be an affine scheme,
and 
let $F$ be a coherent sheaf over $X$
of projective dimension $\le 1$.
Then $\Tor^R_1(F,-)$
is representable by a finite type scheme
over $\Spec R$.
\label{lem:pd-1-representable}
\end{lem}

\begin{proof}
Since $\operatorname{pd}_R(F) \le 1$, we have a projective resolution:
$$ 0 \to P_1 \xrightarrow{\phi} P_0 \to F \to 0. $$

For any $R$-algebra $S$, we tensor this sequence over $R$ with $S$. Since $P_0$ is projective, $\operatorname{Tor}_1^R(P_0, S) = 0$, which yields the exact sequence:
$$ 0 \to \operatorname{Tor}_1^R(F, S) \to P_1 \otimes_R S \xrightarrow{\phi \otimes \operatorname{id}_S} P_0 \otimes_R S \to F \otimes_R S \to 0 $$
From this, we can identify our functor exactly as the kernel of a map of projective modules:
$$ \mathcal{F}(S) = \operatorname{Tor}_1^R(F, S) \cong \ker(P_1 \otimes_R S \to P_0 \otimes_R S). $$

We now show this kernel is representable. For any finitely generated projective $R$-module $P$, the functor $S \mapsto P \otimes_R S$ is represented by the affine group scheme $\mathbb{V}(P^*) = \operatorname{Spec}(\operatorname{Sym}_R(P^*))$, where $P^* = \operatorname{Hom}_R(P, R)$ is the dual module.
Let $W_i = \operatorname{Spec}(\operatorname{Sym}_R(P_i^*))$ for $i = 0, 1$. The $R$-module homomorphism $\phi: P_1 \to P_0$ induces a natural transformation of functors $W_1 \to W_0$.

The presheaf $\mathcal{F}$ is the kernel of the map $W_1 \to W_0$ in the category of abelian group-valued functors. Categorically, the kernel of a morphism to a group object is the fiber product over the zero-section:
$$ \mathcal{F} \cong W_1 \times_{W_0} \{0\} $$
where $\{0\} \cong \operatorname{Spec}(R)$ is the zero-section of the group scheme $W_0$. 
\end{proof}

\begin{example}
Consider $X=\Spec \F_p[t]$,
and let $F$ be the skyscraper sheaf
$R/(t)$ supported as $t=0$.

Then $F$ itself is {\it not} representable
by finite type schemes (or algebraic spaces)
over $X$.
On the other hand, $\Tor_1(F,-)$
is representable by
the union of two axis lines:
$\Spec \F_p[t,s]/(ts)$.
Indeed, it is the fiber at the zero section
of the vector space scheme morphism
\[
\Spec \F_p[t,s] \to \Spec \F_p[t,s],
(t,s)\mapsto \MINOREDITED{(t,ts)}.
\]

Even though the coherent sheaf $F$ is not representable
by algebraic spaces.
The {\it stacky version} of $F$
is indeed representable
by finite type algebraic stacks.
Consider the perfect complex
\[
[C^{-1}\to C^0]:=[R\xrightarrow{t} R],
\]
supported in degrees $-1,0$.
The {\it strict Picard stack}
associated to the perfect complex $[C^{-1}\to C^0]$
is representable by the algebraic stack
\[
[\Spec \F_p[t,s]/\Spec \F_p[s]]
\]
where the action is given by $s\cdot(t, s')=(t,s'+ts)$.
The two stacks $[\Spec \F_p[t,s]/\Spec \F_p[s]]$
and $\Spec \F_p[t,s]/(ts)$
are Koszul dual to each other.
\label{ex:strict-picard}
\end{example}

\begin{lem}
\rm
Suppose $R=\bar\F_p[t,t^{-1}]$
and let $M$ be a rank $1$ \'etale $(\varphi, \Gamma)$-module
over $R$.

If $H^2_\Herr(M)\cong \prod_{i=1}^k R/(t-t_i)$
for $t_i\in \bFp$, then
$\Tor^R_1(H^2(M,R),-)$ is representable by
a scheme obtained by gluing
\[
\Spec \bar\F_p[t,t^{-1},s_i]/(s_i(t-t_i))
\]
over $\Spec \bar\F_p[t,t^{-1}]$.
\label{lem:structure-of-Tor}
\end{lem}

\begin{figure}[H]
\centering
\begin{tikzpicture}[scale=1.5, >=stealth]
  \draw[thick, ->] (-3,0) -- (4.5,0) node[right] {$\operatorname{Spec} \bar{\mathbb{F}}_p[t, t^{-1}]$};

  \fill[white] (0,0) circle (2.5pt);
  \draw[thick] (0,0) circle (2.5pt) node[below=5pt] {$t=0$};

  \draw[thick, blue] (-1.5, -1.5) -- (-1.5, 1.5) node[above] {$t=t_1$};
  \filldraw (-1.5,0) circle (1.5pt) node[below left=1pt] {$t_1$};
  \node[blue, left] at (-1.5, -1.2) {$\operatorname{Spec} \bar{\mathbb{F}}_p[s_1]$};

  \draw[thick, red] (1.5, -1.5) -- (1.5, 1.5) node[above] {$t=t_2$};
  \filldraw (1.5,0) circle (1.5pt) node[below right=1pt] {$t_2$};
  \node[red, right] at (1.5, -1.2) {$\operatorname{Spec} \bar{\mathbb{F}}_p[s_2]$};

  \draw[thick, green!60!black] (2.8, -1.5) -- (2.8, 1.5) node[above] {$t=t_3$};
  \filldraw (2.8,0) circle (1.5pt) node[below right=1pt] {$t_3$};
  \node[green!60!black, right] at (2.8, -1.2) {$\operatorname{Spec} \bar{\mathbb{F}}_p[s_3]$};

\end{tikzpicture}
\caption{Illustration of $\Tor_1(H^2,-)$} 
\end{figure}
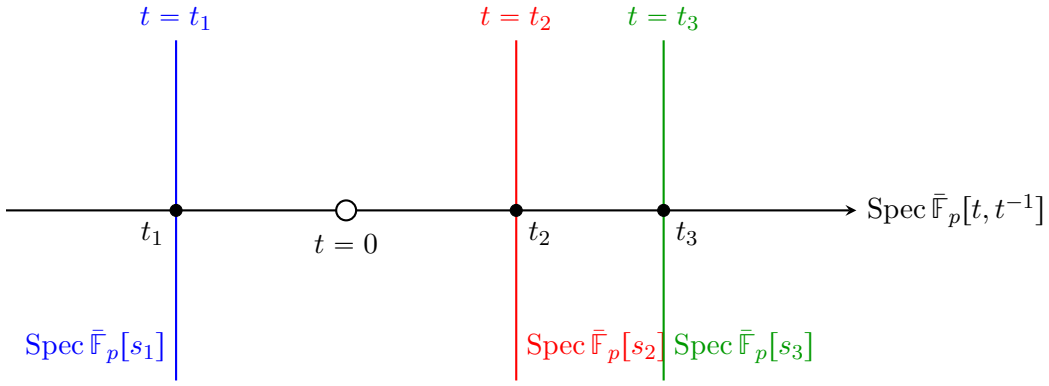

\begin{proof}
This is Zariski local and, after reducing to $k=1$,
follows from
Example~\ref{ex:strict-picard}.
The resulting scheme is precisely the union of the
axis lines depicted above.
\end{proof}

\subsection{Unramified extensions,
and destackification}
\label{subsec:unramified}
By the local Euler characteristics,
for Galois characters,
we have $\dim H^1=\dim H^0+\dim H^2+\dF$.
Lemma \ref{lem:PID-Tor} and Lemma \ref{lem:structure-of-Tor}
can be interpreted as a geometrization
of the intuition that ``$H^1$ contains $H^2$''.
In this subsection, we make precise
the ``$H^1$ contains $H^0$'' counterpart.

Suppose $R=\bFp[t,t^{-1}]$, and let
$M$ be a rank $1$ \'etale $(\varphi, \Gamma)$-module
over $R$ that parametrizes {\it unramified}
Galois characters.
We have $\dim H^1_{\Herr}(\bar\chi_t, \bFp)
=
\begin{cases}
\dF+1 & \bar\chi_t=1 \\
\dF & \bar\chi_t\neq 1,
\end{cases}
$
where $\bar\chi_t$ is
a $\bFp$-fiber of $M$.

From the formula
\[
d^0_\Herr(x) = ((\varphi_M-1)x, (\gamma_M-1)x),
\]
we see that
$B_{\Herr}^1(\bar\chi_t,\bFp)$
contains the constant matrices
$(\bFp,0)$
if and only if $\bar\chi_t\neq 1$.
The upshot is that
the constant matrices $(\bFp,0)\subset Z^1_{\Herr}$
are responsible for
the dimension jump in $H^0$ 
\MINOREDITED{and $H^1$, via the dimension drop in $B^1$,}
when $\bar\chi_t=1$:
\[
\begin{cases}
(\bFp,0) = d^0_\Herr(\bFp), & \bar\chi_t\neq 1 \\
(\bFp,0)\cap B^1_\Herr = 0, & \bar\chi_t= 1.
\end{cases}
\]

\begin{defn}
\label{def:extra-Herr}
\rm
Let $M$ be a rank $1$ \'etale $(\varphi, \Gamma)$-module over
$\bFp[t,t^{-1}]$. We define the modified complex $C^\bullet_{\Herr,+}(M,S)$ by setting:
\begin{align*}
C^0_{\Herr,+}(M,S) &:= C^0_\Herr(M,S), \\
C^1_{\Herr,+}(M,S) &:= (S \otimes_{\F_p} k_{F_{\cyc}}) \oplus C^1_\Herr(M,S), \\
C^2_{\Herr,+}(M,S) &:= C^2_\Herr(M,S).
\end{align*}
The differentials are defined as follows. For degree $0$, we set:
\[
d_{\Herr,+}^0(f) :=
(\text{constant term of } f, d^0_\Herr(f))
\]
where the constant term is evaluated in $S \otimes_{\F_p} k_{F_{\cyc}}$.
For degree $1$, the differential is given by $d^1_{\Herr,+}(s, x) := d^1_\Herr(x)$.

Here, the ``constant term of $f$''
depends on the choice of a basis element
$v\in M_S^{\Gamma}\cong S\otimes_{\MINOREDITED{\F_p}}k_{F_\cyc}$.
\end{defn}

\begin{lem}
\rm
Let $M$ be any rank $1$ \'etale $(\varphi, \Gamma)$-module
over $R=\bFp[t^{\pm1}]$ (not necessarily unramified).
We have
$M^{\Gamma}=\bFp[t^{\pm1}]\otimes_{\F_p}k_{F_{\cyc}}$.
\end{lem}

\begin{proof}
Let $F'/F$ be a tame extension with ramification index
dividing $(p^{k_F}-1)$.
After base change of $\bfE_{F',R}$,
$M$ is unramified and thus
$M^{\Gamma}$ is a quasi-coherent sheaf
whose $\bFp$-fibers are all isomorphic
to $\bFp\otimes_{\F_p}k_{F'_{\cyc}}$.
So, $M^{\Gamma}\cong R\otimes_{\F_p}k_{F'_{\cyc}}$ is finite free by Swan's theorem.
The lemma follows from tame descent.
\end{proof}

\begin{lem}
\label{lem:Herr+}
\rm
Then $[C^\bullet_{\Herr,+}(M,S)]$ is a perfect complex
with $H^0_{\Herr,+}(M,S)=0$ and $H^2_{\Herr,+}(M,S)=H^2_\Herr(M,S)$.

Moreover, $\ker(H^1_{\Herr,+}(M,S)\to \Tor^R_1(H^2_\Herr(M,R),S))$
is a finite free coherent sheaf of constant rank 
$(\dF+\dkF)$.
\end{lem}

\begin{proof}
The perfectness of $C_{\Herr,+}^\bullet$
follows from 
the perfectness of $C_{\Herr}^\bullet$.
It follows from the construction that
$H^0_{\Herr,+}(M,S)=0$
and $H^2_{\Herr,+}(M,S)=H^2_\Herr(M,S)$.
The rest of the lemma follows from local Euler characteristic.
\end{proof}

The stackiness of the Borel-valued Emerton-Gee stacks originates from the $H^0$ terms. By using the modified complex $[C_{\Herr,+}^\bullet]$, we can uniformly eliminate this stackiness.

\section{Destackification and coordinates}

\label{sec:destack}
Let $(G,B,T,\{X_{\alpha}\}_{\alpha\in \Delta})$ be a split
pinned reductive group over $F$,
and let $(\bG, \bB, \bT, \{Y_{\alpha}\}_{\alpha\in \Delta})$ be the dual pinned group of $G$.
Assume $p> h_{\bG}$, where
$h_{\bG}$ is the Coxeter number of $\bG$.
In particular,
the truncated log induces an isomorphism
\[
\log:=\log_{\le h_{\bG}}: U\xrightarrow{\cong} \Lie(U)
\]
of schemes over $\bFp$.

Let $U\subset \bB$ be the unipotent radical
and consider the {\it  descending central series}
\[
\Lie U_k = 
\begin{cases}
\Lie U & k=1 \\
[\Lie U,\Lie U_{k-1}] & k>1.
\end{cases}
\]
\MINOREDITED{and let $U_k\subset U$ be the corresponding connected subgroup.}
Set $\bar U_k:= U_{k}/U_{k+1}$.
Because $p$ is a good prime for $\bG$,
the descending central filtration coincides with the height filtration, and we have
\[
\bar U_k= \prod_{\alpha\in \Phi(B,T), \Ht(\alpha)=k} U_{\alpha^{\!\!\vee}},
\]
where $U_{\alpha^{\!\!\vee}} \subset \bG$
is the root group of root $\alpha^{\!\!\vee}
\in 
\MINOREDITED{\Phi(\bB,\bT)}$.
We remind the reader that $\Ht(\sum_{\alpha\in \Delta}n_\alpha \alpha)
:=\sum n_\alpha\in \Z$,
and that any positive root can be written as a sum
$\alpha_1+\dots+\alpha_s$ with each $\alpha_i\in \Delta$
such that any partial sum 
\MINOREDITED{$\alpha_1+\dots+\alpha_k$}
is still a positive root in $\Phi(B,T)$
(c.f. \cite[Lemma 10.2.B]{Hum72}).
A consequence of the partial sum property is that
\begin{equation}
\label{eq:PSP}
[\Lie \bar U_1,\Lie \bar U_k] = \Lie\bar U_{k+1}.
\end{equation}

\begin{defn}
The $k$-th truncated Borel-valued Emerton-Gee
stack for $G$ is, by definition,
\[
\tr_k\cX_{\bB}:=\cX_{U/U_{k+1} \rtimes \bT}
:=
\cX_{U/U_{k+1} \rtimes \bT}^{\EG}
\times_{\cX_{\bT}^{\EG}} \cX_{\bT,\red}^{\EG}
.
\]
\end{defn}

\subsection{Destackification}
\label{subsec:Borel-destack}
For many technical applications, working with schemes or algebraic spaces is preferable to working with algebraic stacks. 
To facilitate this, we recursively construct a destackification 
$X_{\bB}\to \cX_{\bB}$. 

First, we fix a total ordering ``$\le$'' of all positive roots $\Phi(\bB, \bT)^+$ that is compatible with heights in the sense that
\[
\alpha\le\beta \Rightarrow \Ht(\alpha)\le \Ht(\beta).
\footnote{We apologize for the non-standard notation
as $\alpha< \beta$ usually means
$(\beta-\alpha)$ is dominant in the literature.}
\]
For each $\alpha \in \Phi(\bB, \bT)^+$ of height $h$, the notation below is
understood in the quotient $U/U_{h+1}$. More precisely, we define the
successive subquotients
\[
U_{\le \alpha}:=
\frac{U/U_{h+1}}
{\prod_{\substack{\beta>\alpha\\ \Ht(\beta)=h}}U_{\beta^\vee}},
\qquad
U_{<\alpha}:=
\frac{U/U_{h+1}}
{\prod_{\substack{\beta\ge\alpha\\ \Ht(\beta)=h}}U_{\beta^\vee}}.
\]
\MAJOREDIT{The height-$h$ root groups are central in $U/U_{h+1}$; hence these
quotients are well-defined and fit into an exact sequence
\[
1\to U_{\alpha^\vee}\to U_{\le\alpha}\to U_{<\alpha}\to1.
\]
Thus all root products appearing below are regarded as the corresponding
subquotients.}
For ease of notation, set
$\cX_{\le \alpha} := \cX_{U_{\le \alpha}\rtimes \bT}
:=
\cX_{U_{\le \alpha}\rtimes \bT}^{\EG}
\times_{\cX_{\bT}^{\EG}}\cX_{\bT,\red}
.$

Our method of destackification is a natural geometric extension of Definition \ref{def:extra-Herr}. For each $\alpha \in \Phi(\bB, \bT)^+$, we construct a $\G_a$-torsor 
\[
\cX^+_{\le \alpha} \to \cX_{\le \alpha}
\]
as follows: Consider a morphism $\Spec A \to \cX_{\le \alpha}$ corresponding to an \'etale $(\varphi, \Gamma)$-structure on
a $(U_{\le\alpha}\rtimes \bT)$-torsor $P_A'$ over
$\Spec \bfE_{F,A}$, and set
$P_A:=P_A'\times^{U_{\le\alpha}\rtimes\bT}(U_{<\alpha}\rtimes\bT)$.
The objects of the fiber product $\cX^+_{\le \alpha}\times_{\cX_{\le \alpha}} \Spec A$ are pairs $(P_A', \xi)$, where:
\begin{itemize}
    \item $P_A'$ is a $(U_{\le\alpha}\rtimes \bT)$-torsor over $\Spec \bfE_{F,A}$ equipped with an identification $P_A'\times^{U_{\le \alpha} \rtimes \bT} (U_{<\alpha}\rtimes \bT) = P_A$.
    \item $\xi: P_A \to P_A'$ is a $\Gamma$-stable section of the natural projection $P_A' \to P_A$ .
    Here, a $\Gamma$-stable section is the geometrization
    of the ``constant term coefficient''
    in Definition \ref{def:extra-Herr}.
\end{itemize}

\noindent
{\bf Relation to the modified Herr complex.}
For a rank-$1$ module $M$ over $\bFp[t^{\pm1}]$
that is the universal unramified twist of a Galois character,
an element $x\in M$ expands as $x=\sum_{n\in\Z}a_nT_F^n$,
after the choice of a $\Gamma$-stable generator.
The constant term $\mathrm{ct}(x)=a_0$ is the coefficient
of $T_F^0$. A section $\xi$ as above corresponds to
choosing a $\Gamma$-stable element whose constant term
is prescribed. 

\begin{lem}
We have $[\cX_{\le \alpha}^+/\G_a^{\oplus [k_{F_{\cyc}}:\F_p]}]
\cong \cX_{\le \alpha}$.
\end{lem}

\begin{proof}
By construction $\cX_{\le \MINOREDITED{\alpha}}^+$
is a pseudo $\G_a^{\oplus [k_{F_{\cyc}}:\F_p]}$-torsor over
$\cX_{\le \MINOREDITED{\alpha}}$.
We still need to show that it is a genuine $\G_a^{\oplus [k_{F_{\cyc}}:\F_p]}$-torsor.
This is equivalent to the statement that
all $P_A'\to P_A$
admits a $\Gamma$-stable section fppf locally
in $\Spec A$.
Since $P_A$ is affine
and $\Res_{\bfE_{F,A}/A}\G_a$ is quasi-coherent, the coherent cohomology
$H^1(P_A, \Res_{\bfE_{F,A}/A}\G_a)=0$,
and $P_A'$ is indeed a trivial $\Res_{\bfE_{F,A}/A}\G_a$-torsor
over $P_A$.

Fix a section $\xi_0$,
and consider the set of all sections $\{\xi\}$.
The set $\{\xi-\xi_0\}$ of all sections is an
\'etale $(\varphi,\Gamma)$-module $M_0$ over $\bfE_{F,A}$,
and reducing to the universal case $A=\bFp[t^{\pm1}]$,
$M_0$ is the unramified twist of a rank-$1$
$(\varphi,\Gamma)$-module, which always possesses a
$\Gamma$-invariant generator.
\end{proof}

\begin{defn}
We recursively define
destackifications
$X_{\le \alpha}$ for $\alpha\in \Phi^+$.

Since $\cX_{\bT}=\coprod[\bT/\bT]$,
set
$X_{\MINOREDITED{<} \alpha_{\min}}=X_{\bT}:=\coprod \bT$,
where the disconnected union is indexed over
the set of Serre weights for $\GL_1(\kappa_F)$.

Assuming we have already constructed $X_{<\alpha}$, we define the next stage as:
\[
X_{\le \alpha} := X_{< \alpha} \times_{\cX_{< \alpha}} \cX_{\le \alpha}^+.
\]
Finally, setting $X_{\bB} := X_{\le \alpha_{\max}}$.
\end{defn}

\begin{prop}
$X_{\bB}\to \cX_{\bB}$ is smooth, surjective,
affine, and \MINOREDITED{all its fibers}
are isomorphic to $\Res_{k_{F_{\cyc}}/\F_p}U\rtimes \bT$.

Moreover, the formation of $X_{U_{\le \alpha}\rtimes \bT}$
tautological and does not depend
on the choice of the total ordering on $\Phi(\bB, \bT)$.
\label{prop:Borel-destack}
\end{prop}

\begin{proof}
The fiber description follows inductively from the
recursive construction: at each step we add a
$\G_a^{\oplus [k_{F_{\cyc}}:\F_p]}$-torsor for the current
root, and root groups at the same height commute,
so the product is a torsor for the direct sum,
giving the claimed fiber $\Res_{k_{F_{\cyc}}/\F_p}U\rtimes\bT$.
The ``moreover'' part follows because if $\alpha_1$ and $\alpha_2$
have the same height, then $U_{\alpha_1^\vee}$ commutes
with $U_{\alpha_2^\vee}$.
\end{proof}

\subsection{$\bT$-stable normal subgroups}

If \MINOREDITED{an} algebraic subgroup $K$ of $U$
is stable under the $\bT$-adjoint action,
then $K=\prod_{\alpha\in \Phi_K\subset \Phi^+}U_{\alpha^\vee}$
is a product of root groups.
The group $K$
is normal if and only if 
\[
\alpha\in\Phi_K,\ \beta\in\Phi^+,\ \alpha+\beta\in \Phi^+
\Rightarrow\alpha+\beta\in \Phi_K.
\]

\begin{defn}
Consider the descending central series
$\Lie K_1:=\Lie K$,
$\Lie K_h:=[\Lie K, \Lie K_{h-1}]$.
Set $\bar K_h:=K_h/K_{h+1}$.
Define the $K$-height so that
$\Ht_K(\alpha)=h
\Leftrightarrow
U_{\alpha^{\vee}}\subset \bar K_h$:
\[
\bar K_h = \MINOREDITED{\prod}_{\Ht_K(\alpha)=h} U_{\alpha^\vee}.
\]
Write 
\[
\Delta_K := \{\alpha\in \Phi_K|\Ht_K(\alpha)=1\}
\]
and call them the {\it $K$-simple roots}.
Note that $\Ht_K$ \MINOREDITED{need not be} additive,
unless $K$ is the unipotent radical of a parabolic.
\end{defn}

\subsection{Coordinates}

\label{subsec:nonab}

\begin{defn}
If $S\subset \Phi^+$ is a subset,
write
\[
U_{S}:=\prod_{\alpha\in S}U_{\alpha^\vee},
\]
for some ordering on $S$,
and is treated as a subquotient of $U$ when possible.

Fix a $\bT$-stable normal subgroup $K=U_{\Phi_K}\subset U$.
We fix a total ordering ``$\leK$'' on $\Phi_K$
compatible with the $K$-height.
For $\alpha, \beta\in \Phi_K$,
we write
\[
\alpha \underset{K}{>} \beta
\]
to indicate $\alpha$ is greater
than $\beta$ in the fixed total ordering.
Write
\begin{align*}
\Phi_{K, \le \alpha} &:=\{\beta\in\Phi_{K}:\beta \leK \alpha\},
\\
K_{\le \alpha}&:=U_{\Phi_{K, \le \alpha}},
\\
\tr_{h}K &:= U_{\Phi_{K, \Ht_K\le h}}.
\end{align*}
\end{defn}

\begin{defn}
We fix a standard Chevalley basis 
$\{e_\eta\}_{\eta \in \Phi}$ for $\mathrm{Lie}(\bG)$.
The Lie bracket satisfies 
\begin{equation}
\label{eq:Nab}
[e_\beta, e_\delta] = N_{\beta, \delta} e_\alpha
\end{equation}
whenever $\beta + \delta = \alpha \in \Phi$, where $N_{\beta, \delta} \in \mathbb{Z}$ are the structure constants of the root system.
\end{defn}

\begin{defn}
\label{defn:coordinate}
Let $A$ be an $\bFp$-algebra.
Fix an \'etale $(\varphi, \Gamma)$-module
\[
f_{< \alpha}, g_{< \alpha} \in (K_{< \alpha}\rtimes\bT)
(\bfE_{F, A})
\]
satisfying
\[
f_{< \alpha} \varphi(g_{< \alpha})
=g_{< \alpha} \gamma(f_{< \alpha}).
\]
Write
\begin{align*}
f_{ <\alpha} = c_{f,< \alpha} f^\SS, \qquad
g_{< \alpha} = c_{g,< \alpha} g^\SS
\end{align*}
where
\begin{align*}
c_{f, < \alpha}, c_{g, < \alpha}
\in K_{< \alpha}(\bfE_{F, A}),\qquad
f^{\SS}, g^\SS
\in \bT(\bfE_{F, A}).
\end{align*}
Write
\begin{align*}
\log_{\le h_{\bG}}(c_{f, < \alpha})
=
\sum_{\beta\lK \alpha} c_{f, \beta},\qquad
\log_{\le h_{\bG}}(c_{g, < \alpha})
=
\sum_{\beta\lK \alpha} c_{g, \beta}
\end{align*}
where
$c_{f, \beta}, c_{g,\beta}\in U_{\beta^\vee}(\bfE_{F, A})$.
If $\alpha$ is the smallest root of height $(h+1)$,
we also write
\[
(\tr_h c_f, \tr_h c_g)
=(c_{f, < \alpha}, c_{g, <\alpha}).
\]
\end{defn}

\begin{thm}
\rm
Choose elements
$c_{f, \alpha}, c_{g, \alpha}\in U_{\alpha^\vee}(\bfE_{F, A})$.
Set
\[
c_{f,\le \alpha} := \exp_{\le h_{\bG}}(c_{f,<\alpha}+c_{f,\alpha})
,\qquad
c_{g,\le \alpha} := \exp_{\le h_{\bG}}(c_{g,<\alpha}+c_{g,\alpha})
\]
Then
\[
(f_{\le \alpha}, g_{\le \alpha})
:=(c_{f,\le \alpha}f^\SS, c_{g,\le\alpha}g^\SS)
\]
is an \'etale $(\varphi, \Gamma)$-module
if and only if
\[
(c_{f, \alpha} - g^\SS \gamma(c_{f, \alpha}) g^{\SS-1})
-
(c_{g, \alpha} - f^\SS \varphi(c_{g, \alpha}) f^{\SS-1})
=
D_\alpha(c_{f, <\alpha}, c_{g, <\alpha}),
\]
where
\begin{align}
\label{eq:nonab-herr}
D_\alpha(c_{f, <\alpha}, c_{g, <\alpha})
=& 
\sum_{n=1}^{h_{\bG}}
\sum_{\substack{\beta_1,\ldots,\beta_n\in\Phi_K\\
\beta_i<\alpha,\ \beta_1+\cdots+\beta_n=\alpha}}
\frac{1}{n!}[c_{g, \beta_1}, [c_{g, \beta_2},
\dots[c_{g, \beta_{n-1}}, g^\SS\gamma(c_{f,\beta_n})g^{\SS-1}]
]]
\\
&
- 
\sum_{n=1}^{h_{\bG}}
\sum_{\substack{\beta_1,\ldots,\beta_n\in\Phi_K\\
\beta_i<\alpha,\ \beta_1+\cdots+\beta_n=\alpha}}
\frac{1}{n!}[c_{f, \beta_1}, [c_{f, \beta_2},
\dots[c_{f, \beta_{n-1}}, f^\SS\varphi(c_{g,\beta_n})f^{\SS-1}]]]
\nonumber
\end{align}
where the sums range over ordered tuples of positive roots, repetitions are
allowed, and for $n=1$ the nested bracket denotes its innermost term.
\label{thm:coordinate}
\end{thm}

\begin{proof}
\MAJOREDIT{Note that $U_{\alpha^\vee}$ is contained in the center of
$K_{\le\alpha}$. Write the cocycle equation in logarithmic coordinates and
move the two terms involving $c_{f,\alpha}$ and $c_{g,\alpha}$ to the left.
For the remaining terms, use
\[
\operatorname{Ad}(\exp x)(y)=\exp(\ad x)(y)
=\sum_{m\ge0}\frac{1}{m!}(\ad x)^m(y),
\]
which is often called the Hadamard's Lemma.
The root grading implies that the projection to $U_{\alpha^\vee}$ receives a
contribution precisely from ordered tuples of positive roots whose sum is
$\alpha$; all longer terms vanish because every root occurring here is
positive and has height at most $h_{\bG}$. Expanding the two conjugations gives
the two nested-commutator sums in Equation~\ref{eq:nonab-herr}, with opposite
signs. The remaining linear term is
$-d^1_{\Herr}(c_{f,\alpha},c_{g,\alpha})$, which proves the equivalence.}
\end{proof}

\begin{remark}
For ease of notation,
write $\gamma_g(u):= g^\SS\gamma(u) g^{\SS-1}$
and $\varphi_f(u):= f^\SS\varphi(u) f^{\SS-1}$.

Here are some basic observations:
\[
(c_{f, \alpha} - \gamma_g(c_{f, \alpha}))
-
(c_{g, \alpha} - \varphi_f(c_{g, \alpha}))
=-d^1_\Herr(c_{f, \alpha}, c_{g, \alpha})
\]
is the usual differential in Herr complex;
the usual cup product in Herr complexes is given by
\[
(\frac{c_{f, \beta_1}}{e_{\beta_1}},
\frac{c_{g, \beta_1}}{e_{\beta_1}})
\cup
(\frac{c_{f, \beta_2}}{e_{\beta_2}}, \frac{c_{g, \beta_2}}{e_{\beta_2}})
=
\frac{c_{f, \beta_1}}{e_{\beta_1}}\varphi_f(\frac{c_{g, \beta_2}}{e_{\beta_2}})
-
\frac{c_{g, \beta_1}}{e_{\beta_1}}\gamma_g(\frac{c_{f, \beta_2}}{e_{\beta_2}})
\]
and thus
\[
[c_{g, \beta_1}, \gamma_g(c_{f, \beta_2})]
-
[c_{f, \beta_1}, \varphi_f(c_{g, \beta_2})]
=
-N_{\beta_1, \beta_2} 
(\frac{c_{f, \beta_1}}{e_{\beta_1}}, \frac{c_{g, \beta_1}}{e_{\beta_1}})
\cup
(\frac{c_{f, \beta_2}}{e_{\beta_2}}, \frac{c_{g, \beta_2}}{e_{\beta_2}})e_{\beta_1+\beta_2}
\]
where $N_{\beta_1, \beta_2}$
is the structural constant.
\end{remark}

\begin{defn}
Write
\[
c_{\beta}=(c_{f, \beta}, c_{g, \beta})
\in C^1_{\Herr}(U_{\beta^\vee}(\bfE_{F})),
\]
and
\[
c_{\beta_1}\cup c_{\beta_2}
:=
(\frac{c_{f, \beta_1}}{e_{\beta_1}}, \frac{c_{g, \beta_1}}{e_{\beta_1}})
\cup
(\frac{c_{f, \beta_2}}{e_{\beta_2}}, \frac{c_{g, \beta_2}}{e_{\beta_2}})e_{\beta_1+\beta_2}.
\]
We also define the {\it higher cup product}
\begin{align}
[c_{\beta_1}, c_{\beta_2},\dots, c_{\beta_n}]
:=&
[c_{g, \beta_1}, [c_{g, \beta_2},
\dots[c_{g, \beta_{n-1}}, g^\SS\gamma(c_{f,\beta_n})g^{\SS-1}]
]]
- 
[c_{f, \beta_1}, [c_{f, \beta_2},
\dots[c_{f, \beta_{n-1}}, f^\SS\varphi(c_{g,\beta_n})f^{\SS-1}]]]
\\
&
\in C^2_\Herr([U_{\beta_1^\vee},[U_{\beta_2^\vee}, \dots,
[U_{\beta_{n-1}^\vee}, U_{\beta_n^\vee}]]](\bfE_{F, A})).\nonumber
\end{align}
So,
$
c_{\beta_1}\cup c_{\beta_2}
=-\frac{1}{N_{\beta_1, \beta_2}}
[c_{\beta_1}, c_{\beta_2}].
$
\label{defn:higher-cup}
\end{defn}

\begin{cor}
If $(\tr_h c_f, \tr_h c_g)$
is an \'etale $(\varphi, \Gamma)$-module
valued in $\tr_h K\rtimes \bT$,
then
it can be extended to $\tr_{h+1}K\rtimes \bT$
if and only if
\[
\frac{1}{2}\sum_{\delta\in \Delta_K, \delta+\beta=\alpha,
\Ht_K(\beta)=h, \Ht_K(\alpha)=h+1}
N_{\delta, \beta} c_\delta\cup c_\beta
\in \sum_{\Ht_K(\alpha)=h+1}D_\alpha(\tr_{h-1}c_f, \tr_{h-1}c_g) + B^2_{\Herr}(\bar K_{h+1}(\bfE_{F, A})).
\]
\label{cor:coord-h}
\end{cor}

\begin{proof}
Clear from the definitions.
\end{proof}

\begin{cor}
Suppose $\Ht_K(\alpha)=h+1$.
If $(c_{f, <\alpha}, c_{g, <\alpha})$
is an \'etale $(\varphi, \Gamma)$-module
valued in $K_{<\alpha}\rtimes \bT$,
then
it can be extended to $K_{\le \alpha}\rtimes \bT$
if and only if
\[
\frac{1}{2}\sum_{\substack{\delta\in \Delta_K,\ \beta\in \Phi_{K, \Ht_K=h}\\
\delta+\beta=\alpha}}
N_{\delta, \beta} c_\delta\cup c_\beta
\in D_\alpha(\tr_{h-1}c_f, \tr_{h-1}c_g) + B^2_{\Herr}(U_{\alpha^\vee}(\bfE_{F, A})).
\]
\label{cor:nonab-herr-2}
\end{cor}

\begin{proof}
Clear from the definitions.
\end{proof}

\section{Example: $\PGL_2$}

\subsection{The Weil-Deligne stack
and the equations}

The Borel subgroup $B=B_{\PGL_2}$ of $\PGL_2$
is isomorphic to $\G_a\rtimes \G_m
=\{
\begin{bmatrix}
* & * \\
0 & 1
\end{bmatrix}
\}
\subset \GL_2$.

\begin{lem}
\rm
The morphism $\cX_B\to \cX_{\G_m}$
is relatively representable by the groupoid
$[Z_\Herr^1/C_\Herr^0]$.
Here, $[C_\Herr^0\to C_\Herr^1\to C_\Herr^2]$ 
is the Herr complex attached to the universal family over
$\cX_{\G_m}$.

\end{lem}

\begin{proof}
Unravel the definitions.
\end{proof}

We remark that $\cX_B\to \cX_{\G_m}$
is {\it not} relatively representable by a strict Picard stack
as $Z_\Herr^1$ is not a coherent sheaf.

\begin{defn}
\label{def:WD-PGL2}
Denote by $\cW_B$
for the algebraic stack over $\cX_{\G_m}$
that relatively represents the presheaf 
$\Tor_1^{\cO_{X_{\G_m}}}(H^2_\Herr(\bfE_{F, \cO_{X_{\G_m}}}), R)$
for each morphism $\Spec R\to\cX_{\G_m}$.
Set $W_B:=\cW_B\times_{\cX_{\G_m}}X_{\G_m}$.
\end{defn}

\begin{defn}
\rm
Let $C\subset \cW_B$ be an irreducible component.
We say $C$ is a {\it Steinberg} component
if its image $\bar C$ in $\cX_{\G_m}$
is not an irreducible component.
Set 
\begin{align*}
\cX_{B}(C)&:=C\times_{\cW_B}\cX_B,
\\
X_{B}(C)&:=C\times_{\cW_B}X_B.
\end{align*}
\end{defn}

\begin{remark}
Equivalently, $C$ is Steinberg if
$C\cong [\Spec \bFp[s]/\G_m]$
and $C$ is non-Steinberg if
$C\cong [\Spec \bFp[t^{\pm1}]/\G_m]$.
\end{remark}

\begin{prop}
$\cW_B$ is equidimensional of dimension $0$,
and
there is a canonical morphism $\cX_B\to\cW_B$.
For each irreducible component $C\subset \cW_B$,
$C\times_{\cW_B}X_B\to C$
is a genuine $\G_a^{\dF+\dkF}$-torsor.
\label{prop:PGL2}
\end{prop}

\begin{proof}
Note that $\cX_{\G_m}$
is a disjoint union of
$[\G_m/\G_m]$
where $\G_m$ acts trivially on $\G_m$, indexed by Serre weights for $\GL_1$.
Irreducible components
of $X_{\G_m}$
are isomorphic to $\G_m$
and the universal family over each $\G_m$
is the family of unramified twists of a Galois character
$\Gal_F\to \bFp^\times$.

As a consequence, $W_B$ is equidimensional of dimension $1$ by
Lemma \ref{lem:structure-of-Tor}.
So, $\cW_B=[W_B/\G_m]$ is equidimensional of dimension $0$.
The morphism $X_B\to W_B$
is defined and shown to be surjective
by Lemma \ref{lem:PID-Tor},
and clearly descends to $\cX_B\to \cW_B$.

Let $C\subset \cW_B$ be an irreducible component.
By Lemma \ref{lem:Herr+} and Lemma \ref{lem:PID-Tor},
$X_B(C)\to C$
is surjective and is a pseudo $\G_a^{\dF+\dkF}$-torsor:
we can define a group action
$\G_a^{\dF+\dkF}\times_C X_B(C)
\cong X_B(C)\times_C X_B(C)$.
In general, a pseudo $\G_a$-torsor does not have to be a \MINOREDITED{genuine}
$\G_a$-torsor.
If we can construct a section
$C\to X_B(C)$ fpqc locally, then
$X_B(C)\to C$
is a \MINOREDITED{genuine} $\G_a^{\dF+\dkF}$-torsor.
If $C\subset \cW_B$ is not a Steinberg component,
then there exists the trivial section
$C\to X_B(C)$ that corresponds to the zero extension classes
$
\begin{bmatrix}
* & 0 \\ & *
\end{bmatrix}
$.

It remains to consider the Steinberg component 
$[\Spec \bFp[s]/\G_m]\cong C_{\st}\subset \cW_B$.
We complete the proof by explicitly constructing
a local section in flat topology.
$C_{\st}$ maps to the cyclotomic character point 
in $\cX_{\G_m}$;
write $M_{\cyc}$ for the \'etale $(\varphi, \Gamma)$-module
that corresponds to the mod $p$ cyclotomic Galois character.
Let $(c_f, c_g)\in Z^1(M_{\cyc}, \bFp)$
be a tr\'es ramifi\'e cocycle.
So, we do get a section
$\Spec \bFp[s] \to X_B(C_{\st})$
corresponding to the $1$-cocycle
$(s~c_f, s~c_g) \in Z^1(M_{\cyc}\otimes_{\bFp}\bFp[s], \bFp[s])$.
We don't know if
the composition
$[\Spec \bar\F_p[s]/\G_m] \to
X_B(C_{\st})
\to C_{\st}
\cong [\Spec \bar\F_p[s]/\G_m]$
is an isomorphism.
Nevertheless, it is clear that
it induces
a surjective group scheme homomorphism
$\Spec \bar\F_p[s]\to \Spec \bar\F_p[s]$.
Group scheme homomorphisms
$\Spec \bar\F_p[s]\to \Spec \bar\F_p[s]$
are classified by polynomials
$s\mapsto \sum_{i=0}^N a_i s^{p^i}$
for $a_0,\dots, a_N\in \bFp$,
and they are all finite flat morphisms
(except for the zero morphism).
Thus $X_B(C_{\st})\to C_{\st}$ is indeed an fppf $\G_a^{\dF+\dkF}$
-torsor, and we are done.
\end{proof}

\begin{thm}
\rm
We have
\[
X_B\cong W_B\times \A^{\dF+\dkF},
\]
and thus
the morphism
\[
\WD:\cX_{B}\to \cW_{B}
\]
is smooth of relative dimension $\dF$,
and induces a bijection of irreducible components.
\label{thm:PGL2-eqn}
\end{thm}

\begin{proof}
Note that all irreducible components of $\cW_{B}$
that do not intersect with the Steinberg component
are connected components.
So, by Proposition \ref{prop:PGL2}, it remains to
restrict to the connected component containing the Steinberg
component, which is isomorphic to
$[\Spec \bFp[t, t^{-1}, s]/(s(t-1))/\G_m]$
and is the union of two irreducible components
$[\Spec \bFp[t,t^{-1}]/\G_m]$
and
$[\Spec \bFp[s]/\G_m]$.
The morphism
\[
X_{B}\times_{W_B}\Spec \bFp[t,t^{-1}]
\to \Spec \bFp[t,t^{-1}]
\]
is relatively representable by a vector bundle of
rank $(\dF+\dkF)$,
and is thus a relatively representable by
a trivial vector bundle by Swan's theorem
(that any finite projective module over $\G_m$ is finite free).
So, \[
X_{B}\times_{W_B}\Spec \bFp[t,t^{-1}]
\cong \Spec \bFp[t,t^{-1}]\times \A^{\dF+\dkF}
=:\cY,
\]
which induces a natural morphism
\[
\cY\times_{\Spec \bFp[t,t^{-1}]}
\Spec \bFp[t, t^{-1}, s]/(s(t-1))
\to 
X_{B}\times_{W_B}\Spec \bFp[t, t^{-1}, s]/(s(t-1))
\]
which is an isomorphism by Lemma~\ref{lem:PGL2-local} below.
\end{proof}

Indeed, we have explicitly written down
coordinate rings for $X_B$.
For each irreducible component
$C=[\Spec \cO_C/\G_m]\subset W_B$,
Swan's theorem implies
$X_B(C)\cong \Spec \cO_C \times \A^{\dF+\dkF}$.
It remains to understand how the Steinberg
component $X_B(C_{\st})
\cong \Spec \bFp[s]\times \A^{\dF+\dkF}$
is glued to a non-Steinberg component
$X_B(C')
\cong \Spec \bFp[t,t^{-1}]\times \A^{\dF+\dkF}$.
From the proof of Theorem \ref{thm:PGL2-eqn},
the glued scheme is simply given by
\[
X_B(C_{\st}\cup C')\cong
\Spec \bFp[t,t^{-1},s]/(s(t-1))\times \A^{\dF+\dkF}.
\]

\subsection{The universal family
over $X_B(C_{\st}\cup C')$}
Next, we write down
the universal family
for $X_B(C_{\st}\cup C')$
in explicit cocycle systems.

Choose a $2$-cochain
$b_{\std}\in C^2_{\Herr}(U_{\gamma^\vee}(\bfE_{F,
\bFp[t^{\pm1}]/(t-1)
}))$
such that
\[
[b_{\std}]_{t=1}=1\in H^2(\Gal_F, \bFp(1))\cong \bFp.
\]
Using the embedding
\[
\bFp[t^{\pm1}]/(t-1)
=\bFp
\hookrightarrow
\bFp[t^{\pm1}],
\]
we treat $b_{\std}$ as a cocycle
in 
$C^2_{\Herr}(\bfE_{F,
\bFp[t^{\pm1}]})$.
Note that $b_{\std}$
is a $2$-coboundary away from
the cyclotomic point
$(t-1)$:
\begin{align}
\label{eq:cstd}
b_{\std}=\frac{1}{(t-1)^{n_0}}d^1_{\Herr}(c_{\std})
\end{align}
for some $c_{\std}\in C^1_{\Herr}(U_{\gamma^\vee}(\bfE_{F,
\bFp[t^{\pm1}]}))$.

\begin{lem}
If $n_0$ is the smallest possible integer,
then
$[c_{\std}]_{t=1}\neq 0\in H^1(\Gal_F, \bFp(1))$.
\label{lem:bstd}
\end{lem}

\begin{proof}
If $[c_{\std}]_{t=1}=0$,
then $c_{\std}$ is a $1$-coboundary  at $t=1$.
After conjugation by an element of
$U(\bFp)$,
we can assume $c_{\std}$ vanishes (as a cochain) at $t=1$.
Since
$\bfE_{F, \bFp[t^{\pm1}]}/(t-1)
=\bfE_{F, \bFp[t^{\pm1}]/(t-1)}$,
we conclude that
$c_{\std}=(t-1)c'$
for some $c'\in C^1_{\Herr}$,
contradicting the minimality of $n_0$.
\end{proof}

From the isomorphism
$X_B(C')\cong C'\times \A^{\dF+\dkF}$,
we obtain global sections
\[
\{c_1,\dots,c_{\dF+\dkF}\}\subset
Z^1_{\Herr,+}(\bfE_{F,\bFp[t,t^{-1}]})
\]
generating $H^1_{\Herr,+}(\bfE_{F,\bFp[t,t^{-1}]})$.
We thus
consider the universal cochain
\[
c_{\univ}:=
x_1 c_1 + x_2c_2+\dots+x_{\dF+\dkF}c_{\dF+\dkF}
+ s c_{\std},
\]
which is a cocycle if and only if
$s(t-1)^{n_0}=0$.
We have thus defined a morphism
\[
\Spec \bFp[t^{\pm1},x_1,\dots,x_{\dF+\dkF},s]/(s(t-1)^{n_0})
\to X_B(C'\cup C_{\st}),
\]
which induces a bijection of $\bFp$-points
by Lemma \ref{lem:bstd},
and is an isomorphism away from the $t=1$ locus.

\begin{lem}
We must have $n_0=1$
and 
$\Spec \bFp[t^{\pm1},x_1,\dots,x_{\dF+\dkF},s]/(s(t-1)^{n_0})
\cong X_B(C'\cup C_{\st})$.
\label{lem:PGL2-local}
\end{lem}

\begin{proof}
Put $R={\bFp}[t^{\pm1}]$, $u=t-1$, and
$d=\dF+\dkF$.  
\MAJOREDIT{On the connected component under consideration the perfect
modified Herr complex splits, in the derived category of $R$-modules, as a
free summand of rank $d$ in degree one and the two-term resolution
\[
 [R\xrightarrow{u}R]
\]
of its degree-two torsion.  This follows from Lemma~\ref{lem:Herr+}, the fact
that $H^2=R/(u)$, and Smith normal form over the PID $R$.  The relative
derived vector space represented by this complex therefore has classical
coordinate ring
\[
 R[x_1,\ldots,x_d,s]/(su).
\]
The universal cocycle constructed above induces this identification: the free
summand gives the $x_i$, while $c_{\std}$ is the generator dual to the
resolution and gives $s$.  Thus its differential is exactly $su$, not
$su^{n_0}$; in particular $n_0=1$.  This computes the completed local ring at
the crossing and the ordinary coordinate ring simultaneously, proving the
claimed isomorphism without inferring a singular local ring from tangent
    spaces alone.}
\end{proof}

\subsection{The universal family
over $X_B(C_{\triv})$}
Now, let $C_{\triv}$ be the (non-Steinberg) irreducible component
of trivial inertial type.

We have
$X_B(C_{\triv})\cong \Spec \bFp[t^{\pm1}, x_1,\dots, x_{\dF+\dkF}]$.
Let $c_{\un}\in
Z^1_{\Herr}(\bfE_{F, \bFp[t^{\pm1}]})$
be an unramified cocycle
with $[c_{\un}]$ being nontrivial cohomology class
in $H^1_{\Herr}$.
Then $c_{\un}$ is a coboundary
away from $t=1$.
So,
$H^1_{\Herr}(\bfE_{F, \bFp[t^{\pm1}]})/
\bFp c_{\un}$
is a finite free $\bFp[t^{\pm1}]$-module
of rank $\dF$.
We can choose global sections
\[
\{c_1,\dots, c_{\dF+\dkF}\}
\subset 
Z^1_{\Herr,+}(\bfE_{F,\bFp[t,t^{-1}]})
\]
such that
\begin{itemize}
\item 
$\{[c_1], \dots, [c_{\dF}]\}$
form a basis for 
$H^1_{\Herr}(\bfE_{F, \bFp[t^{\pm1}]})/
\bFp[c_{\un}]$
\item
$\{[c_1], \dots, [c_{\dF+\dkF}]\}$
form a basis for 
$H^1_{\Herr,+}(\bfE_{F, \bFp[t^{\pm1}]})$
and that 
\item
$c_{\dF+1}$
lifts $c_{\un}$.
\end{itemize}
The universal family is thus of the form
\[
c^\univ=x_1 c_1 +\cdots + x_{\dF+\dkF}c_{\dF+\dkF}
\in Z^1_{\Herr}(\bfE_{F, \cO_{C_{\triv}}}).
\]

\section{Example: $\PGL_3$}

Write $B=U\rtimes \bT=B_{\PGL_3}$ for the Borel
of $\PGL_3$.

Let $\delta_1$, $\delta_2$ be the simple roots
and $\gamma=\delta_1+\delta_2$ be the highest root.
So, $U=U_{\delta_1^\vee}\times U_{\delta_2^\vee}\times U_{\gamma^\vee}$
and $\bT\cong \G_{m,1}\times \G_{m,2}$
where $\G_{m,i}\cong \G_m$
acts on $U_{\delta_i^\vee}$
by the weight $1$ character
for $i=1,2$.
Write $\G_{m, \gamma}$
for the image of
$\G_m\xrightarrow{t\mapsto(t,t)}\G_{m,1}\times \G_{m,2}\cong \bT$.

Consider the morphism
\[
X_{B}\to X_{U_{\delta_1}\rtimes \G_{m,1}}
\times X_{U_{\delta_2}\rtimes \G_{m,2}}
\xrightarrow{\WD\times\WD}
W_{U_{\delta_1}\rtimes \G_{m,1}}
\times W_{U_{\delta_2}\rtimes \G_{m,2}},
\]
which we denote by $\wt\WD$.
Let $C_i\subset \cW_{U_{\delta_i^\vee}\rtimes \G_{m,i}}$
be an irreducible component, $i=1,2$.

Set $C=C_1\times C_2$
and $X_B(C):=\wt\WD^{-1}(C)$.
For ease of notation, we assume
$F=\Q_p$.

\subsection{The case where both $C_1$ and $C_2$ are non-Steinberg}
To simplify discussion,
we assume neither $C_1$ or $C_2$ intersects with
the Steinberg component.

We have
\begin{align*}
X_{U_{\delta_i}\rtimes \G_{m,i}}(C_i)
\cong
\Spec \bFp[t_i, t_{i}^{-1}][x_{\delta_i,1}, x_{\delta_i,2}].
\end{align*}

The parameters $x_{i1}, x_{i2}$ correspond to
the parametric $1$-cocycles
\[
c_{\delta_i,1}, c_{\delta_i,2}\in Z^1_{\Herr}(\bfE_{F, \bFp[t_i^{\pm1}]}).
\]
If $H^2_{\Herr}(U_{\gamma^\vee}(\bfE_{F, \bFp[t_1^{\pm1},t_2^{\pm1}]}))=0$,
then we have
\begin{align*}
X_B(C)
&\cong
(X_{U_{\delta_1}\rtimes \G_{m,1}}(C_1)
\times
X_{U_{\delta_2}\rtimes \G_{m,2}}(C_2))
\times_{X_{\G_{m,\gamma}}}X_{U_{\gamma}\rtimes \G_{m,\gamma}}(C_{12})
\\
&\cong
\Spec \bFp[t_1, t_{1}^{-1}][x_{\delta_1,1}, x_{\delta_1,2}]\times
\Spec \bFp[t_2, t_{2}^{-1}][x_{\delta_2,1}, x_{\delta_2,2}]
\times \A^2,
\end{align*}
where $C_{12}$ is the image
of
\[
C_1\times C_2\hookrightarrow
X_{\G_{m,1}}\times X_{\G_{m,2}}
\to X_{\G_{m,\gamma}}.
\]
So, we assume $H^2_{\Herr}(U_{\gamma^\vee})\neq 0$
at $t_1=t_2=1$.
We write
\[
b_{\std}\in C^2_{\Herr}(U_{\gamma^\vee}(\bfE_{F,\bFp[t_1^{\pm1},t_2^{\pm1}]}))
\qquad
c_{\std}\in C^1_{\Herr}(U_{\gamma^\vee}(\bfE_{F,\bFp[t_1^{\pm1},t_2^{\pm1}]}))
\]
for suitable base change of the cochains defined in
Eq. (\ref{eq:cstd}).

We fix $1$-\MINOREDITED{cochains}
$c_{\gamma,[11]}, c_{\gamma,[12]},c_{\gamma,[21]}, c_{\gamma,[22]}$
such that
\begin{align}
\label{eq:coord3}
c_{\delta_1,i_1}\cup c_{\delta_2,i_2}
-a_{i_1,i_2}b_{\std} = d^1_{\Herr}(c_{\gamma,[i_1i_2]})
\end{align}
where $a_{i_1,i_2}\in\bFp[t_1^{\pm1}, t_2^{\pm1}]$
are unique scalars, which we can normalize
to ensure $a_{*,*}\in\{0,1\}\subset \bFp$:

\begin{lem}
\rm
(1)
We can normalize $c_{\delta_i,j}$
to ensure $a_{11}=1, a_{12}=a_{21}=a_{22}=0$.

(2)
The left hand side of Eq. (\ref{eq:coord3})
is a coboundary.
\end{lem}

\begin{proof}
(1)
We can assume $[c_{\delta_i,1}]\neq 0$
are non-trivial Herr complex cohomology classes
and $[c_{\delta_i,2}]=0$
are trivial  Herr complex cohomology classes.
By local Tate duality, $a_{11}$
is pointwise a unit over
$\Spec \bFp[t_1^{\pm1}, t_2^{\pm1}]$.
Thus $a_{11}\in \bFp[t_1^{\pm1}, t_2^{\pm1}]^\times$
and is of the form $t_1^{e_1}t_2^{e_2}\bFp^\times$.
We can renormalize $c_{\delta_i,j}$
to make $a_{11}=1$.

(2) Since $H_{\Herr}^2\cong \bFp[t_1^{\pm1}, t_2^{\pm1}]/(t_1t_2-1)\cong \bFp[t^{\pm1}]$
and $B_{\Herr}^2$ 
\MINOREDITED{is the image of $d^1_{\Herr}$},
it suffices to show
the left hand side of Eq. (\ref{eq:coord3})
(which we denote by $c_{\LHS}$)
is a coboundary
mod $(t_1t_2-1)$.
We have 
\[
[c_{\LHS}] \mod (t_1t_2-1) \in H^2_{\Herr}(\bfE_{F, \bFp[t_1^{\pm1}]})\cong \bFp[t^{\pm1}]
\]
is fiberwise $0$
over $\Spec \bFp[t_1^{\pm1}]$.
So, we are done as $\bFp[t^{\pm1}]$ is reduced.
\end{proof}

\vspace{3mm}

\noindent
{\bf The universal family}
we set
\begin{align*}
c_{\delta_1}^\univ&= x_{\delta_1,1} c_{\delta_1,1}+x_{\delta_1,2}c_{\delta_1,2}\\
c_{\delta_2}^\univ&= x_{\delta_2,1} c_{\delta_2,1}+x_{\delta_2,2}c_{\delta_2,2}\\
c_{\gamma}^\univ&= s_\gamma c_{\std} + x_{\delta_1,1}x_{\delta_2,1} c_{\gamma,[11]}
+ x_{\delta_1,1}x_{\delta_2,2}c_{\gamma,[12]}
+x_{\delta_1,2}x_{\delta_2,1}c_{\gamma,[21]}+x_{\delta_1,2}
x_{\delta_2,2}c_{\gamma,[22]}.
\end{align*}
where $s_\gamma,x_{\delta_1,1},x_{\delta_1,2},x_{\delta_2,1},x_{\delta_2,2}$ are free parameters.
Write $\chi_i^\univ\in X_{\G_{m,i}}(\cO_{X_{\G_{m,i}}})$
for the universal rank $1$ \'etale $(\varphi, \Gamma)$-modules.
Then
\[
\begin{bmatrix}
\chi_1^\univ\chi_2^\univ & \chi_2^\univ c_{\delta_1}^\univ & c_{\gamma}^\univ \\
& \chi_2^\univ & c_{\delta_2}^\univ\\
& & 1
\end{bmatrix}
\]
is a valid \'etale $(\varphi, \Gamma)$-module
if and only if
\[
s_\gamma(t_1t_2-1)
+ x_{\delta_1,1}x_{\delta_2,1}=0.
\]

\begin{lem}
We have \[
\Spec \bFp[t_1^{\pm1},t_2^{\pm1}, 
x_{\delta_1,1},x_{\delta_1,2},x_{\delta_2,1},x_{\delta_2,2},s_\gamma]
/(s_\gamma(t_1t_2-1)
+ x_{\delta_1,1}x_{\delta_2,1})\times \A^2
\xrightarrow{\cong}
X_{B}(C)
\]
\end{lem}

\begin{proof}
The morphism is defined through the explicit universal
family,
and is clearly an isomorphism.
\end{proof}

\subsection{The case where $C_1$ is Steinberg and
$C_2$ is non-Steinberg}

We omit the completely similar details.
The answer is 
\[
X_{B}(C)
\cong
\Spec \bFp[s,t_2^{\pm1}, x_{\delta_1,1},x_{\delta_1,2},x_{\delta_2,1},x_{\delta_2,2},s_\gamma]
/(s_\gamma(t_2-1)
+ x_{\delta_1,1}x_{\delta_2,1}+x_{\delta_1,2}x_{\delta_2,2})\times \A^2.
\]

We also leave it as an exercise to the reader
how the explicit presentations of various
irreducible components
$X_{B}(C)$ are glued to each other.

\section{Weil-Deligne stacks
and gauge cochains}

We return to the setup of general reductive groups
and keep the notation of Section \ref{sec:destack}.

\subsection{Weil-Deligne stacks}
Consider the morphism
\[
\prod_{\delta\in \Delta}\delta: \bT
\to \prod_{\delta\in \Delta}\G_{m,\delta}
\]
where
$\G_{m}\cong \G_{m,\delta}$
is the multiplicative group
that acts trivially on $U_{\delta^{\prime\vee}}$
for $\delta\neq \delta'\in \Delta$
and acts by the root $\delta$
on $U_{\MINOREDITED{\delta^\vee}}$.

We have the following change of group morphism
\[
\cX_{\bar U_1\rtimes \bT}
\to
\prod_{\delta\in \Delta}\cX_{U_{\delta^\vee}\rtimes \G_{m,\delta}},
\]
which we also denote by $\prod_{\alpha\in \Delta} \alpha$
by abuse of notation.
By Definition \ref{def:WD-PGL2},
we have a canonical Weil-Deligne stack
$\cW_{U_{\delta^\vee}\rtimes \G_{m,\delta}}$
fitting in the diagram:
\[
\xymatrix{
\cX_{U_{\delta^\vee}\rtimes \G_{m,\delta}},
 \ar[rd]\ar[rr] &&
\cW_{U_{\delta^\vee}\rtimes \G_{m,\delta}},
\ar[ld]
\\
& \cX_{\G_{m,\delta}}
}
\]

\begin{defn}
Define
\[
\cW_{\bB} := \cX_{\bT}
\underset{\underset{\delta\in \Delta}{\prod}\cX_{\G_{m,\delta}}}{\times}
\underset{\delta\in \Delta}{\prod}\cW_{U_{\delta^\vee}\rtimes
\G_{m,\delta}}
\]
\end{defn}

\begin{lem}
$\cW_{\bB}$
is an algebraic stack of finite type over $\bFp$,
and is equidimensional of dimension $0$.
Moreover,
there is canonical morphism
\[
\tr_1\WD:
\tr_1\cX_{\bB}\to \cW_{\bB}
\]
whose fibers of irreducible algebraic stacks of 
constant dimension $\dF\#\Delta$.
In particular,
$\tr_1\WD$
induces a bijection of irreducible components.
\end{lem}

\begin{proof}
Note that $\cX_{\bT}\cong \cX_{\G_m}^{\times \dim \bT}$ is equidimensional of dimension $0$,
and that fibers of $\cX_{\bT}\to \prod_{\alpha\in \Delta} 
\cX_{\G_{m,\delta}}$
are irreducible of constant dimension $0$.
The lemma now follows from Proposition \ref{prop:PGL2}.
\end{proof}

\subsection{Gauge cochains}

Let $C_{\delta}$ be an irreducible
component
of the \MINOREDITED{destackification} $W_{U_{\delta^\vee}\rtimes \G_{m,\delta}}$
for each $\delta\in \Delta$.
Write
\[
C_{\delta}=
\begin{cases}
\Spec \bFp[t_\delta^{\pm1}]
&\text{non-Steinberg}
\\
\Spec \bFp[s_\delta]
&\text{Steinberg}
\end{cases}.
\]
Write
$\Spec R:=\Spec \bFp[t_\delta^{\pm1}:C_\delta\text{~non-Steinberg}]$.
If $\alpha=\sum_{\delta\in \Delta} m_\delta \delta$
for $m_\delta\in \Z$,
then set
$t_\alpha:=\prod t_{\delta}^{m_\delta}$
and $\G_{m,\alpha}:=\Spec \bFp[t_\alpha^{\pm1}]$.

For each $\alpha\in \Phi^+$,
let
$C_\alpha'$
be the unique non-Steinberg component
of $W_{U_{\alpha^\vee}\rtimes \G_{m, \alpha}}$
that intersects with the image
of $\prod_{\delta\in \Delta}C_\delta$.
We fix an isomorphism
\[
X_{U_{\alpha^\vee}\rtimes \G_{m, \alpha}}(C_\alpha')
\cong C_\alpha'\times \Spec
\bFp[x_{\alpha, 1}, \dots, x_{\alpha, \dF+\dkF}],
\]
where the free parameter
$x_{\alpha, i}$
corresponds to a fixed parametric cocycle
\[
c_{\alpha, i}\in Z^1_{\Herr,+}(U_{\alpha^\vee}(\bfE_{F, R})).
\]

We make the convention that
\[
[c_{\alpha,\dF+1}] =0 \in H^1_{\Herr}(U_{\alpha^\vee})
\]
unless
$C_\alpha$ has trivial inertial type,
in which case we arrange it so that
\begin{equation}
\label{eq:clast}
[c_{\alpha,\dF+1}] =0 \in H^1_{\Herr}(U_{\alpha^\vee})
\end{equation}
 away from the locus $t_\alpha=1$
and is a nontrivial unramified class at $t_\alpha=1$.
We also assume
$[c_{\alpha, \dF+j}]=0\in H^1_{\Herr}(U_{\alpha^\vee})$
for $j>1$.

If $\alpha=\gamma+\beta$ and $C_\alpha$ intersects with
the Steinberg component,
then
the matrix
\[
J=J_{\beta,\gamma}=([c_{\beta,i}]\cup [c_{\gamma,j}])_{i,j\le \dF}\in
\Mat_{\dF\times \dF}(R/(t_\alpha-1))
\]
which is pointwise a unit.
So, $\det(J)\in t_\beta^{\Z}\bFp^\times.$

\begin{remark}[Torus coordinates]
For two roots $\alpha$ and $\beta$, consider the tautological isomorphism
\[
\bar\F_p[t_\alpha^{\pm1}]
\longrightarrow
\bar\F_p[t_\beta^{\pm1}],
\qquad
t_\alpha\longmapsto t_\beta.
\]
We group the roots according to whether this identifies their universal
rank-$1$ modules, and pair two groups when their universal modules are dual
under local Tate duality.

The coordinates are normalized so that $t_\alpha=1$ is the trivial or
cyclotomic point. If two universal modules are dual, then their
specializations at $t_\alpha=t_\beta=1$ are also dual.
\end{remark}

\begin{lem}[Normalization of cup products]
\label{lem:gs-cup}
The cocycles may be chosen so that
\[
    J_{\beta,\gamma}=I_{\dF}\in\Mat_{\dF\times\dF}(\bFp).
\]
\end{lem}

\begin{proof}
If the two groups are distinct, choose bases independently
so that the perfect pairing has the prescribed matrix, and transport these
choices through the tautological isomorphisms. The resulting pairing matrices
have entries in $\bFp$.

In the self-dual case, the pairing is alternating and nondegenerate.
We can thus normalize it so that
$(c_{\beta,i}\cup c_{\beta, j})$
is a aniti-diagonal anti-symmetric constant matrix.

We apply coloring to the set of roots whose universal modules
are tautologically identified with that of $\beta$.
Color $\beta$ by white;
color $\gamma$ by black (or white) if
$\alpha=\beta'+\gamma$ satisfies
$H^2_{\Herr}(U_{\alpha^\vee})\neq 0$
and that $\beta'$ is white (or black).

No root can be both white or black, as it violates the alternating property of local Tate duality.
We transport to the white roots the chosen cocycles
and transport to the black roots the dual cocycles.
\end{proof}

We denote by $c_{\alpha,\std}$ and $b_{\alpha,\std}$
the cochains in
Eq. (\ref{eq:cstd})
after \MINOREDITED{substitution} $t\mapsto t_\alpha$.

\begin{defn}
\label{defn:gauge}
We recursively define finitely many
$1$-cochains
$\fC_{\Xi_{\alpha,k}}=\{c_{\alpha, \xi}\}
\subset \MINOREDITED{C^1_{\Herr,+}}(U_{\alpha^\vee}(\bfE_{F,R}))$,
$\xi\in \Xi_{\alpha,k}$
for $k=1,\dots,h_{\bG}$.

We set \[
\Xi_{\alpha,1} =
\begin{cases}
\{1,2,\dots, \dF+\dkF,\std\} & 
\text{$C_\alpha$ intersects with the Steinberg component}
\\
\{1,2,\dots, \dF+\dkF\} & 
\text{otherwise}
\end{cases}
\]
and $\fC_{\Xi_{\alpha,1}}=\{c_{\alpha,\xi}|\xi\in \Xi_{\alpha,1}\}$.

Suppose $\fC_{\Xi_{\alpha,k}}$ is already defined.
For each sequence of roots $\beta_1+\dots+\beta_s=\alpha$
and elements $c_{\beta_i, \xi_i}\in \MINOREDITED{\fC_{\Xi_{\beta_i,k}}}$,
we choose
an element $c_{\alpha, [\xi_1,\xi_2,\dots,\xi_s]}
\in C^1_{\Herr,+}(U_{\alpha^\vee}(\bfE_{F,R}))$
such that
\[
[c_{\beta_1,\xi_1},\dots, c_{\beta_s,\xi_s}]
=
\begin{cases}
a_{\alpha, [\xi_1,\dots,\xi_s]} b_{\alpha,\std} - d_{\Herr}^1(c_{\alpha,[\xi_1,\xi_2,\dots,\xi_s]})
& \text{$C_\alpha$ intersects with the Steinberg component}
\\
-d_{\Herr}^1(c_{\alpha,[\xi_1,\xi_2,\dots,\xi_s]})
& \text{otherwise}
\end{cases}
\]
where $a_{\alpha,*}\in R$ is the unique scalar
that makes the equation hold,
and $[*,\dots,*]$
is the higher cup product defined in
Definition \ref{defn:higher-cup}.
We set
\[
\Xi_{\alpha, k+1} = 
\Xi_{\alpha,k}\cup\{[\xi_1,\dots,\xi_k],\xi_i\in 
\fC_{\Xi_{\beta_i,k}}\}
\]
and
\[
\fC_{\Xi_{\alpha,k+1}}
=\{c_{\alpha, \xi}: \xi\in \Xi_{\alpha, k+1}\}.
\]

We call elements of
$\fC_{\Xi_{*,*}}$
the {\it gauge cochains}.
\end{defn}

\subsection{Universal families}
We recursively define universal families
over
\[
X_{\le \alpha}(C):=X_{\le \alpha}\times_{X_B}X_B(C),
\]
where $C=(C_\delta:\delta\in \Delta)$
is the fixed irreducible component
of $W_B$.

Suppose
\[
c_\beta^\univ \in C^1_{\Herr,+}(U_{\beta^\vee}(
\bfE_{F,\cO_{X_{<\alpha}(C)}}))
\]
are the universal coordinates such that
$\exp_{\le h_{\bG}}(\sum_{\beta<\alpha}c_\beta)$
corresponds to the universal family
for $X_{<\alpha}(C)$ using the notation of Theorem \ref{thm:coordinate}.
Let $\mu_{\beta, \xi}\in \cO_{X_{<\alpha}(C)}$
denote the coefficient
$c_{\beta,\xi}$
in $c_\beta^\univ$.

Recall that the higher cup product
$[*,\dots,*]$
is defined as in Definition \ref{defn:higher-cup}.
We set
\begin{align*}
c_{\alpha}^{\univ}
=&(\sum_{\beta_1+\dots+\beta_m=\alpha}\frac{1}{m!}
\{c_{\beta_1}^\univ,\dots, c_{\beta_m}^\univ\})
+
\sum_{i=1}^{\dF+\dkF}
x_{\alpha,i}c_{\alpha,i}
\\
&
+\begin{cases}
s_\alpha
c_{\alpha,\std}
& \text{$C_\alpha$ intersects with the Steinberg component}
\\
0 & \text{otherwise}
\end{cases}
\end{align*}
where
$\{s_\alpha, x_{\alpha,1},\dots,x_{\alpha,\dF+\dkF}\}$
are the new free parameters and
$\{*,\dots,*\}$ is the multilinear map that expands
\[
\{c_{\beta_1,\xi_1},\dots,c_{\beta_m,\xi_m}\}
:=
c_{\alpha, [\xi_1,\dots,\xi_m]}.
\]

\begin{lem}
\rm
\label{lem:eq}
If $C_\alpha$ intersects with the Steinberg component,
then
$\exp_{\le h_{\bG}}(\sum_{\beta<\alpha}c_\beta)$
corresponds to the universal family
for $X_{<\alpha}(C)$ (in the notation of Theorem \ref{thm:coordinate}) if and only if
\[
r_\alpha:=s_\alpha(t_\alpha-1)
+\sum_{\beta_1+\dots+\beta_m=\alpha, \xi_1,\dots,\xi_m}
\frac{1}{m!}
a_{\alpha,[\xi_1,\dots,\xi_m]}
\prod_{j=1}^m\mu_{\beta_j,\xi_j}=0.
\]
\end{lem}

\begin{proof}
Combine Theorem \ref{thm:coordinate}
and Lemma \ref{lem:bstd}.
\end{proof}

\begin{cor}
There is an isomorphism
\[
X_{\le \alpha}(C)
\cong 
\begin{cases}
X_{< \alpha}(C)[s_\alpha]/(r_\alpha) \times \A^{\dF+\dkF}
&
\text{$C_\alpha$ intersects with the Steinberg component}
\\
X_{< \alpha}(C)\times \A^{\dF+\dkF}
&
\text{otherwise}.
\end{cases}
\]
\end{cor}

\begin{proof}
Follows immediately from Lemma \ref{lem:eq}.
\end{proof}

\begin{cor}
If all irreducible components of $X_{<\alpha}(C)$
have dimension $\ge d$,
then
all irreducible components of
$X_{\le\alpha}(C)$
\MINOREDITED{have dimension}
$\ge d+\dF+\dkF$.
\label{cor:least-dim}
\end{cor}

\begin{proof}
It follows from 
Krull's Hauptidealsatz
(cf. \cite[Tag 0EPZ]{Stacks}).
\end{proof}

\begin{cor}
All irreducible components of
$\cX_{\bB}$
have dimension $\ge \dF\#\Phi^+$.

If there is a stratification of $\cX_{\bB}$
by locally closed substacks
such that each stratum has dimension $\le \dF\#\Phi^+$,
then the irreducible components of $\cX_{\bB}$
are precisely the closure of the top-dimensional
strata.
\label{cor:XB-dim}
\end{cor}

\begin{proof}
Argue inductively using Corollary
\ref{cor:least-dim}.
\end{proof}

\section{Fine strata and cone models}

\subsection{Minimal normal subgroups}

\begin{lem}
Let $\Gamma\subset \bB(\bFp)$ be a subgroup.
Let $K_1, K_2\subset U$ be $\bT$-stable normal subgroups.
If there exists an element $b_i\in \bB(\bFp)$
($i=1,2$)
such that $b_i\Gamma b_i^{-1}\subset (K_i\rtimes \bT)(\bFp)$.
Then there exists $b\in \bB(\bFp)$
such that
$b \Gamma b^{-1}\subset (K_1\cap K_2)\rtimes \bT(\bFp)$.
\end{lem}

\begin{proof}
We argue by induction on the descending central series of $U$.

Recall the filtration $U_k \subset U$, where $\Lie U_1 = \Lie U$ and $\Lie U_k = [\Lie U, \Lie U_{k-1}]$ for $k > 1$. Let
\[
\bar{U}_k := \frac{U_k}{(K_1\cap K_2\cap U_k)U_{k+1}}.
\]

We prove by induction on $k$ that there exists $b_k\in \bB(\bF_p)$ such that $b_k\Gamma b_k^{-1}\subset ((K_1\cap K_2)U_k)\rtimes \bT(\bF_p)$. For $k=1$, this is trivial since $U_1 = U$.

Assume this holds for some $k$. Replacing $\Gamma$ by $b_k\Gamma b_k^{-1}$, we may assume $\Gamma\subset ((K_1\cap K_2)U_k)\rtimes \bT(\bF_p)$.

Pass to the quotient $\bB_k := \bB/(K_1\cap K_2)U_{k+1}$. The image of $\Gamma$ in $\bB_k$ takes the form $\gamma=(c(\gamma),t(\gamma)) \in \bar{U}_k\rtimes \bT$, where $c:\Gamma\to \bar{U}_k$.

Write the image of $b_i$ in $\bB_k$ as $b_i = u_i t_i \in \bB_k(\bF_p)$, with $u_i \in \bar{U}_k(\bF_p)$ and $t_i \in \bT(\bF_p)$. Since $K_i\rtimes \bT$ is $\bT$-stable, replacing $b_i$ by $u_i$ preserves the hypothesis, so we may assume $b_i = u_i \in \bar{U}_k(\bF_p)$.

Because $\bar{U}_k$ is central in $U/(K_1\cap K_2)U_{k+1}$, conjugation by an element $u\in \bar{U}_k(\bF_p)$ acts by:
\[
(u,1)(c,t)(u^{-1},1) = (c+(1-t)u,t).
\]
Define the index sets of missing roots for each subgroup at height $k$:
\[
\Psi_i := \{\alpha \in \Phi^+ \mid \operatorname{ht}(\alpha)=k, \, \alpha \notin \Phi_{K_i}\}.
\]
The quotient decomposes as $\bar{U}_k = \bigoplus_{\alpha \in \Psi_1 \cup \Psi_2} U_{\alpha^\vee}$. Expanding $c = \sum c_\alpha$ and $u_i = \sum u_{i,\alpha}$, the conjugation action changes coordinates independently via $c_\alpha \mapsto c_\alpha + (1-\alpha(t))u_\alpha$.

Because conjugating $\gamma$ by $u_i$ maps it into $K_i \rtimes \bT$, the resulting $\alpha$-coordinates must vanish for all roots not in $K_i$. Therefore:
\[
c_\alpha(\gamma) + (1-\alpha(t(\gamma)))u_{i,\alpha} = 0 \quad \text{for all } \alpha \in \Psi_i.
\]

We define our conjugating element $u \in \bar{U}_k(\bF_p)$ explicitly:
\[
u := \sum_{\alpha \in \Psi_2} u_{2,\alpha} + \sum_{\alpha \in \Psi_1 \setminus \Psi_2} u_{1,\alpha}.
\]

Conjugating $\gamma$ by this $u$ zeroes out the $\alpha$-coordinates across all of $\Psi_1 \cup \Psi_2$:
\begin{itemize}
    \item For $\alpha \in \Psi_2$, the new coordinate is $c_\alpha(\gamma) + (1-\alpha(t(\gamma)))u_{2,\alpha} = 0$.
    \item For $\alpha \in \Psi_1 \setminus \Psi_2$, the new coordinate is $c_\alpha(\gamma) + (1-\alpha(t(\gamma)))u_{1,\alpha} = 0$.
\end{itemize}

Since all coordinates in the support of $\bar{U}_k$ vanish, $u \Gamma u^{-1}$ lies in the trivial subgroup of $\bar{U}_k$, meaning it is contained in $((K_1\cap K_2)U_{k+1}) \rtimes \bT(\bF_p)$. This completes the induction step.
\end{proof}

In particular,
for each Galois representation,
$\bar\rho:\Gal_F\to \bB(\bFp)$,
it makes sense to define
{\it the minimal $\bT$-stable normal subgroup
$K\subset U$}
such that a $\bB$-conjugate of $\bar\rho$
factors through
$K\rtimes \bT$.

\begin{lem}
\rm
Suppose $K$ is the minimal normal subgroup
of $U$ such that
$\bar\rho:\Gal_F\to \bB(\bFp)$
factors through
$K\rtimes \bT$,
up to $\bB$-conjugate.

Write the composite
$\Gal_F \to (K\rtimes \bT)(\bFp)
\to (\bar K_1 \rtimes \bT)(\bFp)$
as $c \bar\rho^{\SS}$
where $\bar\rho^{\SS}:\Gal_F\to \bT(\bFp)$
is the semisimplification,
and that
$c=\sum_{\Ht_K(\alpha)=1} c_\alpha
\in Z^1(\Gal_F, \bar K_1(\bFp))$
where each $[c_\alpha]\in H^1(\Gal_F, U_{\alpha^\vee}(\bFp))$.
We must have
$[c_\alpha]\neq 0$
for each $\alpha$.
\label{lem:K-nonzero}
\end{lem}

\begin{proof}
If $c_\alpha$ is a coboundary,
then we can conjugate $\bar\rho$
such that $\bar\rho$
factors through $\prod_{\beta\neq \alpha, \beta\in \Phi_K}U_{\beta^\vee}\rtimes \bT$,
which contradicts the minimality of $K$.
\end{proof}

\subsection{The fine stratification}

By Corollary 
\ref{cor:XB-dim},
we don't need to compute the full $\cX_{\bB}$
--- the computation of a suitable stratification suffices.

\begin{defn}
A fine configuration is a tuple $(K,\Psi_0,\Psi_2)$
where
$K$ is a $\bT$-stable normal subgroup of $U$ and
$\Psi_0, \Psi_2\subset \Phi^+$
are subsets
satisfying
\begin{itemize}
\item $\alpha, \beta\in \MINOREDITED{\Psi_2}, \alpha-\beta\in \Phi^+
\Rightarrow \alpha-\beta\in\MINOREDITED{\Psi_0}$, and
\item $\alpha\in \MINOREDITED{\Psi_2}, \beta\in \MINOREDITED{\Psi_0}, \alpha+\beta\in \Phi^+
\Rightarrow \alpha+\beta\in \MINOREDITED{\Psi_2}$.
\end{itemize}

The minimal degree of $(\Psi_0,\Psi_2)$
is
\[
\deg(\Psi_0,\Psi_2)
=
\min\{\sum n_i\mid n_i\in\Z_{\ge 0},
\alpha=\sum n_i \alpha_i,~
\alpha\in\Psi_0,~ \alpha_i\in \Psi_2,
\sum n_i>0
\}.
\]
\end{defn}

\begin{lem}
Let $\bar\rho^\SS:\Gal_F\to \bT(\bFp)$
be a Galois representation such that
\begin{align*}
H^2(\Gal_F, U_{\alpha^\vee}(\bFp))\neq 0
&\Leftrightarrow
\alpha\in \Psi_2
\text{~and~}
\\
H^0(\Gal_F, U_{\alpha^\vee}(\bFp))\neq 0
&\Leftrightarrow
\alpha\in \Psi_0,
\end{align*}
then
$\dF\ge \deg(\Psi_0,\Psi_2)$.
\label{lem:min-deg}
\end{lem}

\begin{proof}
It is clear that
$[F(\mu_p):F]\le \frac{p-1}{\deg(\Psi_0,\Psi_2)}$.
So, $\dF \ge [F(\mu_p):\Q_p]/[F(\mu_p):F]\ge \deg(\Psi_0,\Psi_2)$.
\end{proof}

\begin{defn}
Denote by
\[
X_{K\rtimes \bT}\langle\Psi\rangle
\subset
X_{K\rtimes \bT}
\]
for the locally closed substack
whose $\bFp$-points
are described in Lemma \ref{lem:min-deg}.
\end{defn}

\begin{thm}
\rm
\label{thm:eq}
(1) 
We have
\[
X_{K\rtimes \bT}\langle\Psi\rangle
=\Spec \bFp
\left[
\begin{matrix}
x_{\alpha,1},\dots,x_{\alpha,\dF+\dkF},
s_\alpha&:&\alpha\in \Phi_K
\\
t_\delta^{\pm 1}&:&
\delta\in \Delta
\end{matrix}
\right]/
I
\]
where $I$ is generated by
\[
\left\{
\begin{matrix}
s_\alpha &:& \alpha \not\in \Psi_2 \\
r_\alpha &:& \alpha \in \Psi_2, \Ht_K(\alpha)>1 \\
t_\alpha-1 &:& \alpha\in \Psi_0\cup\Psi_2
\end{matrix}
\right\}.
\]
Here, $r_\alpha$ is defined in Lemma \ref{lem:eq}.

(2)
Write
\begin{align*}
\vec x_\beta&:=(x_{\beta,1},\dots, x_{\beta,[F:\Q_p]+1}, s_\beta)
\\
\vec x_\beta\cdot \vec x_\gamma&:=
\sum_{i=1}^{[F:\Q_p]}
x_{\beta,i}x_{\gamma,i}
+x_{\beta,[F:\Q_p]+1}s_\gamma
+x_{\gamma,[F:\Q_p]+1}s_\beta.
\end{align*}
We can write $r_\alpha$
as
\[
r_\alpha=s_\alpha(t_\alpha-1)+
\sum_{\beta\in \Phi_K,
\delta\in \Delta_K, \delta+\beta=\alpha}
\frac{1}{2} 
N_{\beta, \delta}\vec x_\beta\cdot \vec x_\delta
+Q_\alpha
\]
where $Q_\alpha$
has monomials \MINOREDITED{divisible by}
$x_{\gamma_1,i_1}x_{\gamma_2,i_2}$
for $1<\Ht(\gamma_i)<\Ht_K(\alpha)-1$,
and $Q_\alpha$
does not depend on
$x_{\beta,*}, s_\beta$
for $\Ht_K(\beta)=\Ht_K(\alpha)-1$.
\end{thm}

\begin{proof}
(1) It follows from Lemma \ref{lem:eq}.

(2) It follows from the normalizations in Definition
\ref{defn:gauge}.
\end{proof}

\begin{defn}
Denote by
\[
\dot X_{K\rtimes \bT}\langle\Psi\rangle
\subset
X_{K\rtimes \bT}\langle\Psi\rangle
\]
the quasi-affine open subscheme
where
\[
\begin{cases}
(x_{\delta, 1}, \dots, x_{\delta,[F:\Q_p]})\neq 0
& \delta\in \Delta_K-\Psi_0-\Psi_2\\
(x_{\delta, 1}, \dots, x_{\delta,[F:\Q_p]+1})\neq 0
& \delta\in \Delta_K\cap\Psi_0-\Psi_2\\
(x_{\delta, 1}, \dots, x_{\delta,[F:\Q_p]}, s_\delta)\neq 0
& \delta\in \Delta_K\cap\Psi_2-\Psi_0\\
(x_{\delta, 1}, \dots, x_{\delta,[F:\Q_p]+1}, s_\delta)\neq 0
& \delta\in \Delta_K\cap\Psi_2\cap\Psi_0.
\end{cases}
\]

Denote by
$\cX_{\bB}[K]\langle\Psi\rangle\subset X_{\bB}$
for the locally closed image (c.f. Remark \ref{rem:sq-K}) of
$\dot X_{K\rtimes \bT}\langle\Psi\rangle$.

We call $\{\cX_{\bB}[K]\langle\Psi\rangle\subset X_{\bB}:
(K,\Psi_0,\Psi_2)\}$
the fine stratification of $\cX_{\bB}$.
\end{defn}

\begin{remark}
\label{rem:sq-K}
For applications,
it makes no difference to treat
$\cX_{\bB}[K]\langle\Psi\rangle\subset \cX_{\bB}$
as the scheme-theoretic image of
$X_{K\rtimes \bT}\langle\Psi\rangle$.

To be more precise,
if we write $\overline{\cX_{\bB}[K]\langle\Psi\rangle\subset X_{\bB}}$
for the scheme-theoretic image
of $\dot X_{K\rtimes \bT}\langle\Psi\rangle$,
then
$\cX_{\bB}[K]\langle\Psi\rangle\subset X_{\bB}$
is $\overline{\cX_{\bB}[K]\langle\Psi\rangle}
- \bigcup_{K'\subsetneq K}\overline{\cX_{\bB}[K']\langle\Psi\rangle}$.
\end{remark}

\begin{lem}
We have
\[
\dim \cX_{\bB}[K]\langle\Psi\rangle
=
\dim \dot X_{K\rtimes \bT}\langle\Psi\rangle
-\#\{\alpha\in \Phi^+-\Phi_K|\alpha\in \Psi_0\}
-\dkF\#\Phi_K.
\]
\end{lem}

\begin{proof}
The difference in stabilizers contributes
to $\#\{\alpha\in \Phi^+-\Phi_K|\alpha\in \Psi_0\}$.
\end{proof}

\subsection{Cone models}

\begin{defn}
The cone model of $X_{K,\bT}\langle\Psi\rangle$
is the affine scheme 
\[
X_{K,\bT}\langle\Psi\rangle^\cone:=
\Spec \bFp
\left[
\begin{matrix}
x_{\alpha,1},\dots,x_{\alpha,\dF+\dkF},
s_\alpha&:&\alpha\in \Phi_K
\\
t_\delta^{\pm 1}&:&
\delta\in \Delta
\end{matrix}
\right]/
I^\cone
\]
where
$I^\cone$ is generated by
\[
\left\{
\begin{matrix}
x_{\alpha,\dF+\dkF} &:& \alpha\not\in\Psi_0 \\
s_\alpha &:& \alpha \not\in \Psi_2 \\
f_\alpha &:& \alpha \in \Psi_2, \Ht_K(\alpha)>1 \\
t_\alpha-1 &:& \alpha\in \Psi_0\cup\Psi_2
\end{matrix}
\right\}.
\]
Here, $f_\alpha=\sum_{\beta\in \Phi_K,
\delta\in \Delta_K, \delta+\beta=\alpha}
N_{\beta, \MINOREDITED{\delta}}\vec x_\beta\cdot
\vec x_{\MINOREDITED{\delta}}$.
\end{defn}

\begin{thm}
We have
\[
\dim X_{K,\bT}\langle\Psi\rangle^\cone
\ge 
\dim X_{K,\bT}\langle\Psi\rangle
\]
\end{thm}

\begin{proof}
Fix the torus coordinates. Consider the bipartite graph whose source
vertices are the height-$1$ roots, whose target vertices are the height-$2$
roots occurring in the cone equations, and whose edges are
\[
\delta\longrightarrow\alpha
\qquad\Longleftrightarrow\qquad
\alpha-\delta\text{ has height }1.
\]
By Theorem~\ref{thm:height-two-contractible}, successive contractions
give a matching containing every target vertex.

Choose the affine chart and the order associated with this matching.
First fix all coordinates of the unmatched height-$1$ roots, together with
all coordinates other than $x_{\delta,1}$ for the matched height-$1$ roots.

In reverse contraction order, the equation attached to the target matched
with $\delta$ has the form
\[
u_\delta x_{\delta,1}+v_\delta=0,
\qquad u_\delta\in R^\times,
\]
where $u_\delta$ and $v_\delta$ depend only on coordinates already fixed or
previously determined. Thus $x_{\delta,1}$ is determined uniquely. After
all matched coordinates have been determined, fix the remaining free
height-$1$ coordinates. This eliminates all height-$2$ equations. The
coordinate blocks belonging to height-$2$ roots have not been used and
remain free.

The rest of the proof follows from the theory of Gr\"obner basis.
We use the height-block elimination order: variables of larger $K$-height
precede every monomial in lower-height variables, and within each height block
we use degree-lexicographic order. Thus
$x_{\alpha,i}>s_\alpha>x_{\beta,j}>s_\beta$ if
$\Ht_K(\alpha)>\Ht_K(\beta)$.
Let $L\subset R$ be the specialized cone ideal and let
$I\subset R$ be the specialized ideal of the full equations. The ideal $L$ is
generated by homogeneous linear forms
(as the height $1$ and torus coordinates are fixed). If $q$ is the rank of their coefficient
matrix, Gaussian elimination gives independent forms
$\ell_1,\ldots,\ell_q$ with distinct pivot variables
$z_1,\ldots,z_q$. Hence
\[
\dim R/L=\dim R-q.
\]

Apply the same $q$ row operations to the corresponding equations
$r_\alpha=f_\alpha+Q_\alpha$. This gives elements
$g_1,\ldots,g_q\in I$. For a relation of $K$-height $h$, the form $f_\alpha$ is
linear in the height-$(h-1)$ block and $Q_\alpha$ is independent of that block.
Induction down the height blocks therefore gives
\[
g_i=z_i+(\text{nonpivot linear terms in the same block})
       +(\text{terms in lower-height blocks}),
\]
so $\operatorname{LT}(g_i)=z_i$. Let
$J=(g_1,\ldots,g_q)\subset I$. Since the leading monomials $z_i$ are distinct
variables, they are pairwise coprime, and Buchberger's product criterion shows
that the $g_i$ form a Gr\"obner basis for $J$. Consequently
\[
\dim R/J=\dim R-q=\dim R/L.
\]
Finally $J\subset I$, so additional equations, including equations produced by
dependency rows among the $f_\alpha$, can only decrease dimension. Thus
\[
\dim R/I\le \dim R/J=\dim R/L=\dim R/I^\cone,
\]
which is exactly the desired dimension inequality. No equality of initial
ideals and no Gr\"obner basis for the full ideal $I$ is required.
\end{proof}

\section{Example: $\mathrm{F}_4$}

In this paper, we use the following notation for roots:
for example,
if $\Delta=\{\delta_1, \cdots, \delta_4\}$
is the set of simple positive roots,
then we denote $\delta_1+2\delta_2+2\delta_3+2\delta_4$
by $1222$, and denote by $\delta_2+\delta_3+\delta_4$
by $0111$.
The root system for $\mathrm{F}_4$
is shown in Figure \ref{figure:D4F4}.

We have the following theorem:

\begin{thm}
\rm
If $\bG=F_4$, then
\[
\dim \cX_{\bB}[K]\langle\Psi\rangle^\cone\le
\dF\MINOREDITED{\#}\Phi^+
\]
for all fine configurations $(K,\Psi_0,\Psi_2)$.
As a consequence, $\cX_{\bB}$ is equidimensional
of dimension $\dF\MINOREDITED{\#}\Phi^+$.

(1)
If $F\neq \Q_p$,
then the irreducible components
of  $\cX_{\bB}$ are precisely the closure
of the strata
\[
\cX_{\bB}[U]\langle \Psi_0(C), \Psi_2(C)\rangle
\]
where $C\subset \cW_{\bB}$ is an irreducible component
and that
\begin{align*}
\Psi_2(C) &= \{\alpha\in\Phi^+| 
U_{\alpha^\vee}\text{~is pointwise the cyclotomic Galois character
over $C$}\}\\
\Psi_0(C) &= \{\alpha\in\Phi^+| 
U_{\alpha^\vee}\text{~is pointwise the trivial Galois character
over $C$}\}.
\end{align*}

(2)
If $F=\Q_p$,
then in addition to the irreducible components
$\overline{\cX_{\bB}[U]\langle \Psi_0(C), \Psi_2(C)\rangle}$,
$\cX_{\bB}$ have $4$ extra irreducible components,
listed below:

\begin{table}[H]
\centering
\scriptsize
\begin{tabular}{|c|c|c|c|}
\hline
ID & $\Phi^+-\Phi_K$ & $\Psi_0$ & $\Psi_2$ \\
\hline
I & $\{0010\}$ & $\{0010\}$ & $\{0001,0011,0100,0110,0120\}$ \\
II & $\{0010\}$ & $\{0010\}$ & $\{0001,0011,0100,0110,0120,1000\}$ \\
III & $\{1000,0010\}$ & $\{0010,1000\}$ & $\left\{ \begin{matrix} 0001,0011,0100,0110 \\ 0120,1100,1110,1120 \end{matrix} \right\}$ \\
IV & $\{1000,0001,0011,0010\}$ & $\{0001,0010,0011,1000\}$ & $\left\{ \begin{matrix} 0100,0110,0111,0120,0121,0122 \\ 1100,1110,1111,1120,1121,1122 \end{matrix} \right\}$ \\
\hline
\end{tabular}
\caption{Extra components for $\mathrm{F}_4$}
\label{table:F4-table}
\end{table}
In each of the cases listed above,
we have
$\Psi_2\subset \Delta_K$
and thus
$X_{K\rtimes \bT}\langle\Psi\rangle
=X_{K\rtimes \bT}\langle\Psi\rangle^\cone$
is smooth.
\label{thm:F4-table}
\end{thm}

\begin{proof}
Everything is explicitly computable.
There are $105$ possibilities for $K$.
The maximal cone estimate shows that
the inequality is automatic when $[F:\Q_p]\ge 2$.
For $F=\Q_p$, the $4862$ fine pairs give precisely
Configurations I--IV.
See Appendix~\ref{app:cone-model-algorithms}.
\end{proof}

\newpage
\phantomsection
\addcontentsline{toc}{part}{Part II: The combinatorics
and the twisted Weil-Deligne stacks}

\noindent
{\large Part II: The combinatorics and the twisted Weil-Deligne stacks}

\section{Poset graphs}

The set of positive roots $\Phi^+$
has the structure of directed graph:
two positive roots $\alpha, \beta$
are connected by a directed edge if
and only if
$\alpha-\beta\in \Delta$.

\begin{defn}
For each integer $h$,
write $\fB(\Phi^+, \Delta, h)$
for the induced subgraph of $\Phi^+$,
consisting of vertices of height
either $h$ or $h+1$.
So, $\fB(\Phi^+, \Delta, h)$
is a bipartite graph.
We will also call $\fB(\Phi^+, \Delta, h)$
the associated bipartite graph of height $h$.
\label{def:fB}
\end{defn}

\begin{example}
Root systems of type $A_n$, $B_n$, $C_n$, and $G_2$
have a very convenient feature
(see Figure \ref{figure:G2B4}):
their associated bipartite graphs in each height
is a tree!
In other words, have no loops or cycles.
\begin{figure}[h]
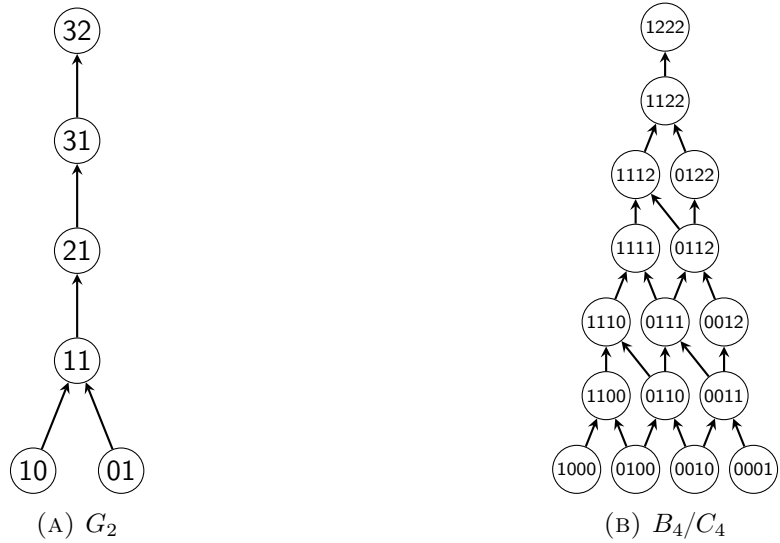

\begin{subfigure}{0.45\textwidth}
\centering
\tikzGtwo
\caption{$G_2$}
\end{subfigure}
\begin{subfigure}{0.45\textwidth}
\centering
\tikzBfour
\caption{$B_4/C_4$}
\end{subfigure}
\caption{Root poset for $G_2$, $B_4$ and $C_4$}
\label{figure:G2B4}
\end{figure}
\end{example}

\begin{example}
As is illustrated in Figure \ref{figure:D4F4},
we see that
\[
\fB(\Phi_{D_4}^+, \Delta, 2)
\cong
\fB(\Phi_{F_4}^+, \Delta, 3)
\]
are isomorphic bipartite graphs,
and they are both isomorphic to the cyclic graph $\vec C_6$.

\begin{figure}[h]
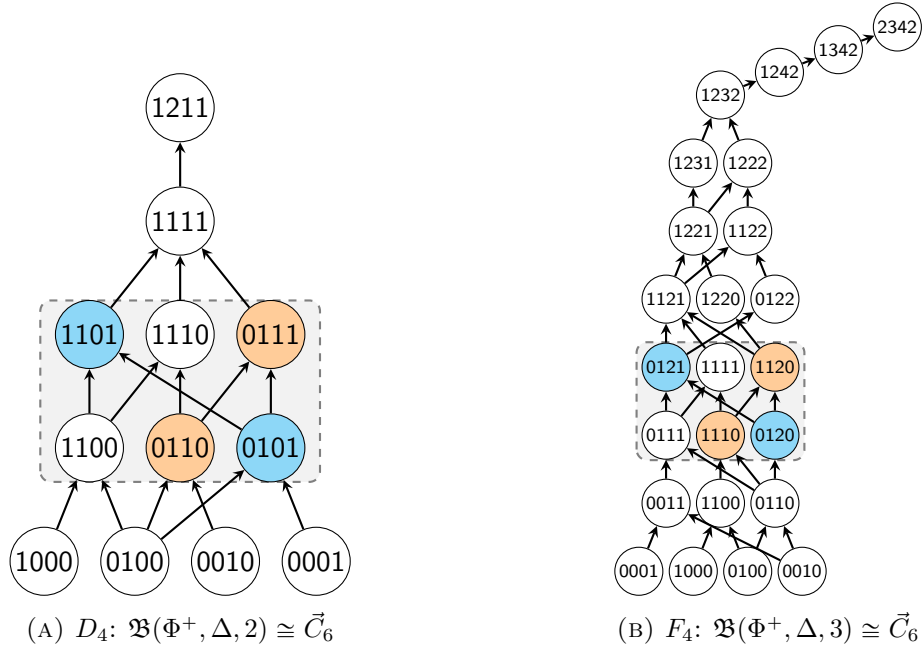

\begin{subfigure}{0.45\textwidth}
\centering
\tikzDfour
\caption{$D_4$: $\fB(\Phi^+, \Delta, 2)\cong \vec C_6$}
\end{subfigure}
\begin{subfigure}{0.45\textwidth}
\centering
\tikzFfour
\caption{$F_4$: $\fB(\Phi^+, \Delta, 3)\cong \vec C_6$}
\end{subfigure}
\caption{Root poset for $D_4$ and $F_4$}
\label{figure:D4F4}
\end{figure}
\end{example}

Indeed, we can give an exhaustive list of all cycles
in associated bipartite graphs:

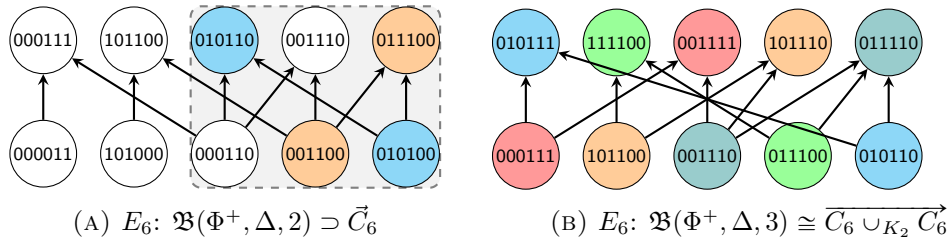
\begin{table}[H]
\centering
\begin{tabular}{c| l| l| l | l}
\hline
Height range & $F_4$ & $E_6$ & $E_7$ & $E_8$ \\
\hline
$[2,3]$& & $\overrightarrow{C_6}$ & $\overrightarrow{C_6}$ & $\overrightarrow{C_6}$ \\
$[3,4]$& $\overrightarrow{C_6}$ &$\overrightarrow{C_6 \cup_{K_2} C_6}$& $\overrightarrow{C_6 \cup_{K_2} C_6}$ & $\overrightarrow{C_6 \cup_{K_2} C_6}$\\
$[4,5]$& & &$\overrightarrow{C_6}$  &$\overrightarrow{C_6}$ \\
$[5,6]$ && &  & $\overrightarrow{C_6 \cup_{K_2} C_6 \cup_{K_2} C_6}$ \\
$[8,9]$ && &$\overrightarrow{C_6}$ &$\overrightarrow{C_6}$\\
$[9,10]$ && & & $\overrightarrow{C_6 \cup_{K_2} C_6}$\\
$[14,15]$ && & & $\overrightarrow{C_6}$\\
\hline
\end{tabular}
\caption{Cycles in associated bipartite graphs.}
\label{tab:layer-graphs}
\end{table}

\begin{figure}[H]
\begin{subfigure}{0.4\textwidth}
\centering
\tikzESixTwo
\caption{$E_6$: $\fB(\Phi^+, \Delta, 2)\supset \vec C_6$}
\label{figure:E6, layer 2}
\end{subfigure}
\begin{subfigure}{0.4\textwidth}
\tikzESixThree
\caption{$E_6$: $\fB(\Phi^+, \Delta, 3)\cong \overrightarrow{C_6 \cup_{K_2} C_6}$}
\label{figure:E6, layer 3}
\centering
\end{subfigure}
\end{figure}

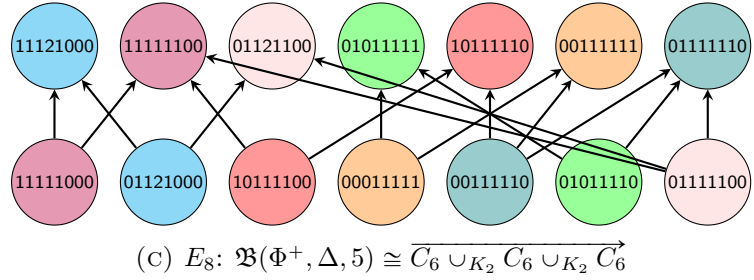
\begin{figure}[H]
\centering
\begin{subfigure}{.8\textwidth}
\centering
\setcounter{subfigure}{2}
\tikzEEightFive
\caption{$E_8$: $\fB(\Phi^+, \Delta, 5)\cong \overrightarrow{C_6 \cup_{K_2} C_6 \cup_{K_2} C_6}$}
\label{figure:E8, layer 5}
\end{subfigure}
\end{figure}

We remind the reader of the following very standard notion:

\begin{defn}
A {\it bipartite matching} for a bipartite graph,
is a collection of edges $\{\beta_i\to \alpha_i|i\in I\}$
such that $\{\alpha_i, \beta_i\}$
consists of exactly $2\#I$ vertices.
\end{defn}

The following notion is less standard.

\begin{defn}
Let $\fB$ be a directed graph.
We say $\fB$ is {\it contractible}
if there exists a vertex $v\in \fB$
such that there does not exist any edge $v'\to v$ pointing
to $v$,
and that there is exactly one edge
$v\to v'$ coming out of $v$,
and we say
$\fB-\{v, v'\}$ is a contraction of $\fB$.

We say a directed bipartite graph is a cycle
if it is connected, not contractible,
and non-empty.
\end{defn}

\begin{thm}
For every reduced root system, the bipartite graph
\[
\fB(\Phi^+,\Delta,1)
\]
contracts to isolated vertices.
\label{thm:height-two-contractible}
\end{thm}

\begin{proof}
The height-$2$ roots are the sums of adjacent simple roots. Thus
$\fB(\Phi^+,\Delta,1)$ is the barycentric subdivision of the underlying
Dynkin diagram, which is a tree. Successively remove its leaves.
\end{proof}

\begin{prop}
If $\Phi$ is a simple reduced root system,
then each $\fB(\Phi^+, \Delta, h)$
contracts to, up to isolated vertices, one of
\[
\emptyset,~ \vec C_6,~
\overrightarrow{C_6 \cup_{K_1}P_3},~
\overrightarrow{C_6 \cup_{K_2} C_6},~
\overrightarrow{C_6 \cup_{K_2} C_6 \cup_{K_2} C_6}.
\]
where $P_3$ is the chain
$\circ\leftarrow \circ \to \circ$
with $3$ vertices.
\label{prop:perfect-match}
\end{prop}

In the next subsection,
we deal with situations
where $\fB$ is not contractible to points.

\subsection{Admissible bipartite matching}

For each bipartite graph, we will attach to it an
algebraic variety that we call
the discriminant variety:

\begin{defn}
Let $\fB$ be a directed bipartite graph
with source vertices $\fB^{\src}\subset \Phi^+$
and target vertices $\fB^{\tar}\subset \Phi^+$.
Set
\[
\Delta_{\fB}=\{\alpha-\beta| \alpha\in \fB^{\tar},
\beta\in \fB^{\src}\}\cap \Phi^+,
\]
and write
$\G_m^{\Delta_\fB}=\Spec \bFp[x_\delta, x_\delta^{-1}|\delta\in \Delta_\fB]$.
Note that we don't require $\fB$ to be an induced subgraph
of $\Phi^+$,
and we only require that the vertices of $\fB$
can be interpreted as positive roots.

The {\it discriminant variety}
$\fD_{\fB}$ is defined as
the closed subvariety
of
\[
\G_m^{\Delta_\fB} \times \Spec \bFp [y_\beta, \beta\in \fB^{\src}]
\]
defined by the equations
\[
\sum_{\beta\to \alpha} N_{\beta, \alpha-\beta}
x_{\alpha-\beta}y_\beta=0
\]
for each $\alpha\in \fB^{\tar}$,
where $N_{**}\in \Z_{\neq 0}$ 
\MINOREDITED{are the standard structure constants}
defined in Eq. (\ref{eq:Nab}).

We say the graph $\fB$ is {\it admissible}
if $\fD_{\fB}\to \G_m^{\Delta_\fB}$ is smooth
of relative dimension $\#\fB^{\src}-\#\fB^{\tar}$.
\label{def:admissible}
\end{defn}

\begin{example}
Consider the cyclic graph $\vec C_6$:
\raisebox{-7ex}{\scalebox{.6}{\GraphCSix}}.
Write $\delta_1=\alpha_3-\beta_2$,
$\delta_2=\alpha_3-\beta_1$, and
$\delta_3=\alpha_2-\beta_1$.
If $\alpha_*$ are of height $h+1$ and $\beta_*$
are of height $h$,
then $\Delta_{\vec C_6} = \{\delta_1, \delta_2, \delta_3\}
\subset \Delta$.
Then the discriminant variety
is given by the equations
\begin{align*}
\begin{cases}
N_{\beta_2\delta_1}x_{\delta_1}y_{\beta_2} +
N_{\beta_1\delta_2}x_{\delta_2}y_{\beta_1}=&0 \\
N_{\beta_3\delta_2}x_{\delta_2}y_{\beta_3} +
N_{\beta_2\delta_3}x_{\delta_3}y_{\beta_2}=&0 \\
N_{\beta_1\delta_3}x_{\delta_3}y_{\beta_1} +
N_{\beta_3\delta_1}x_{\delta_1}y_{\beta_3}=&0,
\end{cases}
\end{align*}
whose
(relative) Jacobian matrix
is
$
\begin{pmatrix}
N_{\beta_1\delta_2}x_{\delta_2} &
N_{\beta_2\delta_1}x_{\delta_1} &
0 \\[6pt]
0 &
N_{\beta_2\delta_3}x_{\delta_3} &
N_{\beta_3\delta_2}x_{\delta_2} \\[6pt]
N_{\beta_1\delta_3}x_{\delta_3} &
0 &
N_{\beta_3\delta_1}x_{\delta_1}
\end{pmatrix}
$.
So,
$\vec C_6$ is admissible
if and only if
$
N_{\beta_2\delta_1}N_{\beta_3\delta_2}N_{\beta_1\delta_3}
+
N_{\beta_1\delta_2}N_{\beta_2\delta_3}N_{\beta_3\delta_1}
\neq 0.
$
Indeed, this value is \MINOREDITED{always} equal to one of
$\{2, -2, 3\}$
across all root systems,
and is non-zero if $p>3$.
\end{example}

\begin{thm}
All induced subgraphs $\fB\subset \fB(\Phi^+,\Delta,h)$ 
such that $\fB$ contains all height $h$ roots are
admissible.
\label{thm:C6}
\end{thm}

\begin{proof}
For type ABCG, $\fB(\Phi^+, \Delta, h)$
is always a chain, and thus $\fD_{\fB}\to \G_m^\Delta$
is not only smooth, but also a vector bundle;
by elimination of variables.
The remaining types follow from
Appendix~\ref{app:root-poset-cycles}.
\end{proof}

\begin{cor}
$\cX_{\bB}[U]\langle\Psi\rangle=
\dot \cX_{\bB}\langle\Psi\rangle
$
is always smooth
of dimension
\[
\dF\#\Phi^+-
(\dim \Span_{\Q}
\Psi_2
-\dim \Span_{\Q}(\Psi_2\cap \Delta)).
\]
Write $\cW_{\bB}[U]\langle\Psi\rangle\subset \cW_{\bB}$
for the image of
$\cX_{\bB}[U]\langle\Psi\rangle$,
then
$
\cX_{\bB}[U]\langle\Psi\rangle \to
\cW_{\bB}[U]\langle\Psi\rangle$
induces a bijection of irreducible components.
\label{cor:C6}
\end{cor}

\begin{proof}
The Jacobian matrix
\[
(\frac{\partial r_\alpha}{\partial x_{\beta,1}})
_{\alpha\in \Psi_{2,\Ht\ge 2}, \beta\in \Phi^+}
\]
is of full rank by 
Theorem 
\ref{thm:C6}.
\MINOREDITED{The full-rank Jacobian gives smoothness by \cite[Tag 01V9]{Stacks}.}
\MAJOREDIT{On the stratum under consideration every $x_{\delta,1}$ for
$\delta\in\Delta$ is invertible, and the full-rank Jacobian minor is a monomial
in these variables. Ordered by height, each cone equation is therefore linear
in a new variable with unit coefficient. Successive elimination identifies
every fiber over $\cW_{\bB}[U]\langle\Psi\rangle$ with an affine space of
constant dimension. Counting the eliminated variables gives the displayed
dimension formula. The morphism is consequently smooth and surjective with
geometrically irreducible fibers, and hence induces a bijection on irreducible
components.}
\end{proof}

\begin{cor}
For each irreducible component
$C\subset \cW_{\bB}$,
$\cX_{\bB}(C)$
contains the top-dimensional irreducible component
$\overline{\dot\cX_{\bB}\langle \Psi_0(C), \Psi_2(C)\rangle}$.

On the other hand,
If $(\Psi_0,\Psi_2)$ is not of the form 
$(\Psi_0(C), \Psi_2(C))$,
then
$\cX_{\bB}[U]\langle\Psi\rangle$
is nowhere dense in $\cX_{\bB}$.
\label{cor:reg}
\end{cor}

\begin{proof}
Combine Corollary 
\ref{cor:C6}
and \MINOREDITED{Corollary} \ref{cor:XB-dim}.
\end{proof}

\begin{thm}
\rm
If $\dF\ge \#\Phi^+ + 2$,
then
each $\cX_{\bB}(C)$
is irreducible.
As a consequence,
$\cX_{\bB}\to \cW_{\bB}$
is generically smooth and induces a bijection
of irreducible components.
\label{thm:reg}
\end{thm}

\begin{proof}
We always have the trivial upper bound
\[
\dim X_{K\rtimes \bB}\langle \Psi\rangle
\le (\dF+\dkF+1)\#\Phi_K+\dim \bT
\]
by Theorem \ref{thm:eq}.
So, if $\dF\ge \#\Phi^+ + 2$
and $K\neq U$,
then $\dim \cX_{\bB}[K]\langle\Psi\rangle
< \dF\#\Phi^+$.
The theorem now follows from Corollary \ref{cor:reg}.
\end{proof}

There are only finitely many $F/\Q_p$ cases
left, thanks to
Theorem \ref{thm:reg}.

\subsection{$K$-restricted posets}

Recall that $K\subset U$
is a $\bT$-stable normal subgroup.

\begin{defn}
We attach to $\Phi_K$
the structure of a directed graph:
there is a directed edge
$\beta\to \alpha$
if and only if
$\alpha-\beta\in \Delta_K$.

Similar to  Definition \ref{def:fB},
we write
$\fB(\Phi_K, \Delta_K, h)$
for the induced subgraph of $\Phi_K$
consisting of vertices of $K$-height
either $h$ or $h+1$.
\end{defn}

\begin{example}
Consider $\bG=F_4$, $\Phi_K=\Phi^+-\{0010, 0001, 0011\}$.
The root poset structure of $\Phi_K$
is illustrated below
(each row has roots of the same $K$-height,
and two vertices are connected with a directed edge
if their difference is a root of $\Delta_K$):

\begin{figure}[H]
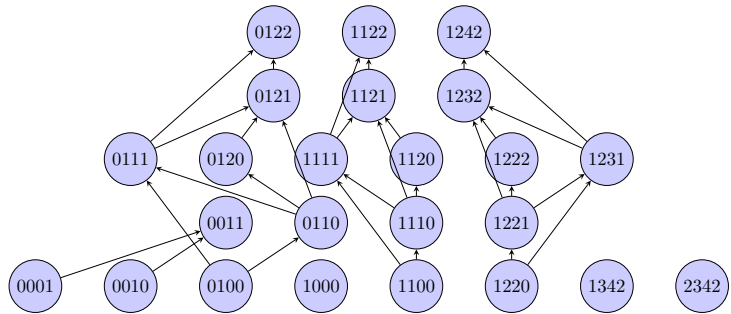

\scriptsize
\scalebox{.7}{
\FKfull
}
\caption{The full root poset}
\label{figure:F4-K-example}
\end{figure}

\begin{figure}[H]
\scriptsize
\scalebox{.7}{\FKfour}
\caption{The graph $\fB(\Phi_K, \Delta_K, 4)$}
\end{figure}

\begin{figure}[H]
\scriptsize
\scalebox{.7}{\FKthree}
\caption{The graph $\fB(\Phi_K, \Delta_K, 3)$}
\end{figure}

\begin{figure}[H]
\scriptsize
\scalebox{.7}{\FKtwo}
\caption{The graph $\fB(\Phi_K, \Delta_K, 2)$}
\label{figure:F4-K-example-2}
\end{figure}

\begin{figure}[H]
\scriptsize
\scalebox{.7}{\FKone}
\caption{The graph $\fB(\Phi_K, \Delta_K, 1)$}
\end{figure}
\end{example}

As is illustrated in Figure \ref{figure:F4-K-example},
the poset graph $\Phi_K$
may have directed edges
connecting non-adjacent $K$-heights.
In other words,
$K$-height is not additive.
In this paper,
edges that skips layers
will not be relevant
to us anyway,
and we may as well redefine
the directed tree structure
as $\beta\to \alpha$
if and only if
$\Ht_K(\alpha)-\Ht_K(\beta)=1$
and $\alpha-\beta\in \Delta_K$.
All that matter are the bipartite graphs.

\section{Twisted Borels and relative root posets}

We recall the structure theorem for mod $p$ Galois representations
established in our earlier work:

\begin{thm} (\cite[Theorem 4]{Lin22})
\rm
If $\bar\rho:\Gal_F\to \bG(\bFp)$
is a semisimple Galois representation,
then, up to conjugation, there
is a Levi subgroup $\bM\subset \bG$
with a maximal torus $\bT\subset \bM$
such that
\begin{itemize}
\item the connected centralizer of $\bar\rho(I_F)$
in $\bM$ is equal to $\bT$,

\item there is an {\it elliptic Weyl group element}
$w\in N_{\bM}(\bT)/\bT$ such that
$\bar\rho$ factors through $\bT(\bFp) \langle w\rangle
:=\langle\bT(\bFp), w\rangle
\subset \bM$.
\end{itemize}
Once a Levi $\bM$ is chosen, the maximal torus
$\bT\subset \bM$ is uniquely determined.
\end{thm}

\begin{proof}
The only part not explicitly mentioned in \cite[Theorem 4]{Lin22}
is the ellipticity of $w$,
which is a consequence of the uniqueness
of $\bT$.
\end{proof}

\begin{remark}
In other words,
the image of $\bar\rho$ is generated by two elements
$s\in \bT(\bFp)$ and $t\in N_{\bM}(\bT)(\bFp)$
such that
\begin{itemize}
\item $s$ is a {\it regular semisimple element} of $\bM$, and
\item the action of $w$ on the root lattice of $\bM$ does not have eigenvalue $1$,
\item $t s t^{-1}=s^q$, where $q$ is the size of the residue field of $F$.
\end{itemize}
Here, $s$ is the image of a generator of the tame quotient
of $I_F$,
and $t$ is the image of a Frobenius element.

Note that we embed
$N_{\bM}(\bT)/\bT\hookrightarrow \bM$
through the canonical Tits representatives
(which does not need to be a group homomorphism);
indeed, the Tits representatives \MINOREDITED{generate}
a group of order 
$2^{k_0}\#(N_{\bM}(\bT)/\bT)$ for some integer $k_0$.
The order of $w$ in this paper is interpreted as
the order of the Tits representative of $w$
(we remark that it is a standard fact that the order of
$w$ in $\bM$ and the order of $w$ in $N_{\bM}(\bT)/\bT$
are either equal, or differ by a factor of $2$).

We don't necessarily have
$\bT\langle w\rangle\cong \bT\rtimes \langle w \rangle$.
\end{remark}

\begin{setup}
\label{setup:w}
Let $\bM\subset \bG$ be a proper Levi.
Let $w\in N_{\bM}(\bT)$ be an elliptic Weyl group element.

Let $U_{\bM}:=\bB\cap \bM$, which is the unipotent radical
of the Borel $\bB_{\bM}=\bB\cap \bM$.

Let $U^{\bM}:=\ker(U\to U_{\bM})$,
so that $U^{\bM}\rtimes \bM$ is the parabolic subgroup
for $\bM$.

Write $\Delta_{\bM}\subset \Delta$
for the simple positive roots
of $\bM$.

Write $\Phi_{\bM}$ for the roots of $\bM$.

Write $\Phi^{\bM+}\subset \Phi^+$
for the positive roots
that appear in $U^{\bM}$.

Let $\bar\rho^\SS:\Gal_F\to\bT(\bFp) \langle w\rangle$
be an elliptic Galois representation.
\end{setup}

\begin{defn}
A {\it $w$-relative root}
$\Balpha=\langle w\rangle \alpha$
is a $w$-orbit of roots
$\{\alpha, w \alpha, w^2\alpha, \ldots\}$.
The {\it $w$-relative root group}
is the product
\[
U_{\Balpha}=
U_{\langle w\rangle \alpha}
:=
\prod U_{w^n\alpha^{\vee}}.
\]
\end{defn}

The structure theorem imposes
the following very strong constraint:

\begin{lem}
\rm
Set
$
\Balpha-\Balpha=\{\alpha-\alpha'|\alpha,\alpha'\in \Balpha\}
$.
If $H^2(\Gal_F, U_{\Balpha}(\bFp))\neq 0$,
then
$(\Balpha-\Balpha)\cap \Phi_{\bM}=\{0\}$.
\end{lem}

\begin{proof}
Write $m$ for $\#\Balpha$
and let $F_m/F$ be the unramified extension of degree $m$.
Then as a $\Gal_{F_m}$-module,
$U_{\Balpha}(\bFp)$
is a direct sum of Galois characters
$\chi, \chi^q, \cdots, \chi^{q^{m-1}}$
where $\chi:\Gal_{F_m}\to \bFp^\times$
is the adjoint action on $U_{\alpha^\vee}$.
If
\[
H^2(\Gal_F, U_{\Balpha}(\bFp))
\cong H^0(\Gal_F, U_{\Balpha}(\bFp)^\vee(1))^\vee
\neq 0,
\]
then at least one of
$\{\chi, \chi^q, \cdots, \chi^{q^{m-1}}\}$
must be the cyclotomic character,
and thus they must all be the cyclotomic character.
This implies each quotient
$\chi/\chi^{q^k}$ is the trivial character.
If $\Balpha-\Balpha$ contains a nontrivial root
$\beta$
of $\Phi_{\bM}$,
then the action of $\bar\rho^\SS(I_F)$
on the root group $U_{\beta^\vee}$ is trivial,
which contradicts the setup
that $\bar\rho$ is an elliptic Galois representation
(which implies $\bar\rho(I_F)\subset \bT$
contains a regular semisimple element of $\bM$).
\end{proof}

\begin{example}
We first consider the case where $\bG=\GL_5$
and that $\bM=\GL_2\times \GL_3$:

\blockFive

\noindent
Say $\bar\rho^\SS=\bar\rho_1\times \bar\rho_2$
where $\bar\rho_1:\Gal_F\to \GL_2(\bFp)$
and $\bar\rho_2:\Gal_F\to \GL_3(\bFp)$
are irreducible Galois representations.
We have
$\Hom_{\Gal_F}(\bar\rho_1(1), \bar\rho_2)=0$
as they are irreducible Galois representations
of different rank.
The single six-element $w$-orbit is drawn above.
Note that $1111-1100=0011$, which is a root
for $\GL_3\subset\bM$.

On the other hand, consider $\bG=\GL_4$:

\begin{tikzpicture}[
    x=8mm,y=8mm,
    >=Stealth,
    every node/.style={font=\scriptsize}
]

\draw[step=1] (0,0) grid (4,-4);

\fill[gray!55] (0,0) rectangle (2,-2);
\fill[gray!55] (2,-2) rectangle (4,-4);

\foreach \i in {1,...,4}{
  \node[left]  at (0,-\i+0.5) {$\i$};
  \node[above] at (\i-0.5,0) {$\i$};
}


\node[text=red]  (r13) at (2.5,-0.5) {110};
\node[text=red]  (r24) at (3.5,-1.5) {011};

\node[text=blue] (b14) at (3.5,-0.5) {111};
\node[text=blue] (b23) at (2.5,-1.5) {010};

\draw[->,red,thick,bend left=18] (r13) to (r24);
\draw[->,red,thick,bend left=18] (r24) to (r13);

\draw[->,blue,thick,bend left=18] (b14) to (b23);
\draw[->,blue,thick,bend left=18] (b23) to (b14);

\end{tikzpicture}

\noindent
There are two $w$-orbits:
$\{110,011\}$ and $\{111,010\}$.
The difference $110-011$ is not even a root.
Neither is the difference $111-010=101$.
This shows that both
\[
H^2(\Gal_F, U_{\langle w\rangle 110}(\bFp))
\text{~and~}
H^2(\Gal_F, U_{\langle w\rangle 010}(\bFp))
\]
can be non-trivial.
Nevertheless, they cannot be \MINOREDITED{non-trivial} {\it simultaneously},
as $110-010=100$ is a root of $\bM$.
This observation can also be confirmed by
local Tate duality --
nonvanishing of $H^2(\Gal_F, U^{\bM}(\bFp))$
\MINOREDITED{implies} $\bar\rho_1(1)\cong \bar\rho_2$,
and by Schur's lemma, we have
$\Hom_{\Gal_F}(\bar\rho_1(1), \bar\rho_2)\cong\bFp$.
\end{example}

\begin{lem}
\rm
(1) If $H^2(\Gal_F, U_{\Balpha}(\bFp))$
and $H^2(\Gal_F, U_{\Balpha'}(\bFp))$
are both nontrivial,
then
\[
((\Balpha\cup\Balpha')-(\Balpha\cup\Balpha'))
\cap \Phi_{\bM}=\{0\}.
\]

(2) If $H^i(\Gal_F, U_{\Balpha}(\bFp))\neq 0$
for $i=0,2$,
then it is equal to $\bFp$.
\end{lem}

\begin{proof}
(1) Clear.
(2) It follows from Frobenius reciprocity.
\end{proof}

\begin{defn}
Let $\alpha=\sum_{\delta\in \Delta}n_\delta\delta\in \Phi^+$.
Define
\begin{align*}
\Ht^{\bM}(\alpha)
:=\sum_{\delta\not\in \Delta_{\bM}}n_\delta,\text{~and}
\Ht_{\bM}(\alpha)
:=\sum_{\delta\in \Delta_{\bM}}n_\delta.
\end{align*}
\end{defn}

The descending central filtration for $U^{\bM}$
coincides with the $\Ht^{\bM}$ filtration:
if we write
\[
\Lie U^{\bM}_k
=
\begin{cases}
\Lie U^{\bM} & k=1 \\
[\Lie U^{\bM},\Lie U^{\bM}_{k-1}] & k>1,
\end{cases}
\]
then
\[
\bar U^{\bM}_k:=U^{\bM}_k/U^{\bM}_{\MINOREDITED{k+1}}
=\prod_{\alpha\in \MINOREDITED{\Phi^{\bM+}}, \Ht^{\bM}(\alpha)=k}
U_{\alpha^{\vee}}.
\]

Since $w$ preserves the height $\Ht^{\bM}$,
each $U^{\bM}_{k}$
is a product
of $w$-relative root groups.

\subsection{Relative root posets}
Recall that $\Phi^+$
is a directed graph,
whose directed edges $\beta\to \alpha$
are precisely positive roots
$\alpha, \beta$ such that $\alpha-\beta\in \Delta$.
We can similarly equip $\Phi^{\bM+}$
with a directed graph structure
whose directed edges
$\beta\to \alpha$
are positive roots $\alpha, \beta$
satisfying $\alpha-\beta\in \Phi^{\bM+}$
and $\Ht^{\bM}(\alpha-\beta)=1$.

\begin{defn}
The {\it $w$-root poset}
\[
\BPhi^+
=\BPhi^{\bM+}:=\Phi^{\bM+}/w
\]
is by definition
the quotient graph of $\Phi^{\bM+}$
by the cyclic group action of $\{1, w, w^2, \cdots\}$.
Write
\[
\BDelta=\BDelta^{\bM}
:=\{\Balpha\in \BPhi^+|\Ht^{\bM}(\Balpha)=1\},
\]
and call it the {\it $w$-relative simple roots}.

Denote by $\fB(\BPhi^+, \BDelta, h)\subset \BPhi^+$
for the induced bipartite subgraph consisting
of $w$-relative roots $\Balpha$ with $\Ht^{\bM}(\Balpha)\in\{h,h+1\}$.

We say a $w$-relative root $\Balpha$
is {\it non-obstructing}
if $(\Balpha-\Balpha)\cap \Phi_{\bM}\neq \{0\}$.
We say a set $S$ of $w$-relative root
is {\it non-obstructing}
if $(\bigcup_{\Balpha\in S}\Balpha-\bigcup_{\Balpha\in S}\Balpha)\cap \Phi_{\bM}\neq \{0\}$.

We say \MINOREDITED{an} induced bipartite subgraph
$\fB\subset \BPhi^+$
is {\it non-obstructing}
if the set of target vertices
$\fB^{\tar}$ is non-obstructing.
\label{def:non-obstructing}
\end{defn}

\begin{defn}
Let $\fB\subset \BPhi^+$ be an induced bipartite subgraph.
The {\it associated absolute bipartite graph}
$\fB_{\abs}$ is by definition the directed \MINOREDITED{bipartite} graph such that
\begin{itemize}
\item the source vertices of $\fB_{\abs}$ are the absolute roots
contained in the $w$-relative roots
that are source vertices of $\fB$,
\item the target vertices $\fB_{\abs}$ are the absolute roots
contained in the $w$-relative roots
that are target vertices of $\fB$,
\item the graph structure of $\fB_{\abs}$
is given by the induced subgraph structure of $\Phi^+$.
\end{itemize}
\end{defn}

\begin{thm}
\rm
(1) If $\bG=F_4$, then all cycles in
$\fB(\BPhi^+, \BDelta, h)$ are non-obstructing.

(2) If $\bG=E_6, E_7$ and $E_8$,
then all cycles $\fC$ in
$\fB(\BPhi^+, \BDelta, h)$ \MINOREDITED{are} either non-obstructing
or have admissible associated absolute bipartite graph
$\fC_{\abs}$.
\label{thm:twist-C6}
\end{thm}

\begin{proof}
This follows from Appendix~\ref{app:root-poset-cycles}.
\end{proof}

\begin{example}
($w$-root poset graphs for $F_4$)
We will only show the rank $1$ Levi cases here:
there are four rank $1$ Levi's, corresponding to the four
simple roots;
and each rank $1$ Levi has a unique elliptic Weyl group element.
The non-obstructing $w$-relative roots are colored by gray.

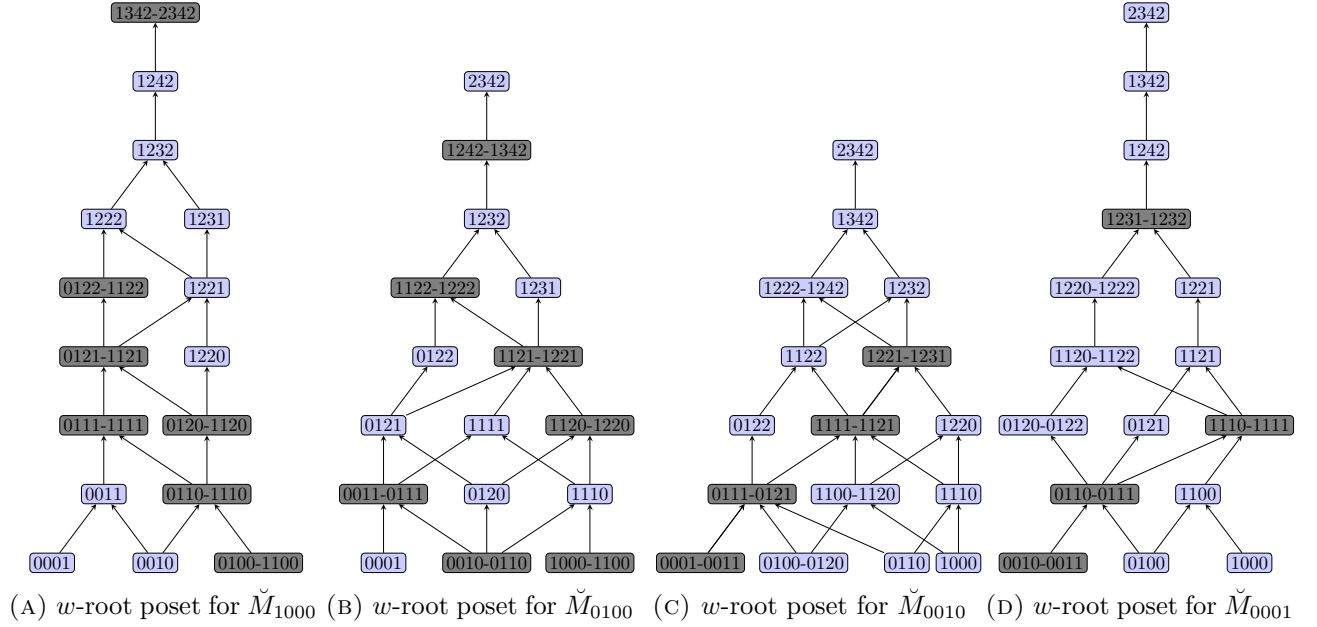
\begin{figure}[H]
\begin{subfigure}{0.24\textwidth}
\centering
\scalebox{.7}{
\FwPosetA
}
\caption{$w$-root poset for $\bM_{1000}$}
\end{subfigure}
\begin{subfigure}{0.24\textwidth}
\centering
\scalebox{.7}{
\FwPosetB
}
\caption{$w$-root poset for $\bM_{0100}$}
\end{subfigure}
\begin{subfigure}{0.26\textwidth}
\centering
\scalebox{.7}{
\FwPosetC
}
\caption{$w$-root poset for $\bM_{0010}$}
\end{subfigure}
\begin{subfigure}{0.24\textwidth}
\centering
\scalebox{.7}{
\FwPosetD
}
\caption{$w$-root poset for $\bM_{0001}$}
\end{subfigure}
\end{figure}
\end{example}

\section{Unramified descent}
Let $F_n/F$ be the
unramified extension
of $p$-adic fields
of degree $n$ coprime to $p$.
Then $\bfE_{F_n}/\bfE_F$ is
also an unramified extension.
As a consequence,
$\bfE_{F_n}/\bfE_F$ is a
finite \'etale extension.
Let $A$ be an $\bFp$-algebra.

\begin{lem}
$\bfE_{F_n}/\bfE_{F}$
is necessarily an unramified extension of degree $n$.
\end{lem}

\begin{proof}
Recall that $\Gamma_F=\Gal_{F(\mu_{p^\infty})/F}
\cong \Z_p \times \Delta_F$,
and 
$\bfE_{F}= (k_{F(\mu_{p^\infty})}(\!(T_{F}')\!))^{\Delta_F}
=
k_{F(\mu_{p^\infty})^{\Delta_F}}(\!(T_{F})\!)
$.
Here, $k_F$ is the residual field of $F$.
Note that
$F(\mu_{p^\infty})^{\Delta_F}/F$ is a $p$-power extension,
and thus
$F_n(\mu_{p^\infty})^{\Delta_F}/F(\mu_{p^\infty})^{\Delta_F}$
is unramified of degree $n$.
\end{proof}

Let $M$ be an \'etale
$(\varphi, \Gamma)$-module
over $\bfE_{F,A}$.
Then $M\otimes_{\bfE_{F,A}}\bfE_{F_n,A}$
is an \'etale
$(\varphi, \Gamma)$-module
over $\bfE_{F_n,A}$,
together with Galois descent data:

\begin{prop}
Write $\sigma$ for a generator
of $\Gal(\bfE_{F_n}/\bfE_F)$.
There is an equivalence
of categories
\[
\{\text{\'etale }
(\varphi,
\Gamma)\text{-modules }
(M,
\phi_M,
\gamma_M)
\text{ over }
\bfE_{F, A}\}
\cong
\left\{
\begin{matrix}
\text{\'etale }
(\varphi,
\Gamma)\text{-modules }
(M',
\phi_{M'},
\gamma_{M'})
\text{ over }
\bfE_{F_n, A}
\\
\text{together with }
\sigma_{M'}:
\sigma^*M'\xrightarrow{\cong}M'
\\
\text{that commutes with 
$\phi_{M'}$
and $\gamma_{M'}$, and satisfies the cyclic cocycle condition}
\end{matrix}
\right\}
\]
\label{prop:udescent}
\end{prop}

\begin{proof}
Because $\bfE_{F_n}/\bfE_F$
is a finite \'etale Galois
extension with Galois group
$\Gal(\bfE_{F_n}/\bfE_F)$,
standard Galois
descent establishes
the equivalence at the level
of underlying modules.
Furthermore,
since $F_n/F$ is unramified,
the action of
$\Gal(\bfE_{F_n}/\bfE_F)$
commutes with both
the absolute Frobenius $\varphi$
and the cyclotomic Galois group
$\Gamma$.
Thus,
the descent datum naturally
respects the
$(\varphi, \Gamma)$-structure.
\end{proof}

\begin{cor}
Let $M$ be an \'etale
$(\varphi, \Gamma)$-module
over $\bfE_{F,A}$.
Write
$M':=M\otimes_{\bfE_F}\bfE_{F_n}$.

(1)
The map $c\mapsto c\otimes 1$
induces an embedding
\[
Z^1_{\Herr}(M)\hookrightarrow
Z^1_{\Herr}(M').
\]

(2)
$\sigma_{M'}$ induces a
semilinear action of
$\Gal(\bfE_{F_n}/\bfE_F)$
on $Z^1_{\Herr}(M')$
such that $Z^1_{\Herr}(M)$
is identified with
the fixed points.

(3)
The trace
$\tr_{\sigma_{M'}} =
\frac{1}{[\bfE_{F_n}:\bfE_F]}
\sum_{i} \sigma_{M'}^i$
induces a section
\[
Z^1_{\Herr}(M') \to Z^1_{\Herr}(M)
\]
of the embedding in part (1).
\label{cor:udescent}
\end{cor}

\begin{proof}
We have $C_{\Herr}^\bullet(M') =
C_{\Herr}^\bullet(M)
\otimes_{\bfE_F} \bfE_{F_n}$.
The map $c \mapsto c \otimes 1$
respects $d^1_\Herr$,
proving (1).
Because the semilinear action
of $\sigma_{M'}$ on $M'$
commutes with $\varphi$ and
$\Gamma$,
it extends to the Herr complex.
The fixed points of
$Z^1_{\Herr}(M')$
under this finite group
action recover exactly
$Z^1_{\Herr}(M)$,
establishing (2).
Finally,
because the degree
$[\bfE_{F_n}:\bfE_F]$
divides $n$,
which is coprime to $p$,
it is invertible in the
coefficient field $\bar \F_p$,
making the trace map in (3)
a well-defined idempotent
projecting $Z^1_{\Herr}(M')$
onto $Z^1_{\Herr}(M)$.
\end{proof}

\begin{remark}
Since $\Gal_F\to \Gal(F_n/F)$
has no (tautological)
splitting,
we do not have an analogue
of Corollary \ref{cor:udescent}
for Galois cohomology.
The action of $\Gal(F_n/F)$
on
$H^1(\Gal_{F_n},
\Res_{F_n/F}V)$
does not lift to an action
of $\Gal(F_n/F)$
on cocycles
$Z^1(\Gal_{F_n},
\Res_{F_n/F}V)$.

The non-liftability of the
$\Gal(F_n/F)$-action
to cocycles indeed carries
over to the setup of
Herr complex cohomology.
There is no \textit{linear}
action of $\Gal(F_n/F)$
on $Z^1_{\Herr}(M')$
lifting the action on
$H^1_{\Herr}(M')$:
for any (hypothetical) linear action
$\tilde{\sigma}$,
the power $\tilde{\sigma}^n$
will never be the identity map,
rather
$\tilde{\sigma}^n - \mathrm{id}$
is a nontrivial coboundary.
If one systematically
corrects this coboundary,
one will always end up with
a semilinear action.
Such a correction is
impossible to do for
Galois cocycles,
as the coefficient module $V$
is not large enough
to absorb the twist.

One of the advantages of
the Herr complex is that
it enables us to lift linear
(inflation-restriction)
Galois actions on cohomology
groups to semilinear actions
on cocycle groups.
\end{remark}

\begin{thm}
\rm
Let $\bar\chi:\Gal_F\to\bar{\F}_p^\times$ be a Galois character, and let
$\bar\chi_n$ denote its restriction to $\Gal_{F_n}$. Let $\bar\chi_{\cyc}$
denote the mod $p$ cyclotomic character. The full cohomology group $H^1(\Gal_{F_n},\bar\chi_n)$ decomposes
as an $\bar{\F}_p[\Gal(F_n/F)]$-module as follows:
\[
H^1(\Gal_{F_n},\bar\chi_n)
\cong
\begin{cases}
\bar{\F}_p[\Gal(F_n/F)]^{\oplus [F:\Q_p]},
& \bar\chi_n\not\cong\bar{\F}_p,\ \bar{\F}_p(1),
\\[4pt]
\bar{\F}_p[\Gal(F_n/F)]^{\oplus [F:\Q_p]} \oplus \bar\chi,
& \bar\chi_n\cong\bar{\F}_p,\ \bar\chi_n\not\cong\bar{\F}_p(1),
\\[4pt]
\bar{\F}_p[\Gal(F_n/F)]^{\oplus [F:\Q_p]} \oplus 
(\bar\chi\otimes\bar\chi_{\cyc}^{-1}),
& \bar\chi_n\cong\bar{\F}_p(1),\ \bar\chi_n\not\cong\bar{\F}_p,
\\[4pt]
\bar{\F}_p[\Gal(F_n/F)]^{\oplus [F:\Q_p]} \oplus \bar\chi \oplus 
(\bar\chi\otimes\bar\chi_{\cyc}^{-1}),
& \bar\chi_n\cong\bar{\F}_p \cong \bar{\F}_p(1).
\end{cases}
\]
\label{thm:h1-structure}
\end{thm}

\begin{proof}
Since the order of $\Gal(F_n/F)$ is prime to $p$, the group algebra
$\bar{\F}_p[\Gal(F_n/F)]$ is split semisimple. A representation over this
algebra is entirely determined by the dimensions of its isotypic
components.

Let $\vartheta: \Gal(F_n/F) \rightarrow\bar{\F}_p^\times$ be a character,
which we inflate to $\Gal_F$. Because $H^i(\Gal(F_n/F), -) = 0$ for $i > 0$,
the inflation--restriction exact sequence yields an isomorphism of the
$\vartheta$-isotypic component:
\[
H^1(\Gal_{F_n},\bar\chi_n)_\vartheta
\cong
H^1(\Gal_F,\bar\chi\otimes\vartheta^{-1}).
\]

The local Euler characteristic formula states that
\[
\dim H^1(\Gal_F,\bar\chi\otimes\vartheta^{-1})
=
[F:\Q_p]
+
\dim H^0(\Gal_F,\bar\chi\otimes\vartheta^{-1})
+
\dim H^0(\Gal_F,(\bar\chi\otimes\vartheta^{-1})^\vee(1)).
\]
Because the character $\bar\chi\otimes\vartheta^{-1}$ is one-dimensional,
each $H^0$-term is either $0$ or $1$. Specifically,
$\dim H^0(\Gal_F,\bar\chi\otimes\vartheta^{-1}) = 1$ if and only if
$\bar\chi\otimes\vartheta^{-1} \cong \bar{\F}_p$, which occurs exactly
when $\vartheta = \bar\chi$. Note that this case is possible if and only
if $\bar\chi_n \cong \bar{\F}_p$.

Similarly, $\dim H^0(\Gal_F,(\bar\chi\otimes\vartheta^{-1})^\vee(1)) = 1$
if and only if $\bar\chi\otimes\vartheta^{-1} \cong \bar{\F}_p(1)$, meaning
$\vartheta = \bar\chi\otimes\bar\chi_{\cyc}^{-1}$. This case is possible
if and only if $\bar\chi_n \cong \bar{\F}_p(1)$.

For any character $\vartheta$ not meeting these conditions, the $H^0$
terms vanish, giving a baseline dimension of $[F:\Q_p]$. The regular
representation $\bar{\F}_p[\Gal(F_n/F)]$ contains exactly one copy of
every character $\vartheta$. Thus, the uniform dimension of $[F:\Q_p]$
across all isotypic components contributes precisely
$\bar{\F}_p[\Gal(F_n/F)]^{\oplus [F:\Q_p]}$ to the overall structure.

The only deviations from this regular baseline are the extra dimensions
arising from non-zero $H^0$ terms. If $\bar\chi_n \cong \bar{\F}_p$, the
cohomology group gains one extra copy of the character $\vartheta = \bar\chi$.
If $\bar\chi_n \cong \bar{\F}_p(1)$, it gains one extra copy of the
character $\vartheta = \bar\chi\otimes\bar\chi_{\cyc}^{-1}$. Summing the
baseline regular representation and these specific anomalous
one-dimensional pieces yields the four cases.
\end{proof}

\begin{cor}
Write $M'=M'_{\bar\chi_n}$ for the rank $1$ \'etale
$(\varphi, \Gamma)$-module
over $\bfE_{F_n,\bFp[t^{\pm 1}]}$
that is the universal unramified twist
of $\bar\chi_n$.
Then, we have
\[
H^1_{\Herr}(M')
\cong
\begin{cases}
\bFp[t^{\pm1}][\Gal(F_n/F)]^{\oplus [F:\Q_p]}
& \bar\chi_n\neq \bFp \\
\bFp[t^{\pm1}][\Gal(F_n/F)]^{\oplus [F:\Q_p]}
\oplus \bar\chi
& \bar\chi_n= \bFp.
\end{cases}
\]
So, 
\[
H^1_{\Herr,+}(M')
\cong
\bFp[t^{\pm1}][\Gal(F_n/F)]^{\oplus [F:\Q_p]}
\oplus 
(\bFp[t^{\pm1}]\bar\chi \otimes_{\F_p} k_{F_{n,\cyc}})
\]
as a $\bFp[t^{\pm 1}][\Gal(F_n/F)]$-module
in both cases.
\end{cor}

\begin{proof}
By Lemma \ref{lem:PID-Tor},
$H^1_{\Herr}(M')$
does not have the cyclotomic terms,
compared to Theorem \ref{thm:h1-structure}.
When $\bar\chi_n$ is trivial,
there is the unramified
cocycle $(1,0)\in Z^1_{\Herr}(M')$ (c.f. Subsection
\ref{subsec:unramified}),
which is a coboundary over the locus
$(t\neq 1)$.
Quotienting out by $[(1,0)]$,
we obtain a finite free
$\bFp[t^{\pm1}]$-module
whose isotypic parts
have the same rank by Theorem \ref{thm:h1-structure}.

Write $H^1_{\Herr}(M')_{\vartheta}$
for the $\vartheta$-isotypic part
of the (maximal) finite free quotient
$H^1_{\Herr, \text{tf}}(M')$
of $H^1_{\Herr}(M')$.
Choose a basis
$\{v_{k, \vartheta}\}_{k=1}^{[F:\Q_p]}$
for each isotypic part.
Set
\[
c'_k:=\sum_{\vartheta} v_{k, \vartheta},~
k=1,\dots,[F:\Q_p].
\]
Note that we have the projection operators
$$e_\vartheta = \frac{1}{\MINOREDITED{n}} \sum_{\sigma \in \Gal(F_n/F)} \vartheta(\sigma)^{-1} \, \sigma$$
satisfying
\[
e_\vartheta(c'_k) = v_{k, \vartheta}.
\]
The upshot is that
\[
H^1_{\Herr,\text{tf}}(M')=
\bigoplus_{k=1}^{[F:\Q_p]}
\bFp[t^{\pm1}][\Gal(F_n/F)]c'_k.
\qedhere
\]
\end{proof}

Note that
\[
k_{F_n,\cyc}\cong\F_p[\Gal(k_{F_{n,\cyc}}/\F_p)]
\cong\F_p[\Gal(F_n/F)]^{\oplus \dkF}
\]
by the normal basis theorem.
Thus
\[
H^1_{\Herr,+}(M')
\cong
\bFp[t^{\pm1}][\Gal(F_n/F)]^{\oplus \dF+\dkF}
\]
as a $\bFp[t^{\pm 1}][\Gal(F_n/F)]$-module.
So, we get the following corollary:

\begin{cor}
\label{cor:Ggen}
There \MINOREDITED{exist} cocycles
\[
c_1',\cdots, c_{\dF+\dkF}'\in
Z^1_{\Herr,+}(M')
\]
such that
\[
\{\sigma^j\cdot c_i'|\sigma^j\in\Gal(F_n/F),
1\le i\le \dF+\dkF\}
\]
forms a basis for $H^1_{\Herr,+}(M')$.
So,
\[
\{c_i:=\sum_{j=0}^{n-1}\MINOREDITED{\sigma^j}c_i'|
i=1,\cdots, \dF+\dkF\}
\]
forms a basis for $H^1_{\Herr,+}(M)$.
\end{cor}

\begin{example}
Let $F=\Q_p$ and $F_2=\Q_{p^2}$.

Let $\chi:\Gal_{F_2}\to \bFp^\times$ be a
character of order \MINOREDITED{$p^2-1$}.
For arbitrary scalars $a,b\neq 0\in \bFp^\times$,
we can construct an irreducible Galois representation
$\rho_{a,b}:\Gal_{\MINOREDITED{F}}\to \GL_2(\bFp)$
such that
$\rho_{a,b}|_{\Gal_{F_2}}=
\begin{pmatrix}
\chi & \\ & \chi^p
\end{pmatrix}
$
and $\rho_{a,b}(\Frob)=
\begin{pmatrix}
& a \\ b & 
\end{pmatrix}
$.
We form a $3$-dimensional Galois representation
$
\left(
\begin{array}{c|c}
\rho_{a,b} & 
\begin{matrix}
c_1 \\
c_2
\end{matrix}
\\
\hline
& 1
\end{array}
\right)
$
We remark that
the conjugation action
of $\rho_{a,b}(\Frob)$ on
$\Ext^1(\MINOREDITED{\bFp,\rho_{a,b}})$
is not necessarily idempotent.
Indeed, $\rho_{a,b}(\Frob^2)$ acts
by scalar multiplication by $ab$.

Next, we convert the Galois representations
to \'etale $(\varphi, \Gamma)$-modules.
Write $D_F(-)$ for the Fontaine functor
from Galois representations to \'etale
$(\varphi, \Gamma)$-modules over $\bfE_{F}$.
Then $D_F(\rho_{a,b})$
is an \'etale $(\varphi, \Gamma)$-module
over $\bfE_{F, \bFp}=\bfE_{F,\bFp}$.
We have
\[
D_{F_2}(\rho_{a,b}|_{\Gal_{F_2}})
=D_F(\rho_{a,b}) \otimes_{\bfE_{F,\bFp}}\bfE_{F_2,\bFp}
=D_F(\rho_{a,b})\otimes_{\F_p}\F_{p^2},
\]
and is equipped with Galois
descent data
$\tau: x\otimes y (\in D_F(\rho_{a,b})\otimes_{\F_p}\F_{p^2})
\mapsto x\otimes y^p$.
The descent data $\tau$ is always idempotent:
$\tau^2=1$.

Over the unramified extension $F_2 = \Q_{p^2}$, the restricted Galois representation decomposes as a direct sum of the characters $\chi$ and $\chi^p$. Correspondingly, the \'etale $(\varphi, \Gamma)$-module splits as 
$M' := D_{F_2}(\rho_{a,b}|_{\Gal_{F_2}}) = M'_\chi \oplus M'_{\chi^p}$.
The $\tau$-action on
$$Z^1_{\Herr}(M') = Z^1_{\Herr}(M'_\chi) \oplus Z^1_{\Herr}(M'_{\chi^p})$$
exchanges the two isotypic factors.

Let $\{\MINOREDITED{\delta_1}, \delta_2, \gamma=\delta_1+\delta_2\}$
be the positive roots of $\GL_3$.
We have
\begin{align*}
M'_{\chi} &= U_{\gamma^\vee}(\bfE_{F_2,\bFp})
\\
M'_{\chi^p} &= U_{\delta_2^\vee}(\bfE_{F_2,\bFp})
\end{align*}
So, if
\[
\{c_{\delta_2, 1}, c_{\delta_2,2}, c_{\delta_2,3},
c_{\delta_2,4}
\}
\subset Z^1_{\Herr,+}
(U_{\delta_2^\vee}(\bfE_{F_2,\bFp[a^{\pm1},b^{\pm1}]}))
\]
forms a basis for
$H^1_{\Herr,+}(
U_{\delta_2^\vee}(\bfE_{F_2,\bFp[a^{\pm1},b^{\pm1}]}))$,
then
\[
\{\tau c_{\delta_2, 1}, \tau c_{\delta_2,2}, \tau c_{\delta_2,3},
\tau c_{\delta_2,\MINOREDITED{4}}
\}
\subset Z^1_{\Herr,+}(
U_{\gamma^\vee}(\bfE_{F_2,\bFp[a^{\pm1},b^{\pm1}]}))
\]
forms a basis for
$H^1_{\Herr,+}(
U_{\gamma^\vee}(\bfE_{F_2,\bFp[a^{\pm1},b^{\pm1}]}))$.
Moreover,
\[
\{
(c_{\delta_2,1}, \tau c_{\delta_2,1}),
(c_{\delta_2,2}, \tau c_{\delta_2,2}),
(c_{\delta_2,3}, \tau c_{\delta_2,3}),
(c_{\delta_2,4}, \tau c_{\delta_2,4})
\}
\subset Z^1_{\Herr,+}(
(U_{\delta_2^\vee}\oplus U_{\gamma^\vee})(\bfE_{F,\bFp[a^{\pm1},b^{\pm1}]}))
\]
forms a basis for
$H^1_{\Herr,+}(
(U_{\delta_2^\vee}\oplus U_{\gamma^\vee})(\bfE_{F,\bFp[a^{\pm1},b^{\pm1}]}))
$.
Here, we always choose
$\{c_{\delta_2, 3}, c_{\delta_2,4}\}$ \MINOREDITED{to be} the unramified cocycles.
\end{example}

\begin{example}
Now consider $\bG=\PGL_5$:

\blockFive

Consider the Galois representations of the form
\[
\rho|_{I_n}
=
\begin{pmatrix}
\chi & & * & * & * \\
& \chi^p & * & * & * \\
& & \omega & & \\
& & & \omega^p & \\
& & & & \omega^{p^2}
\end{pmatrix},
\qquad
\rho(\Frob)
=
\begin{pmatrix}
& a_1 & * & * & * \\
a_2 & & * & *& * \\
& & & a_3 & \\
& &  & & a_4 \\
& & a_5 & &
\end{pmatrix}.
\]
Set $R=\bFp[a_i^{\pm1}:i=1,2,3,4,5]$.
The order of $w$ is $6$.
So, assume $F=\Q_p$
and $F_6=\Q_{p^6}$.
\MINOREDITED{There is one $w$-relative root containing the six positive cross-block roots:
\[
\Balpha = \{1100,0110,1111,0100,1110,0111\}.
\]}
\MAJOREDIT{Note that
$F_6$ is the splitting field of this relative root. Let
$c_{1100,1},\ldots,c_{1100,m}$ be cocycle representatives for an $R$-basis of
$H^1_{\Herr,+}(U_{1100}(\bfE_{F_6,R}))$. Write $\sigma$ for a generator of
$\Gal(F_6/F)$ and $\tau$ for the induced semilinear action on
$U_{\Balpha}$. By Shapiro descent, the universal cocycle for
$Z^1_{\Herr,+}(U_{\Balpha}(\bfE_{F,R}))$ is
\[
c_{\Balpha}^\univ
=
\sum_{i=1}^m x_i\sum_{j=0}^{5}\tau^j c_{1100,i}
\in Z^1_{\Herr,+}(U_{\Balpha}(\bfE_{F,R[x_1,\ldots,x_m]})).
\]}
\end{example}

\section{Twisted destackification and twisted
Weil-Deligne stacks}

Setup \ref{setup:w} is in force in this subsection.
We also keep the notation of the previous section.
Let $\BK\subset U^{\bM}$
\MINOREDITED{be} a $\bT\langle w\rangle$-stable
subgroup.
Write $\BK=\prod_{\Balpha\in \BPhi_{\BK}}U_{\Balpha}$.
Write $\BDelta_{\BK}\subset \BPhi_{\BK}$
for the set of relative roots \MINOREDITED{occurring}
in $\Lie \BK/[\Lie\BK, \Lie\BK]$.
Write $n$ for the order of $w$ in the Weyl group.

\begin{lem}
If $p>h_{\bG}$, then $p\nmid (\#N_{\bG}(\bT)/\bT)$.
\end{lem}

\begin{proof}
It is a standard fact.
See, for example, 
\cite[Theorem 3.9]{Hum90}.
\end{proof}

As a consequence, $p\nmid n$.

\subsection{Destackification of
$\cX_{\bT \langle w\rangle}$}

We have
\[
{\cX_{\bT \langle w\rangle}}
\cong
\coprod [\bT/_w\bT]
\]
is a disjoint union of $[\bT/_w\bT]$.
Unlike the untwisted case,
the $\bT$-action on $\bT$ is non-trivial.
We set $X_{\bT \langle w\rangle}
:=\coprod \bT$.

\subsection{Destackification of $\cX_{\BK\rtimes \bT \langle w\rangle}$}
We fix a total ordering
on $\BPhi_{\BK}$ compatible with $\BK$-height.

We recursively build destackifications $X_{\le \Balpha}$
for $\cX_{\le \Balpha}:=\cX_{U_{\le \Balpha}
\rtimes \bT \langle w\rangle}$
in a manner similar to the procedure
of Subsection \ref{subsec:Borel-destack},
with the only difference being that sections
$\xi:P_A\to P_A'$ are required to be both $\Gamma$-stable
and $\Gal(F_n/F)$-stable.

\begin{prop}
The morphism
$X_{\le \Balpha}\to \cX_{\le \Balpha}$ is smooth, surjective,
affine, and \MINOREDITED{has fibers}
are isomorphic to $\Res_{k_{F_{\cyc}}/\F_p}U^{\bM}\rtimes \bT$.

Moreover, the formation of $X_{\le \Balpha}$
tautological and does not depend
on the choice of the total ordering on $\BPhi_{\BK}$.
\end{prop}

\begin{proof}
Similar to Proposition 
\ref{prop:Borel-destack}.
\end{proof}

\subsection{Twisted Weil-Deligne stacks}

For each relative root $\Balpha\in \BPhi^+$,
set
$\G_{m,\Balpha}
:=
\prod_{\alpha\in \Balpha}\G_{m, \alpha}
$.
Let
\[
\cW_{U_{\Balpha}\rtimes \G_{m,\Balpha} \langle w\rangle}
\to
\cX_{U_{\Balpha}\rtimes \G_{m,\Balpha} \langle w\rangle}
\]
be the morphism
relatively representing
the presheaf
\[
R\mapsto
\Tor_1^{\cO_{X_{\G_{m,\Balpha} \langle w\rangle}}}(H^2_\Herr(\bfE_{F, \cO_{X_{\G_{m,\Balpha} \langle w\rangle}}}), R).
\]
Note that
\begin{align}
\label{eq:Tor-G}
\Tor_1^{\cO_{X_{\G_{m,\Balpha} \langle w\rangle}}}(H^2_\Herr(\bfE_{F, \cO_{X_{\G_{m,\Balpha} \langle w\rangle}}}), R)
=
\Tor_1^{\cO_{X_{\G_{m,\Balpha}}}}(H^2_\Herr(\bfE_{F_n, \cO_{X_{\G_{m,\Balpha}}}}), R)^{\Gal(F_n/F)}.
\end{align}

\begin{lem}
We have
\[
\cW_{U_{\Balpha}\rtimes \G_{m,\Balpha} \langle w\rangle}
=
(\cW_{F_n, U_{\Balpha}\rtimes \G_{m,\Balpha}}
\times_{\cX_{F_n,\G_{m,\Balpha}}}
\cX_{\G_{m,\Balpha} \langle w\rangle})^{\Gal(F_n/F)}.
\]
\end{lem}
\begin{proof}
Follows from Eq. (\ref{eq:Tor-G})
\end{proof}

Set
\[
W_{U_{\Balpha}\rtimes\G_{m,\Balpha}\langle w\rangle}
:=X_{\G_{m,\Balpha}\langle w\rangle}
\times_{\cX_{\G_{m,\Balpha}\langle w\rangle}}
\cW_{U_{\Balpha}\rtimes\G_{m,\Balpha}\langle w\rangle}
\]
for the destackified Weil-Deligne stack. We add a subscript $F$ to $W$ or $X$
when the ground field must be emphasized.
Note that each connected component of
\[
W_{F_n, U_{\Balpha}\rtimes \G_{m,\Balpha}}
\times_{X_{F_n,\G_{m,\Balpha}}}
X_{\G_{m,\Balpha} \langle w\rangle}
\]
is of the form $\Spec \bFp[t_1^{\pm1}, \cdots, t_k^{\pm1},
s_1, \cdots, s_{k'}]$;
the $\Gal(F_n/F)$-action
fixes the $t_*$-parameters
and permutes the $s_*$-parameters.

\begin{defn}
A connected component of 
$W_{U_{\Balpha}\rtimes \G_{m,\Balpha} \langle w\rangle}$
is either of the form
\[
\G_{m, \Balpha}
\]
or of the form
\[
\G_{m, \Balpha} [s]/(s~u)
\]
where $u\in \cO_{\G_{m,\Balpha}}$
is the equation defining the closed subscheme
of $\G_{m,\Balpha}$
that acts on $U_{\Balpha}$
via the cyclotomic character
$\bFp$-pointwise
when restricted to $\Gal_{F_n}$.

We call $\G_{m,\Balpha}/(u) \times \Spec \bFp[s]
\subset \G_{m, \Balpha} [s]/(s~u)$
the Steinberg component.
\end{defn}

\subsection{The $b_{\std}$ and $c_{\std}$ cochains}
Let $\alpha\in \Balpha$ be any absolute root.
Write $F_d/F$ for the splitting field of $\alpha$;
so $d=\#\Balpha$.

Let $\Spec \bFp[t_{\alpha}^{\pm 1}, s_\alpha]/(s_\alpha(t_{\alpha}-1))\subset W_{F_n, U_{\alpha^\vee}\rtimes\G_{m, \alpha}}$
be the connected component containing the Steinberg component.
Set
$R:=\bigotimes_{\alpha\in \Balpha}
\bFp[t_{\alpha}^{\pm 1}, s_\alpha]/(s_\alpha(t_{\alpha}-1)).
$

Recall that there exists a cochain
\[
c_{\alpha, \std}'\in C^1_{\Herr,+}(U_{\alpha^\vee}(\bfE_{F_n, R}))
\]
such that
\[
d_{\Herr}^1(c_{\alpha,\std}')
=(t_\alpha-1)b_{\alpha, \std}'
\]
and that
$
[b_{\alpha, \std}']
$ generates
$H^2_{\Herr}(U_{\alpha^\vee}(\bfE_{F_n,R}))$.
Note that
\[
b_{\Balpha,\std}:=\sum_{\sigma^{j}\in \Gal(F_n/F)}
\sigma^{j}b_{\alpha,\std}'
\]
generates
$H^2_{\Herr}(U_{\Balpha}(\bfE_{F, R}))$,
by Frobenius reciprocity.
The $\alpha$-coordinate of $b_{\Balpha,\std}$
is
\[
b_{\alpha,\std}:=
\sum_{\sigma^{jd}\in \Gal(F_n/F_d)}
\sigma^{jd}b_{\alpha,\std}'
\in C^2_{\Herr,+}(U_{\alpha^\vee}(\bfE_{F_d,R})),
\]
which is a coboundary
if and only if $H^2_{\Herr}(U_{\Balpha}(\bfE_{F, R}))=0$.
Set $c_{\Balpha, \std}:=\sum_{\sigma^{j}\in \Gal(F_n/F)}
\sigma^{j}c_{\alpha,\std}'$.

\subsection{Gauge cochains}
Set
\[
W_{\BK\rtimes \bT \langle w\rangle}
:=
\prod_{\Balpha\in \BPhi}W_{U_{\Balpha}\rtimes \G_{m,\Balpha}\langle w\rangle}
\underset{\prod_{\Balpha\in \BPhi}
X_{\G_{m,\Balpha}\langle w\rangle}
}{\times}
X_{\bT \langle w\rangle}.
\]
Let $C\subset W_{\BK\rtimes \bT \langle w\rangle}$
be an irreducible component,
and set $R=\cO_C$.
Write
\begin{align*}
\BPsi_0(C) &= \{\Balpha\in \BPhi^+|
H^0_{\Herr}(U_{\Balpha}(\bfE_{F,R}))
\text{~is finite free of rank $1$}
\} 
\\
\BPsi_2(C) &= \{\Balpha\in \BPhi^+|
H^2_{\Herr}(U_{\Balpha}(\bfE_{F,R}))
\text{~is finite free of rank $1$}\}
\\
\overline\BPsi_0(C) &= \{\Balpha\in \BPhi^+|
H^0_{\Herr}(U_{\Balpha}(\bfE_{F,R}))\neq 0
\}
\\
\overline\BPsi_2(C) &= \{\Balpha\in \BPhi^+|
H^2_{\Herr}(U_{\Balpha}(\bfE_{F,R}))\neq 0
\}.
\end{align*}

We fix a representative $\alpha\in \Balpha$
once for all.
For each $\alpha$,
we choose
a set of cocycles
\[
\{
c'_{\alpha, 1}, \cdots, c'_{\alpha, n\dF+n\dkF}
\}
\subset
Z^1_{\Herr,+}(U_{\alpha^\vee}(\bfE_{F_n,R}))
\]
which is stable under the semilinear $\Gal(F_n/F)$-action
and 
forms a basis of $H^1_{\Herr,+}(U_{\alpha^\vee}(\bfE_{F_n,R}))$.
Set
\[
c_{w^j\cdot \alpha, i}':= \sigma^j c_{\alpha, i}'
\]
for $1\le i\le n(\dF+\dkF)$
and $0\le j< \#\Balpha=:d$.
Also write
\[
\sigma^{jd}c'_{\alpha, i}=
c'_{\alpha, \sigma^{jd}(i)}.
\]

\begin{defn}
\label{defn:twisted-gauge}
We recursively define finitely many
$1$-cochains
$\fC_{\Xi_{\Balpha,k}}=\{c_{\Balpha, \xi}\}
\subset \MINOREDITED{C^1_{\Herr,+}}(U_{\Balpha}(\bfE_{F,R}))$,
$\xi\in \Xi_{\Balpha,k}$
for $k=1,\dots,h_{\bG}$.

We set \[
\Xi_{\Balpha,1} =
\begin{cases}
\{1,2,\dots, \dF+\dkF,\std\} & 
\text{$\Balpha\in \overline\BPsi_2(C)$}
\\
\{1,2,\dots, \dF+\dkF\} & 
\text{otherwise}
\end{cases}
\]
and $\fC_{\Xi_{\alpha,1}}=\{c_{\Balpha,\xi}|\xi\in \Xi_{\alpha,1}\}$,
where
\[
c_{\Balpha,\xi}=\sum_{\sigma^j\in\Gal(F_n/F)}\sigma^jc_{\alpha,\xi}'.
\]

Suppose $\fC_{\Xi_{\alpha,k}}$ is already defined.
For each sequence of roots $\Bbeta_1+\dots+\Bbeta_s\supset \Balpha$
and elements $c_{\Bbeta_i, \xi_i}\in C_{\Xi_{\Bbeta_i,k}}$,
denote by
\[
[c_{\Bbeta_1,\xi_1},\dots, c_{\Bbeta_s,\xi_s}]_{\Balpha}
\]
the $\Balpha$-factor of
\[
[c_{\Bbeta_1,\xi_1},\dots, c_{\Bbeta_s,\xi_s}]
\in C^1_{\Herr,+}((\prod_{\beta\in 
\Bbeta_1+\dots+\Bbeta_s}U_{\beta^\vee})(\bfE_{F, R})).
\]
We choose
an element $c_{\Balpha, [\xi_1,\xi_2,\dots,\xi_s]}
\in C^1_{\Herr,+}(U_{\Balpha^\vee}(\bfE_{F,R}))$
such that
\[
[c_{\Bbeta_1,\xi_1},\dots, c_{\Bbeta_s,\xi_s}]_{\Balpha}
=
\begin{cases}
a_{\Balpha, [\xi_1,\dots,\xi_s]} b_{\Balpha,\std} - d_{\Herr}^1(c_{\Balpha,[\xi_1,\xi_2,\dots,\xi_s]})
& \Balpha\in \overline\BPsi_2(C)
\\
-d_{\Herr}^1(c_{\Balpha,[\xi_1,\xi_2,\dots,\xi_s]})
& \text{otherwise}
\end{cases}
\]
where $a_{\Balpha,*}\in R$ is the unique scalar
that makes the equation hold.
We set
\[
\Xi_{\alpha, k+1} = 
\Xi_{\alpha,k}\cup\{[\xi_1,\dots,\xi_k],\xi_i\in 
\fC_{\Xi_{\beta_i,k}}\}
\]
and
\[
\fC_{\Xi_{\alpha,k+1}}
=\{c_{\alpha, \xi}: \xi\in \Xi_{\alpha, k+1}\}.
\]

We call elements of
$\fC_{\Xi_{*,*}}$
the {\it gauge cochains}.
\end{defn}

\begin{lem}
If all irreducible components of $X_{<\Balpha}(C)$
have dimension $\ge d$,
then
all irreducible components of
$X_{\le\Balpha}(C)$
have irreducible components of dimension
$\ge d+(\dF+\dkF)\#\Balpha$.
\label{cor:w-least-dim}
\end{lem}

\begin{proof}
Completely similar to
Corollary \ref{cor:least-dim}.
\end{proof}

\subsection{The fine stratification}

\begin{defn}
For subsets $\BPsi_0, \BPsi_2\subset \BPhi^+$
such that $\BPsi_2$ is {\it not} non-obstructing,
denote by
\[
X_{\BK\rtimes \bT\langle w\rangle}\langle\BPsi\rangle
\subset
X_{\BK\rtimes \bT\langle w\rangle}
\]
the locally closed subscheme
where
$H^i_{\Herr}(U_{\Balpha})$
is finite free of rank $1$
if and only if $\Balpha\in \BPsi_i$,
for $i\in\{0,2\}$.
\end{defn}

\begin{thm}
\rm
\label{thm:weq}
(1) 
$
X_{\BK\rtimes \bT \langle w \rangle}\langle\BPsi\rangle$
is a (possibly empty) union of certain connected components of
\[
\Spec \bFp
\left[
\begin{matrix}
x_{\Balpha,1},\dots,x_{\Balpha,\dF+\dkF,
s_{\Balpha}}&:&\Balpha\in \BPhi_{\BK}
\\
t_\delta^{\pm 1}&:&
\delta\in \Delta
\end{matrix}
\right]/
I
\]
where $I$ is generated by
\[
\left\{
\begin{matrix}
s_{\Balpha} &:& \Balpha \not\in \BPsi_2 \\
r_{\Balpha} &:& \Balpha \in \BPsi_2, \Ht_{\BK}(\Balpha)>1 \\
(t_\alpha-1)^n &:& \alpha\in \bigcup_{\Balpha\in\Psi_0\cup\Psi_2}
\Balpha\\
x_{\Balpha,i}=\sigma^{j\#\Balpha}x_{\Balpha,i}&:&
0\le j\#\Balpha<n, 1\le i\le \dF+\dkF
\end{matrix}
\right\}.
\]
Here, \[
r_{\Balpha}:=
\sum_{\Bbeta_1+\dots+\Bbeta_m=\Balpha, \xi_1,\dots,\xi_m}
\frac{1}{m!}
a_{\Balpha,[\xi_1,\dots,\xi_m]}
\prod_{j=1}^m\mu_{\Bbeta_j,\xi_j}
\]
where $\mu_{\Bbeta_j,\xi_j}$
is the coefficient before
$c_{\Bbeta_j,\xi_j}$
in the universal cochain $c_{\Bbeta_j}^\univ$.

\MINOREDITED{(2)}
We can normalize it so that
\[
r_{\Balpha}=
\sum_{\beta\in\Bbeta\in \BPhi_{\BK},
\delta\in\Bdelta\in \BDelta_{\BK}, \Bdelta+\Bbeta\supset \Balpha}
\frac{1}{2}N_{\beta, \delta}
x_{\Bbeta,1}x_{\Bdelta,1}
+Q_{\Balpha}
\]
where 
$Q_{\Balpha}$
does not depend on
$x_{\Bbeta,1}$
for $\Ht_{\BK}(\Bbeta)=\Ht_{\BK}(\Balpha)-1$.
\end{thm}

\begin{proof}
Completely similar to Theorem \ref{thm:eq}.
\end{proof}

\begin{defn}
Write $f^\univ, g^\univ
\in (\BK\rtimes \bT \langle w \rangle)(\bfE_{F,
\cO_{X_{\BK\rtimes \bT \langle w\rangle}\langle\BPsi\rangle}})$
for the universal \'etale $(\varphi, \Gamma)$-module
satisfying
\[
f^\univ \varphi(g^\univ)=g^\univ \gamma(f^\univ).
\]
Set
\[
\log_{\le h_{\bG}}(f^\univ (f^{\univ})^{\SS-1},
g^{\univ}(g^{\univ})^{\SS-1})
=\sum_{\Balpha\in \BPhi_{\BK}}c_{\Balpha}
\]
Denote by
\[
\dot X_{\BK\rtimes \bT\langle w\rangle}
\subset
 X_{\BK\rtimes \bT\langle w\rangle}
\]
the quasi-affine open subscheme
where
$[c_{\Bdelta}]\neq 0\in 
H^1_{\Herr}(U_{\Bdelta})$
for $\Bdelta\in \BDelta_{\BK}$.
\end{defn}

\begin{cor}
\label{cor:w-C6}
(1) If
the Jacobian
\[
J=(\frac{\partial r_{\Balpha}}{\partial x_{\Bbeta,1}})
\]
is of full rank,
then
$\dot X_{\BK\rtimes \bT \langle w\rangle}
\langle \BPsi\rangle$
is smooth.
Write 
$\dot W_{\BK\rtimes \bT \langle w\rangle}\langle\Psi\rangle\subset W_{\BK\rtimes \bT \langle w\rangle}$
for the image of 
$\dot X_{\BK\rtimes \bT \langle w\rangle}
\langle \BPsi\rangle$,
then
all fibers of
$\dot X_{\BK\rtimes \bT \langle w\rangle}
\langle \BPsi\rangle\to
\dot W_{\BK\rtimes \bT \langle w\rangle}
\langle \BPsi\rangle$
are irreducible of dimension
$[F:\Q_p]\#\bigcup_{\Balpha\in \BPhi_{\BK}}\Balpha$.

(2)
Write $r$ for the corank of $J$.
Then
all fibers of
$\dot X_{\BK\rtimes \bT \langle w\rangle}
\langle \BPsi\rangle\to
\dot W_{\BK\rtimes \bT \langle w\rangle}
\langle \BPsi\rangle$
are of dimension at most
$[F:\Q_p]\#\bigcup_{\Balpha\in \BPhi_{\BK}}\Balpha+r$.
\end{cor}

\begin{proof}
Completely similar to Corollary \ref{cor:C6}.
\end{proof}

It turns out that for the twisted situation,
we don't even need the full power of the cone model.

\begin{thm}
\rm
\label{thm:twist-reg}
Let $C\subset \cW_{U^{\bM}\rtimes \bT \langle w\rangle}$
be an irreducible component.

(1)
$\cW_{U^{\bM}\rtimes \bT \langle w\rangle}$
is equidimensional of dimension $0$.

(2)
If $[F:\Q_p]\ge \#\Phi^+ + 2$,
then
each $\cX_{U^{\bM}\rtimes \bT \langle w\rangle}(C)$
is irreducible.
As a consequence,
$
\cX_{U^{\bM}\rtimes \bT \langle w\rangle}
\to
\cW_{U^{\bM}\rtimes \bT \langle w\rangle}
$
is generically smooth and induces a bijection
of irreducible components.

(3)
If $\bG=\mathrm{F}_4$,
then part (2) holds without the $[F:\Q_p]\ge \#\Phi^+ + 2$
assumption.
\end{thm}

\begin{proof}
(1) 
\MAJOREDIT{Let
\[
S_C:=\BPsi_2(C)\cap\BDelta
\]
be the relative height-one orbits on which $C$ is Steinberg. A relative
height-one root has coefficient one in a unique simple-root direction outside
$\Delta_{\bM}$ and coefficient zero in every other omitted direction. Since
$w\in W_{\bM}$, this direction is constant on each cyclic orbit, giving a map
\[
\pi:S_C\longrightarrow\Delta-\Delta_{\bM}.
\]
We only need this map on $S_C$, not a bijection on all of $\BDelta$.

The map $\pi$ is injective. Indeed, argue contrapositively. In a fixed omitted
direction, the relative-height-one slice is connected by Levi root strings.
The root-string decomposition under the cyclic action has the following
dichotomy: either one cyclic orbit contains adjacent roots whose difference is
in $\Phi_{\bM}$, or two distinct cyclic orbits in that direction contain such
an adjacent pair. Thus two elements of $S_C$ with the same image under $\pi$
would make $S_C$ non-obstructing. This is impossible because
$S_C\subset\BPsi_2(C)$ and the fine stratum requires $\BPsi_2(C)$ to be not
non-obstructing.

Each $\Bdelta\in S_C$ contributes one Steinberg additive parameter, so the
additive dimension gain is $\#S_C$. On the multiplicative side, all relative
roots in one omitted direction impose one torus condition after quotienting by
the characters $\delta-w\delta$. Conditions in distinct omitted directions are
independent, as their images are distinct fundamental-coweight coordinates
modulo the Levi character lattice. Hence the multiplicative dimension loss is
$\#\pi(S_C)$. Injectivity gives
\[
\#S_C=\#\pi(S_C),
\]
so the additive gain exactly cancels the multiplicative loss. More additive
gains would force two elements over one omitted direction and hence violate the
non-obstructing constraint. Therefore every allowed component has dimension
zero. Over $F_n$, the restricted actions on $U_{w\delta^\vee}$ and
$U_{\delta^\vee}$ need not be trivial but are equal, so the indicated torus
quotient is compatible with the root-group factors.}

(2) Follows from Corollary \ref{cor:w-C6}
and Theorem \ref{thm:twist-C6}.

(3)
By Part (1),
it remains to show
when $\BK\neq U^{\bM}$,
the corank of 
$J=(\frac{\partial r_{\Balpha}}{\partial x_{\Bbeta,1}})$
does not exceed
\[
\#\bigcup_{\Balpha\in \BPhi^+-\BPhi_{\BK}}\Balpha
-
\dim \cW_{\BK\rtimes \bT \langle w\rangle}.
\]
We also have
\[
\dim \cW_{\BK\rtimes \bT \langle w\rangle}
\le
\# \BPsi_2\cap \BDelta_{\BK}
-\dim \frac{\Span_{\Q}\{\alpha|\alpha\in \BPsi_2\cup
\BPsi_0\}}{
\Span_{\Q}\{\delta-w \delta| \delta\in \Delta\}}.
\]
Since every term is easily computable by linear algebra,
Part (3) follows from enumeration
of all configurations of $(\BK, \BPsi)$.
\end{proof}

We remark that part (1) of 
Theorem 
\ref{thm:twist-reg}
does not have to be true for general $\BK$:
in general, $\BDelta_{\BK}$ may have more elements
than $\Delta-\Delta_{\bM}$.

\newpage
\phantomsection
\addcontentsline{toc}{part}{Part III: 
The $p$-adic formal Weil-Deligne stacks}

\noindent
{\large Part III: The $p$-adic formal Weil-Deligne stacks}

\section{Geometrization of the $p$-adic local Euler characteristic}

We use the integral and rigid moduli theory of
\cite[Lemma 11 and Theorems 10--12]{Lin25}.
For the twisted version, sections are additionally required
to be Galois-stable.
There is no other difference in the construction.

\begin{lem}
\rm
Let $X = \mathrm{Spf} R$ where
$R = \mathbb{Z}_p\langle T \rangle$.
The presheaf over $X$ defined by
$S \mapsto \mathrm{Tor}^R_1(R/(T), S)$
is representable by the formal scheme
$\mathrm{Spf}(R\langle Y \rangle / (TY))$.
\end{lem}

\begin{proof}
First, compute the functor
$F(S) = \mathrm{Tor}^R_1(R/(T), S)$.
Since $T$ is a non-zero-divisor in $R$,
we have a free resolution:
\[ 0 \to R \xrightarrow{\cdot T} R \to
R/(T) \to 0 \]
Tensoring with an $R$-algebra $S$,
we get the complex:
\[ 0 \to S \xrightarrow{\cdot T} S \to
S/TS \to 0 \]
The first homology group is
$\ker(S \xrightarrow{\cdot T} S) = S[T]$.
Thus, $F(S)$ assigns each $S$ its
$T$-torsion submodule, $S[T]$.

To represent this functor, consider
the $p$-adically complete $R$-algebra
$A = R\langle Y \rangle / (TY)$.
A continuous $R$-algebra homomorphism
$\phi: A \to S$ is determined by
$y = \phi(Y) \in S$.
This requires $Ty = 0$ in $S$.
Because $A$ is a Tate algebra,
any such $y$ gives a continuous map.
This yields a bijection
$\mathrm{Hom}_{R}(A, S) \cong S[T]$.
\end{proof}

Consider the moduli
stack 
$\cX_{F,\G_m}^{\underline{0},\tame}$
of 
potentially semistable
rank $1$
\'etale $(\varphi, \Gamma)$-module
of trivial Hodge type
and tame inertial type
over $\bfA_{F, \bar\Z_p}$.
Here, $\bfA_{F,A}$ is the usual thickening of $\bfE_{F,A}$.
All such 
\'etale $(\varphi, \Gamma)$-modules
correspond to unramified twists
of Teichm\"uller lifts
of a mod $p$ Galois representation.
So, we have
\[
\cX_{F,\G_m}^{\underline{0},\tame}
=
\coprod_{
\tau\in \F_q^\times}
\cX_{F, \G_m}^{\underline{0}, \tau}
=
\coprod_{
\tau\in \F_q^\times}
[\Spf \bar\Z_p\langle T^{\pm1}\rangle/\Spf \bar\Z_p\langle T^{\pm1}\rangle],
\]
and admits a destackification
\[
X_{F,\G_m}^{\underline{0},\tame}
:=
\coprod_{
\tau\in \F_q^\times}
\Spf \bar\Z_p\langle T^{\pm1}\rangle.
\]
Write
$M^{\underline{0},\tame}$
for the universal \'etale
$(\varphi, \Gamma)$-module
over $R=\bfA_{F, \bar\Z_p\langle T^{\pm1}\rangle}$.
Write
$M_{\cyc}$ for the cyclotomic
\'etale $(\varphi, \Gamma)$-module
over $\bfA_{F, \bar\Z_p}$.
Consider the presheaf
\[
S\mapsto
\Tor^R_1(H^2_{\Herr}(M^{\underline{0},\tame}\otimes M_{\cyc}), S)
=
\Tor^R_1(\bar\Z_p\langle T^{\pm1}\rangle/(T-1),S),
\]
which is representable by
\[
\Spf \bar\Z_p\langle T^{\pm1}, Y\rangle
/(Y(T-1)).
\]
The formal scheme
$\Spf \bar\Z_p\langle T^{\pm1}, Y\rangle
/(Y(T-1))$ 
has two irreducible components
\[
\Spf \bar\Z_p\langle T^{\pm1}\rangle
\cup
\Spf \bar\Z_p\langle Y\rangle
\]
with special fiber
\[
\Spec \bar\F_p[ T^{\pm1}]
\cup
\Spec \bar\F_p[ Y]
\]
and generic fiber
\[
\Sp \bar\Q_p\langle T^{\pm1}\rangle
\cup
\Sp \bar\Q_p\langle Y\rangle.
\]

\begin{lem}
\rm
If $\bG$ has connected center,
then there exists a cocharacter
$\lambda\in X_*(\bT)$
such that
\[
\langle \lambda, \delta\rangle=-1
\]
for all $\delta\in \Delta$.
\end{lem}

\begin{proof}
Recall that
``having simply-connected derived subgroup''
$\Leftrightarrow$
``having fundamental weights''.
Dualizing both sides, we get
``having connected center''
$\Leftrightarrow$
``having fundamental coweights''.
Take $\lambda$ to be the \MINOREDITED{negative} sum of fundamental coweights.
\end{proof}

We fix such a cocharacter $\lambda$ once for all.
Let $w\in N_{\bG}(\bT)/\bT$ be a Weyl group element
of order $n$, that is elliptic in a proper Levi
$\bM\subset \bG$.
Let 
\[
\chi_{\LT}:\Gal_{F_n}\to \G_m(\cO_{F_n})=\cO_{F_n}^\times
\]
be the Lubin-Tate character.
Write $
\chi_{\LT}^{\lambda}
$
for the composite
\[
\Gal_{F_n}\xrightarrow{\chi_{\LT}} \G_m(\cO_{F_n})
\xrightarrow{\lambda(\cO_{F_n})}
\bT(\cO_{F_n}).
\]
Fix a lift $\sigma \in \Gal_F$ of a generator of $\Gal(F_n/F)$.
We define the Galois representation $N_w(\chi_{\LT}^{\lambda}): \Gal_F \to \bT(\bar\Z_p) \langle w\rangle$ by setting:
\[
N_w(\chi_{\LT}^{\lambda})(g \sigma^k)
:=
\left( \prod_{i=0}^{n-1} {w}^i \chi_{\LT}^{\lambda}(\sigma^{-i}g\sigma^i) {w}^{-i} \right) {w}^k.
\]

\begin{lem}
Let $\delta^\vee\in \Delta-\Delta_{M}$
and set $\Bdelta:=\langle w\rangle \cdot \delta^\vee$.
we have
\[
H^2(\Gal_F, U_{\Bdelta}(\bar\Z_p))
\cong \bar\Z_p.
\]
\end{lem}

\begin{proof}
Write $F_m$ for the splitting field of $\delta$.
By Shapiro's Lemma, we have an isomorphism of cohomology groups:
\[
H^2(\Gal_F, U_{\Bdelta}(\bar{\mathbb{Z}}_p)) \cong H^2(\Gal_{F_m}, U_{\delta^\vee}(\bar{\mathbb{Z}}_p))
=H^2(\Gal_{F_m}, \bar\Z_p(1))
=\bar\Z_p.
\qedhere
\]
\end{proof}

\begin{defn}
Write $\underline{\lambda}_w
\in X^*(\Res_{F/\Q_p}\bT)$
for the Hodge type of $N_w(\chi_{\LT}^{\lambda})$.

Write
\[
\cX_{U^{\bM}\rtimes \bT \langle w\rangle}
^{\underline{\lambda}_w, \tame}
\]
for the moduli stack
of potentially semistable
Galois representations
of Hodge type $\underline{\lambda}_w$
and tame inertial type.
Write
\[
\fX_{U^{\bM}\rtimes \bT \langle w\rangle}
^{\underline{\lambda}_w, \tame}
\]
for the rigid generic fiber of 
$
\cX_{U^{\bM}\rtimes \bT \langle w\rangle}
^{\underline{\lambda}_w, \tame}$.
\end{defn}

Next, we construct the integral
Weil-Deligne stack.

\begin{defn}
Let $R=\cO_{X_{\bT \langle w\rangle}^{\underline{\lambda}_w, \tame}}$.
Define
$\sW_{U^{\bM}\rtimes \bT \langle w\rangle}
\to \cX_{\bT\langle w\rangle}^{\underline{\lambda}_w, \tame}$
to be the morphism
that represents the presheaf
\[
S\mapsto \prod_{\Bdelta\in \BDelta}
\Tor^R_1(H_{\Herr}^2(U_{\Bdelta}(\bfA_{F,R})),S).
\]
Write
$\fW_{U^{\bM}\rtimes \bT \langle w\rangle}$
for the rigid generic fiber of
$\sW_{U^{\bM}\rtimes \bT \langle w\rangle}$,
whose representability follows from unramified descent
of prime-to-$p$ degree.
Note that the special fiber of
$\sW_{U^{\bM}\rtimes \bT \langle w\rangle}$
is precisely
$\cW_{U^{\bM}\rtimes \bT \langle w\rangle}$.
\end{defn}

We recall the following theorem from our previous work:

\begin{thm} (\cite[Theorem 3]{Lin25})
\rm
$
\fX_{U^{\bM}\rtimes \bT \langle w\rangle}
^{\underline{\lambda}_w, \tame}
\to
\fW_{U^{\bM}\rtimes \bT \langle w\rangle}
$
induces a bijection of irreducible components.
\label{thm:rigid}
\end{thm}

\begin{proof}
This is used as the external input \cite[Theorem 3]{Lin25}.
\end{proof}

\begin{cor}
The special fiber of 
$
\cX_{U^{\bM}\rtimes \bT \langle w\rangle}
^{\underline{\lambda}_w, \tame}$
is equidimensional of dimension
$[F:\Q_p]\#\Phi^{\bM}$,
and its number of irreducible components
is at least as many as the number of
irreducible components of
$\cW_{U^{\bM}\rtimes \bT \langle w\rangle}$.
\label{cor:Ch-lower-bd}
\end{cor}

\begin{proof}
Note that
irreducible components of
$\cW_{U^{\bM}\rtimes \bT \langle w\rangle}$
and
$\fW_{U^{\bM}\rtimes \bT \langle w\rangle}$
are precisely the special fiber
and the rigid generic fiber
of irreducible components $\sC$ of the
$p$-adic formal stack
$\sW_{U^{\bM}\rtimes \bT \langle w\rangle}$.

Consider
\[
\cX_{U^{\bM}\rtimes \bT\langle w\rangle}
^{\underline{\lambda}_w, \tame}(\sC)
:=
\cX_{U^{\bM}\rtimes \bT \langle w\rangle}
^{\underline{\lambda}_w, \tame}
\times_{\sW_{U^{\bM}\rtimes \bT \langle w\rangle}}
\sC,
\]
whose rigid generic fiber
$\fX_{U^{\bM}\rtimes \bT \langle w\rangle}
^{\underline{\lambda}_w, \tame}(\sC)$
is of dimension $[F:\Q_p]\#\Phi^{\bM}$
and non-empty
by Theorem \ref{thm:rigid}.
Since 
$\cX_{U^{\bM}\rtimes \bT \langle w\rangle}
^{\underline{\lambda}_w, \tame}(\sC)$
is a closed sub-formal scheme
of
$\cX_{U^{\bM}\rtimes \bT \langle w\rangle}
^{\underline{\lambda}_w, \tame}$,
the Zariski closure 
$\cY$ of 
$\fX_{U^{\bM}\rtimes \bT \langle w\rangle}
^{\underline{\lambda}_w, \tame}(\sC)$
in 
$\cX_{U^{\bM}\rtimes \bT \langle w\rangle}
^{\underline{\lambda}_w, \tame}$
is contained in 
$\cX_{U^{\bM}\rtimes \bT \langle w\rangle}
^{\underline{\lambda}_w, \tame}(\sC)$.
Take $\cY$ to be the schematic closure of the indicated
nonempty rigid component, with its $p$-power torsion removed.  It is then
$\bar\Z_p$-flat by construction.  A nonzero flat, topologically finite-type
formal scheme over the complete DVR (all stacks descend to a DVR) cannot have empty special
fiber: otherwise $p$ would be a unit on its coordinate ring, contrary to
adic completeness.  Hence $\cY\otimes\bar\F_p$ is nonempty.  Flatness and
Krull's Hauptidealsatz give
$\dim(\cY\otimes\bar\F_p)=\dim\cY-1
=[F:\Q_p]\#\Phi^{\bM}$.  This argument is applied separately to every
$\sC$.
So, there exists at least one irreducible component
of $\cX_{U^{\bM}\rtimes \bT \langle w\rangle}
^{\underline{\lambda}_w, \tame}\otimes \bar\F_p$
that maps onto $C:=\sC\otimes \bar\F_p$
and does not map onto other
irreducible components of
$\cW_{U^{\bM}\rtimes \bT \langle w\rangle}$.

The equidimensionality follows from the 
Krull's Hauptidealsatz argument applied to the entire stack 
$\cX_{U^{\bM}\rtimes \bT \langle w\rangle}
^{\underline{\lambda}_w, \tame}$.
\end{proof}

\begin{thm}
\rm
If the number of irreducible components
of $\cX_{U^{\bM}\rtimes \bT \langle w\rangle}$
does not exceed the number of irreducible components
of 
$\cW_{U^{\bM}\rtimes \bT \langle w\rangle}$,
then
all $\bFp$-points of
$\cX_{U^{\bM}\rtimes \bT \langle w\rangle}$
\MINOREDITED{have} a potentially semistable lift of regular Hodge type
$\underline{\lambda}_w$
and tame inertial type.

The condition is satisfied in the following cases:
\begin{itemize}
\item $\bG$ and $w$ arbitrary and $[F:\Q_p]\ge \#\Phi^++2$,
\item $\bG=\mathrm{F}_4$ and $w\neq 1$, or
\item $\bG=\mathrm{F}_4$ and $F\neq \Q_p$.
\end{itemize}
\label{thm:lift0}
\end{thm}

\begin{proof}
Combine 
Theorem \ref{thm:twist-reg},
Theorem \ref{thm:reg},
Theorem \ref{thm:F4-table} and
Corollary \ref{cor:Ch-lower-bd}.
\end{proof}

\section{The qualitative Breuil-M\'ezard conjecture}

\begin{lem}
\rm
The scheme-theoretic image $Y$ of each
\[
\dot\cX_{\bB}\langle\Psi\rangle
\to \cX_{\bG}^{\EG}
\]
has dimension at least $\dim \dot\cX_{\bB}\langle\Psi\rangle$.
\label{lem:maximal-ns}
\end{lem}

\begin{proof}
It suffices to show for
each
$\Spec \bFp\to Y$,
the fiber
\[
\Spec \bFp\times_{Y}\dot\cX_{\bB}\langle\Psi\rangle
\]
has dimension (at most) $0$.
Suppose this $\bFp$-point corresponds to a Galois 
representation
$\bar\rho:\Gal_F\to \bB(\bFp)$,
then this fiber embeds
in $\Aut_{\bG}(\bar\rho)/\Aut_{\bB}(\bar\rho)=\{1\}$,
as $\bar\rho$ is maximally non-split:

\begin{sublem}
Let $g = tu \in \bB$
where $t\in \bT$ and $u\in U$,
such that $\log_{\le h_{\bG}} u = \sum_{\alpha \in \Phi^+} u_{\alpha}$. If $u_{\delta} \neq 0$ for all simple roots $\delta \in \Delta$, then $Z_{\bG}(u) \subset \bB$.
\end{sublem}

\begin{proof}
Let $z \in Z_{\bG}(u)$.
By the Bruhat decomposition, we can uniquely write 
$z = b \dot{w} v$, where:
\begin{itemize}
\item $b \in \bB$,
\item $\dot{w}$ is a representative of a Weyl group element,
\item $v \in U_w = U \cap \dot{w}^{-1} U^- \dot{w} \subset U$.
\end{itemize}

Since $z$ centralizes $u$, we have $z^{-1} u z = u$. Substituting the Bruhat decomposition of $z$, we get:
\[ (b \dot{w} v)^{-1} u (b \dot{w} v) = u \implies \dot{w}^{-1} (b^{-1} u b) \dot{w} = v u v^{-1} \]

Let $u_1 = b^{-1} u b$. Because $b \in \bB$, the conjugation action of $b$ preserves $U$ and its lower central series. Specifically, the induced action on the abelianization $U / [U,U] \cong \prod_{\delta \in \Delta} \mathfrak{g}_\delta$ is given by non-zero scalars on each simple root space. Therefore, $u_1$ is also regular unipotent, meaning $\log u_1$ has a non-zero component in $\mathfrak{g}_\delta$ for every $\delta \in \Delta$.

On the right side of the equation, let $u_2 = v u v^{-1}$. Since both $v$ and $u$ are in $U$, $u_2 \in U$. 

We now have the equality $\dot{w}^{-1} u_1 \dot{w} = u_2$. Passing to the Lie algebra, this gives:
\[ \text{Ad}(\dot{w}^{-1}) \log u_1 = \log u_2 \]

Assume for the sake of contradiction that $\dot w \neq 1$. By the standard properties of the Weyl group, there exists at least one simple root $\delta \in \Delta$ such that $w^{-1}(\delta) < 0$ (i.e., it is a negative root). 

Because $u_1$ is regular unipotent, $\log u_1$ has a non-zero component $X_\delta \in \mathfrak{g}_\delta$. Under the adjoint action of $\dot{w}^{-1}$, this component is mapped to $\text{Ad}(\dot{w}^{-1})X_\delta$, which is a non-zero element in the negative root space $\mathfrak{g}_{w^{-1}(\delta)}$. 
This contradiction forces $\dot w = 1$. 
\end{proof}
\end{proof}

\begin{prop}
If 
\begin{itemize}
\item 
$\bG$ is arbitrary and $[F:\Q_p]\ge \#\Phi^++2$ or
\item
$\bG=\mathrm{F}_4$ and $F\neq\Q_p$,
\end{itemize}
then
$\cX_{\bG, \red}^{\EG}$
is equidimensional and its irreducible components
are in bijection with the irreducible components
of $\cW_{\bB}$.
\label{prop:TBM}
\end{prop}

\begin{proof}
By Theorem \ref{thm:lift0},
$\cX_{\bG, \red}^{\EG}$
is contained in the special fiber
of
$\cX_{\bG}^{\underline{\lambda}_1,\tame}$,
which is equidimensional
of dimension $[F:\Q_p]\#\Phi^+$.
\MAJOREDIT{Here the crystalline stack is a closed substack of the integral
Emerton--Gee stack by its scheme-theoretic-image construction. Conversely,
Theorem~\ref{thm:lift0} gives a closed immersion from the reduced
Emerton--Gee stack into the reduced special fiber of this crystalline stack.
The two opposite closed immersions identify these reduced substacks. Hence
$\cX_{\bG,\red}^{\EG}$ is the reduced special fiber of the indicated
crystalline stack and is equidimensional of the same dimension; in particular,
there are no additional lower-dimensional irreducible components.}
By Lemma
\ref{lem:maximal-ns},
the image of
$\dot \cX_{\bB}(C)$
in $\cX_{\bG}^{\EG}$
is a top-dimensional irreducible component
for each irreducible component $C\subset \cW_{\bB}$.
We remark
that $\dot \cX_{\bB}(C)$ is always irreducible,
despite that $\cX_{\bB}(C)$ may not be.
\end{proof}

\begin{lem}
\label{lem:rot-dim}
\rm
The morphism
\[
\cX_{K\rtimes \bT}[K]\langle \Psi\rangle
\to
\cX_{\bG}^{\EG}
\]
factors through
$
[\cX_{K\rtimes \bT}[K]\langle \Psi\rangle/
\prod_{\beta\in \Psi_0}U_{(-\beta)^{\vee}}
]
$.

So, its scheme-theoretic image
has dimension
$\le \dim \cX_{K\rtimes \bT}[K]\langle \Psi\rangle-\#\Psi_0$.
\end{lem}

\begin{proof}
Note that $\cX_{K\rtimes \bT}[K]\langle \Psi\rangle$
is clearly stable under the conjugation action
of the negative root group action
of $U_{(-\beta)^{\vee}}$
for $\beta\in \Psi_0$.
\end{proof}

\begin{cor}
When $\bG=\mathrm{F}_4$ and $F=\Q_p$,
the top-dimensional irreducible components
$\cX_{\bG, \red}^{\EG}$
are in bijection with the irreducible components
of $\cW_{\bB}$.

The only possibilities
for the non-top-dimensional irreducible components
are the scheme-theoretic image
of $\cX_{\bB}[K]\langle \Psi\rangle$
for the configurations $(K, \Psi)$
listed in Table \ref{table:F4-table}.
\label{cor:baby-F4}
\end{cor}

We don't know whether 
$\cX_{\mathrm{F}_4, \red}^{\EG}$
is equidimensional yet.

\begin{proof}
Combine 
Lemma \ref{lem:rot-dim},
Theorem 
\ref{thm:F4-table},
and the proof of Proposition \ref{prop:TBM}.
\end{proof}

\section{Rotation by negative root groups}
\label{sec:rotate}
In this section, we treat the $4$ extra components
of $\cX_{\bB}$ for $\bG=\mathrm{F}_4$ and $F=\Q_p$.

\begin{table}[H]
\centering
\scriptsize
\begin{tabular}{|c|c|c|c|}
\hline
ID & $\Phi^+-\Phi_K$ & $\Psi_0$ & $\Psi_2$ \\
\hline
I & $\{0010\}$ & $\{0010\}$ & $\{0001,0011,0100,0110,0120\}$ \\
II & $\{0010\}$ & $\{0010\}$ & $\{0001,0011,0100,0110,0120,1000\}$ \\
III & $\{1000,0010\}$ & $\{0010,1000\}$ & $\left\{ \begin{matrix} 0001,0011,0100,0110 \\ 0120,1100,1110,1120 \end{matrix} \right\}$ \\
IV & $\{1000,0001,0011,0010\}$ & $\{0001,0010,0011,1000\}$ & $\left\{ \begin{matrix} 0100,0110,0111,0120,0121,0122 \\ 1100,1110,1111,1120,1121,1122 \end{matrix} \right\}$ \\
\hline
\end{tabular}
\end{table}

The goal is \MINOREDITED{to} prove the following:

\begin{prop}
(Theorem \ref{thm:ht1-codim0})
\rm
There exists a Zariski dense substack
$\cU\subset \cX_{K\rtimes \bT}\langle \Psi\rangle$
where $(K, \Psi_0,\Psi_2)$ is
one of the four configurations above,
such that
for each
$\bFp$-point of $\cU$
corresponding to
$\bar\rho:\Gal_F\to (K\rtimes \bT)(\bFp)
\subset \bB(\bFp)$,
there exists
\begin{itemize}
\item 
a deformation
\[
\rho:\Gal_F\to \bB(\bFp[\![\varepsilon]\!])
\]
\item
an element $\mu\in \bFp[\![\varepsilon]\!]$, and
\item
a root $\delta_{\rot}\in \Psi_0$
\end{itemize}
satisfying
\[
\bar\rho=\exp(-\mu e_{-\delta_{\rot}})\rho
\exp(\mu e_{-\delta_{\rot}})
\mod \varepsilon
\]
and $\MINOREDITED{\rho}[\frac{1}{\varepsilon}]$
is a $\bFp(\!(\varepsilon)\!)$-point
of $\dot\cX_{\bB}=\cX_{\bB}[U]$.
\label{prop:rotate}
\end{prop}

\begin{remark}
The proposition is 
reduced to Theorem \ref{thm:ht1-codim0}
which is proved
towards
the end of this section.
The insight here is that,
even though
these $\cX_{\bB}[K]\langle\Psi\rangle$
are not in the closure of
$\dot \cX_{\bB}$,
their scheme-theoretic image
in $\cX_{\bG}^{\EG}$
is in the closure
of the scheme-theoretic image of
$\dot \cX_{\bB}$
in $\cX_{\bG}^{\EG}$.
\end{remark}

The key observation is that
$\Psi_2\subset \Delta_K$.
Write 
\[
V_K:=\prod_{\alpha\in (\Phi^+-\Phi_K)\cup \Delta_K}
U_{\alpha^\vee},
\]
so there is a short exact sequence
\[
\MINOREDITED{1}\to
\bar K_1 \to V_K \to \frac{U}{K}\to 1
\]

\begin{lem}
\rm
Proposition \ref{prop:rotate}
is equivalent to the 
same statement
after replacing
$K\rtimes \bT$ and $\bB$
by
$\bar K_1\rtimes \bT$
and $V\rtimes \bT$:
There exists a Zariski dense substack
$\cU\subset \cX_{\bar K_1\rtimes \bT}\langle \Psi\rangle$
where $(K, \Psi_0,\Psi_2)$ is
one of the four configurations above,
such that
for each
$\bFp$-point of $\cU$
corresponding to
$\bar\rho:\Gal_F\to (\bar K_1\rtimes \bT)(\bFp)
\subset (V\rtimes \MINOREDITED{\bT})(\bFp)$,
there exists
\begin{itemize}
\item 
a deformation
\[
\rho:\Gal_F\to (V\rtimes \bT)(\bFp[\![\varepsilon]\!])
\]
\item
an element $\mu\in \bFp[\![\varepsilon]\!]$, and
\item
a root $\delta_{\rot}\in \Psi_0$
\end{itemize}
satisfying
\[
\bar\rho=\exp(-\mu e_{-\delta_{\rot}})\rho
\exp(\mu e_{-\delta_{\rot}})
\mod \varepsilon
\]
and that
$\MINOREDITED{\rho}[\frac{1}{\varepsilon}]$
does not factor through any proper
$\bT$-stable normal subgroup
of $V\rtimes \bT$.
\end{lem}

\begin{proof}
By Nakayama's lemma
$H^2(\Gal_F, U_{\alpha^\vee}(\bFp[\![\varepsilon]\!]))
=H^2(\Gal_F, U_{\alpha^\vee}(\bFp))=0$
for each $\alpha\in \Phi_K -\Delta_K$.
\end{proof}

It is clear that
\[
X_{\bar K_1\rtimes\bT}\langle\Psi\rangle
=\Spec \bFp[t_\delta^{\pm1}, x_{\alpha,1}, x_{\alpha,2},
s_{\beta}:
\delta\in \Delta, \alpha\in \Delta_K,
\beta\in \Psi_2]/
(t_\alpha-1: \alpha\in \Psi_0\cup\Psi_2).
\]

\begin{example} (Configuration I)
We work out the first row of the table:

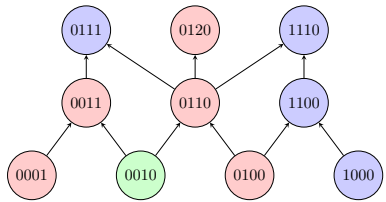
\begin{figure}[H]
\scalebox{.6}{\StratumK}
\end{figure}

\noindent
where $\Psi_2$ roots are colored by red
and the $\Psi_0$ root is colored by green.
It is harmless to assume
$x_{\alpha,2}=0$
as $\Psi_0\cap \Delta_K=\emptyset
\Rightarrow
[c_{\alpha,2}]=0\in H^1(\Gal_F, U_{\alpha^\vee}(\bFp))$;
the variable $x_{\alpha,2}$
only plays the role of framing.
We can also remove the roots $1000,1100,1110,0111$
as they are not dominated by the $\Psi_2$ roots,
and replace
$V$ by $\prod_{\alpha\in \{0001,0010,0100,0011,0110,0120\}}
U_{\alpha^\vee}$.
Write $x_{\alpha}:=x_{\alpha,1}$.

The universal family of $X_{\bar K_1\rtimes \bT}\langle \Psi\rangle$
is given by
\begin{align*}
c^{\univ}_{0010} &= 0 \\
c^{\univ}_{0001} &= x_{0001}c_{0001,1}
+ s_{0001}c_{0001,\std}\\
c^{\univ}_{0011} &= x_{0011}c_{0011,1}
+ s_{0011}c_{0011,\std}\\
c^{\univ}_{0100} &= x_{0100}c_{0100,1}
+ s_{0100}c_{0100,\std}\\
c^{\univ}_{0110} &= x_{0110}c_{0110,1}
+ s_{0110}c_{0110,\std}\\
c^{\univ}_{0120} &= x_{0120}c_{0120,1}
+ s_{0120}c_{0120,\std}.
\end{align*}

Set $\delta_{\rot}:=0010$.
Let $\mu$
be a free parameter.
We replace
$\bar\rho$ by
\[
\bar\rho_\mu:=\exp(\mu e_{-0010})\bar\rho \exp(-\mu e_{-0010})
\]
where $e_{-0010}$ is the standard Chevalley basis element.
Note that
$\bar\rho^\SS\exp(\mu e_{-0010})=\exp(\mu e_{-0010})\bar\rho^\SS$.
Also note that
\begin{align*}
\exp(\mu e_{-0010})e_{0120} \exp(-\mu e_{-0010})
&=e_{0120}
+ \mu[e_{-0010},e_{0120}]
+ \frac{\mu^2}{2}[e_{-0010},[e_{-0010},e_{0120}]]
\\&=
e_{0120}
-\mu e_{0110}
-\mu^2 e_{0100}
\\
\exp(\mu e_{-0010})e_{0110} \exp(-\mu e_{-0010})
&=
e_{0110} + 2\mu e_{0100}
\\
\exp(\mu e_{-0010})e_{0011} \exp(-\mu e_{-0010})
&=
e_{0011} -\mu e_{0001}
\end{align*}
This transforms the universal coordinates
according to the rule
\begin{align*}
c_{0110}^\univ &\mapsto
c_{0110}(\mu):=  c^{\univ}_{0110}-\mu c^{\univ}_{0120}\\
c^\univ_{0100} &\mapsto c_{0100}(\mu):= 
c^\univ_{0100}+2\mu c^\univ_{0110}
-\mu^2c^\univ_{0120}\\
c^\univ_{0011} &\mapsto c_{0011}(\mu):= 
c_{0011}^\univ\\
c^\univ_{0001} &\mapsto c_{0001}(\mu):= 
c^\univ_{0001}-\mu c^\univ_{\MINOREDITED{0011}}.
\end{align*}

Next, we construct the universal
$\varepsilon$-deformation
\begin{align*}
c_{0010}(\varepsilon, \mu)&:=
\varepsilon (b_{0010r}c_{0010,1}
+b_{0010u}c_{0010,2}) \\
c_{0001}(\varepsilon, \mu)&:=
c_{0001}(\mu) \\
c_{0100}(\varepsilon, \mu)&:=
c_{0100}(\mu) \\
c_{0011}(\varepsilon, \mu)&:=
c_{0011}(\mu) +\mathrm{O}(\varepsilon^2)
\\
c_{0110}(\varepsilon, \mu)&:=
c_{0110}(\mu) +\mathrm{O}(\varepsilon^2)
\\
c_{0120}(\varepsilon, \mu)&:=
c_{0120}(\mu) +\mathrm{O}(\varepsilon^2)
\\
t_{0001}(\varepsilon, \mu)&=1 \\
t_{0100}(\varepsilon, \mu)&=1 \\
t_{0010}(\varepsilon, \mu)&=1+\varepsilon \lambda_{0010} \\
t_{1000}(\varepsilon, \mu)&= t_{1000}
\end{align*}
where $\mathrm{O}(\varepsilon^2)$
is the linear combination of the gauge cochains
having coefficients divisible by $\varepsilon^2$
coming from (higher) cup products,
and $\lambda_{0010}$ is a free parameter.

We plug the universal $\varepsilon$-deformation
into the explicit equations in Lemma \ref{lem:eq}.
Clearly all the equations are satisfied mod $\varepsilon$.
Note that there are $4$ free parameters
$\lambda, \mu, b_{0010r}, b_{0010u}$
(mod $\varepsilon^2$),
and three equations
$\{r_{0011}, r_{0110}, r_{0120}\}$.
We add one more equation that $\lambda_{0010}=1$
so that the number of parameters matches the number
of equations.
The mod $\varepsilon^2$ terms of the equations
($=\frac{r_*}{\varepsilon}$ mod $\varepsilon^2$) 
are given by
\[
\scalebox{.7}{
$\StratumKEqns$
}
\]
whose Gr\"obner basis can be calculated explicitly:
\[
\scalebox{.7}{$\StratumKGrob$}
\]
The results above are overwhelming both to compute
and to decipher.
We simplify the procedure using the following lemma:

\begin{lem}
\rm
Let $\Spec A$ be an irreducible
affine variety over $\bFp$
and let $B = A[s_1, \dots, s_m]/(f_1, \dots, f_m)$. Let $g: \mathrm{Spec}(B) \to \mathrm{Spec}(A)$ be the induced morphism of affine schemes. Suppose there exists a point $\mathfrak{p} \in \mathrm{Spec}(B)$ such that the Jacobian matrix $J = \left( \frac{\partial f_i}{\partial s_j} \right)_{1 \le i, j \le m}$ is invertible at $\mathfrak{p}$. Then the image $g(\mathrm{Spec}(B))$ contains an open neighborhood of $g(\mathfrak{p})$. In particular,
after replacing both $\Spec A$ and $\Spec B$ by non-empty open
subschemes, $\Spec B\to \Spec A$
is \MINOREDITED{\'etale and surjective}.
\label{lem:pointwise-check}
\end{lem}

\begin{proof}
Let $\cU_B \subseteq \mathrm{Spec}(B)$ be the principal open subscheme where the determinant of the Jacobian matrix $J$ is invertible. By hypothesis, the point $\mathfrak{p}$ lies in $\cU_B$.
The morphism $g$ restricted to $\cU_B$ is \'etale and thus
has open image.
\end{proof}

By Lemma \ref{lem:pointwise-check},
we are allowed to assign random values
to $x_*$ and $s_*$,
and check the invertibility of Jacobians
and the existence of solutions
after specializing to that particular
set of values of $\{x_*, s_*\}$.
For example, if we set
\[
\begin{matrix}
s_{0001}=2 &
x_{0001}=2 &
s_{0011}=4 &
x_{0011}=6 &
s_{0100}=4 \\
x_{0100}=5 &
s_{0110}=7&
x_{0110}=3&
s_{0120}=7&
x_{0120}=1
\end{matrix},
\]
then the Gr\"obner basis becomes
\[
G=\{\mu^2 - 2\mu - 15,\qquad b_{0010r},\qquad
b_{0010u} + 1/16\mu - 1/16,\qquad
\lambda_{0010} - 1
\}
\]
and the Jacobian matrix is
\[
J = \begin{pmatrix} 2\mu - 2 & 0 & 0 & 0 \\ 0 & 1 & 0 & 0 \\ \frac{1}{16} & 0 & 1 & 0 \\ 0 & 0 & 0 & 1 \end{pmatrix}
\]
which has determinant
\[
\det(J)=2(\mu-1)
\]
which is clearly non-trivial modulo $G$.

By the Newton's method or the multivariable Henselian's lemma,
if a polynomial equation has a solution
mod $\varepsilon$
and the Jacobian is a unit mod $\varepsilon$,
then the equation has a solution
in $\bFp[\![\varepsilon]\!]$.
\end{example}

{
\begin{tcolorbox}
\begin{algorithmic}[1]
    \Require A configuration $(K,\Psi)$ 
such that $\Psi_2\subset \Delta_K$, and a root $\delta_{\rot}\in \Psi_0$
for rotation.
\State Write down the universal family
$\bar\rho^{\univ}$ for $X_{\bar K_1\rtimes \bT}$.
\State Set $\bar\rho(\mu)\gets
\exp(\mu e_{-\delta_\rot})\bar\rho^{\univ} \exp(\mu e_{-\delta_\rot})$.
\State
Write $t_{\alpha}(\mu), s_{\alpha}(\mu), x_{\alpha,i}(\mu)$
for the parameters defining $\bar\rho(\mu)$.
\State Set
\[
t_\delta(\varepsilon, \mu)
:= 
\begin{cases}
t_\delta(\mu) & \delta\not\in \Psi_0 \\
1+\varepsilon\lambda_\delta & \delta\in \Psi_0
\end{cases}
\]
\[
x_{\alpha, i}(\varepsilon, \mu) = 
\begin{cases}
x_{\alpha,i}(\mu) & \alpha\not\in \Psi_0
\\
\varepsilon b_{\alpha,i}
&\alpha\MINOREDITED{\in} \Psi_0
\end{cases}
\]
\[
s_{\alpha}(\varepsilon, \mu) = s_{\alpha}(\mu),
\]
and treat
\[
\{\lambda_*, \mu, b_{*,*}\}
\]
as the free variables.
\State Plug $t_\alpha(\varepsilon, \mu), s_\alpha(\varepsilon, \mu), x_{\alpha, i}(\varepsilon, \mu)$
\MINOREDITED{into} the equations
$\{\frac{r_\alpha}{\varepsilon}:\alpha\in \Psi_2, \Ht(\alpha)>1\}$
and write
$P_\alpha$ for the constant terms (with respect to $\varepsilon$).
\State $G\gets$ Gr\"obner basis of $\{P_\alpha:
\alpha\in \Psi_2, \Ht(\alpha)>1
\}\cup\{\sum_{\delta}\lambda_\delta-1\}$
over the function field of $X_{\bar K_1\rtimes \bT}$.
\If{$G=\{1\}$} \Comment{There is no mod $\varepsilon$ solution.}
\State Report failure!
\EndIf
\State $G_{\text{aug}}\gets$ Gr\"obner basis of $\{P_\alpha:
\alpha\in \Psi_2, \Ht(\alpha)>1
\}\cup\{\sum_{\delta}\lambda_\delta-1,
\text{\MINOREDITED{Jacobian} of $\{P_\alpha\}$}\}$.
\If{$G=G_{\text{aug}} $} \Comment{Jacobian is not a unit.}
    \State Report failure!
\EndIf
\State Report success.
\end{algorithmic}
\end{tcolorbox}
\vspace{-5mm}
\renewcommand{\figurename}{Algorithm}
\captionof{figure}{The Rotational Deformational Test}
\label{box:rot}
}

\begin{thm}
\label{thm:ht1-codim0}
\rm
Each of the Configurations I-IV pass the Rotational Deformation Test
(c.f. Algorithm \ref{box:rot})
by using the following rotation roots:
\begin{itemize}
\item[(I, II, IV)] $\delta_\rot=0010$,
\item[(III)] $\delta_\rot=1000$.
\end{itemize}
\end{thm}

\begin{proof}
This follows from the exact computation in
Appendix~\ref{app:rotation-algorithm}.
We remark that for Configuration IV, there are three choices
$\delta_\rot\in\{1000,0010,0001\}$,
but only $0010$ works;
similarly for \MINOREDITED{Configuration} III, there are two choices
$\delta_\rot\in\{1000,0010\}$,
but only $1000$ works.
\end{proof}

\begin{remark}
We don't exclude the possibility that
for more complex groups like $E_6$,
we need to rotate using multiple roots
$\{\delta_{\rot}\}\subset \Psi_0$.
\end{remark}

\begin{cor}
When $\bG=\mathrm{F}_4$,
$\cX_{\bG,\red}^{\EG}$
is equidimensional of dimension
$[F:\Q_p]\#\Phi^+$
and has irreducible components identified
with
the irreducible components
of $\cW_{\bB}$.
\label{cor:F4-equidim}
\end{cor}

\section{The existence of crystalline lifts}

We use the integral fixed-type stacks of \cite{Lin25}.
The crystalline stack is obtained by imposing the closed condition $N=0$
on the stack of Breuil--Kisin--Fargues $\Gal_F$-modules.
Thus the proofs of
\cite[Proposition B.5.10 and Theorem B.6.3]{Lin25}
apply without change.

\begin{lem}
\label{lem:crystalline-proper-image}
Fix a bounded Hodge type and inertial type.  The morphism from the stack of
crystalline Breuil--Kisin--Fargues lattices of that type to the integral
Emerton--Gee stack is proper.  Its scheme-theoretic image is
$\Z_p$-flat, and a geometric point of its special fiber belongs to the image
if and only if the corresponding residual representation has a crystalline
lift of that type.
\end{lem}
\begin{proof}
This is the $N=0$ case of
\cite[Proposition B.5.10 and Theorem B.6.3]{Lin25}.
The only difference is that we use the closed crystalline substack.
\end{proof}

\begin{lem}
\rm
Suppose $\bG=\PGL_2$.
There \MINOREDITED{exist} finitely many
Hodge types $\{\underline{\lambda}_s|s\in S\}$
such that the special fiber
of the crystalline stack
$\bigcup_{s\in S}
\cX_{\PGL_2}^{\underline{\lambda}_s,\tau_{\triv}}$
contains all irreducible components
of $\cX_{\PGL_2,\red}^{\EG}$.
Here, $\tau_{\triv}$ stands for the trivial inertial type.
The lemma also holds after replacing $\bG$ by its Borel.
\label{lem:crys-PGL2}
\end{lem}

\begin{proof}
This is \cite[Theorem 6.4.4]{EG23}.
\end{proof}

\begin{lem}[Borel crystalline assembly]
\label{lem:Borel-crystalline-assembly}
Let $C\subset\cW_{\bB}$ be irreducible.  There is a regular Borel Hodge type
$\underline\lambda_C$ with trivial inertial type such that the crystalline
special fiber contains the component
$\overline{\dot\cX_{\bB}(C)}$.
\end{lem}
\begin{proof}
Apply Lemma~\ref{lem:crys-PGL2} to the simple roots
and choose a sufficiently regular common Hodge type.
Then use \cite[Corollary 5 and the proof of Lemma 19]{Lin25}.
Finally apply Lemma~\ref{lem:crystalline-proper-image}.
\end{proof}

\begin{cor}
Suppose $\bG$ is arbitrary.

(1) There \MINOREDITED{exist} finitely many
Hodge types $\{\underline{\lambda}_s|s\in S\}$,
such that the special fiber
of the crystalline stack
$\bigcup_{s\in S}
\cX_{\bB}^{\underline{\lambda}_s,\tau_{\triv}}$
maps surjectively on the Weil-Deligne stack
$\cW_{\bB}$.

(2)
The special fiber
of the crystalline stack
$\bigcup_{s\in S}
\cX_{\bB}^{\underline{\lambda}_s,\tau_{\triv}}$
contains 
$\overline{\dot \cX_{\bB}(C)}$
for all irreducible components $C\subset \cW_{\bB}$.
\label{cor:crys}
\end{cor}

\begin{proof}
(1) Apply Lemma~\ref{lem:Borel-crystalline-assembly} to every
irreducible component of the finite-type stack $\cW_{\bB}$ and take the finite
set of resulting Hodge types.  Each component is dominated by the indicated
crystalline special fiber, giving the asserted surjection.

(2)
The reduced special fiber of 
$\bigcup_{s\in S}
\cX_{\bB}^{\underline{\lambda}_s,\tau_{\triv}}$
is a union $\bigcup_{\cC\in S_\cC} \cC$ of (top-dimensional) irreducible components
of \MINOREDITED{the special fiber of the corresponding
$\cX_{\bB}$-stack}
that maps surjectively onto $\cW_{\bB}$.

For each irreducible component $C\subset \cW_{\bB}$,
there exists a unique
irreducible component
$C_{\cX}$ of $\cX_{\bB}$
(necessarily of the form $\overline{\dot \cX_{\bB}(C)}$)
satisfying the condition that
$C_{\cX}\to C$ is scheme-theoretically dominant.
Thus $C_{\cX}\in S_{\cC}$
for all $C$ by part (1).
\end{proof}

\begin{thm}
\label{thm:maindup}
If $\bG=\mathrm{F}_4$, then all $\Gal_F\to \bG(\bFp)$
have a crystalline lift.

If $\bG$ is a reductive group satisfying the standing hypotheses
and $[F:\Q_p]>\frac{\dim \bG}{2}$, then all $\Gal_F\to \bG(\bFp)$
have a crystalline lift.
\end{thm}

\begin{proof}
\MAJOREDIT{First suppose that the semisimple rank is at least two. Since
$\dim\bG=2\#\Phi^++\operatorname{rank}\bG$, the strict inequality
$[F:\Q_p]>\dim\bG/2$ implies
$[F:\Q_p]\ge\#\Phi^++2$. Proposition~\ref{prop:TBM} then shows that every
irreducible component of $\cX_{\bG,\red}^{\EG}$ is one of the components
dominating a component of $\cW_{\bB}$. Corollary~\ref{cor:crys} places each of
these components in the special fiber of a crystalline stack. Every point lies
on an irreducible component and hence in that special fiber; by the pointwise
characterization in Lemma~\ref{lem:crystalline-proper-image}, it has a
crystalline lift.

For a rank-one factor the same assertion is Corollary~\ref{cor:crys}, based on
Lemma~\ref{lem:crys-PGL2}; for a torus it is the corresponding rank-one
character statement. For a general reductive group, choose the Hodge type
factor by factor using fundamental coweights and take their Cartesian product.
This construction is compatible with the central torus and therefore gives a
single Hodge type for the original group.

When $\bG=\mathrm F_4$, use Corollary~\ref{cor:F4-equidim} in place of the
large-degree application of Proposition~\ref{prop:TBM}, and conclude in the
same way from Corollary~\ref{cor:crys}.}
\end{proof}

\newpage
\phantomsection
\addcontentsline{toc}{part}{Appendices}

~

\appendix

\section{Bounded cone and $F_4$ strata algorithms}
\label{app:cone-model-algorithms}

This appendix gives the exhaustive certification.  Root data are generated
from the $F_4$ Cartan matrix.  The reviewed fine-pair list is checked directly
and reconstructed when absent.  Roots are coefficient tuples in Bourbaki
order, and all rank calculations are exact.

\subsection{Field-degree bound}

For each of the $105$ downward-closed complements $K$, the program builds the
complete grouped cone matrix for the largest possible target set.  An integral
nonzero minor certifies its rank.  The resulting dimension estimate is affine
in $d=[F:\Q_p]$ and is automatic for $d\ge2$.  The contents of all witness
minors have prime support contained in $\{2,3\}$, so the estimate is unchanged
in the standing range $p>12$.

\begin{tcolorbox}
\begin{algorithmic}[1]
\Require The $F_4$ Cartan matrix.
\State Generate $\Phi^+$, its root order, and all $105$ complements $K$.
\For{each $K$}
  \State Build the full grouped cone matrix at $\Psi_2=\Phi^+$.
  \State Find an integral maximal nonzero minor and its bad primes.
  \State Solve the resulting affine dimension inequality for $d$.
\EndFor
\State Return the largest exceptional degree range.
\end{algorithmic}
\end{tcolorbox}
\renewcommand{\figurename}{Algorithm}
\captionof{figure}{A priori field-degree bound}
\label{alg:F4-degree-bound}

Thus only $d=1$ requires enumeration.

\subsection{Fine pairs}

At degree one, a candidate level set is a subset $S\subset\Phi^+$ on which a
linear functional can satisfy $f(\alpha)=1$ for every $\alpha\in S$.  Every
such $S$ contains a linearly independent subset with the same affine span.
This gives exhaustive small seeds and avoids scanning all $2^{24}$ subsets.
We keep the resulting list of $4862$ fine pairs.

\begin{tcolorbox}
\begin{algorithmic}[1]
\Require $\Phi^+$, the degree bound $d=1$, and the reviewed pair list $P$.
\If{$P$ is absent}
  \State Generate the independent affine-consistent seeds.
  \State Expand them by BFS, stopping at affine inconsistency or the degree
  bound, and write the resulting canonical list $P$.
\EndIf
\For{$(\Psi_0,\Psi_2)\in P$}
  \State Recompute the difference lattice, $\Psi_0$, and the closure axioms.
  \State Regard the pair as a terminal BFS state; do not generate children.
\EndFor
\For{each independent subset of the required degree}
  \State Check that its fine pair occurs in $P$.
\EndFor
\State Return $P$.
\end{algorithmic}
\end{tcolorbox}
\renewcommand{\figurename}{Algorithm}
\captionof{figure}{Fine-pair certification and fallback BFS}
\label{alg:F4-bounded-bfs}

There are $4862$ distinct fine pairs.  All $3002$ relevant independent subsets
occur in this list.  The cutoff is sound because adjoining a root to $\Psi_2$
strictly increases minimal degree.

\subsection{Exhaustive classification}

We apply the cone calculation to the $105$ possibilities for $K$ and the
$4862$ fine pairs.  Exactly four cases have $\varepsilon=0$; these are
Configurations I--IV.  All integral rank witnesses remain nonzero for
$p>12$.

For the four displayed grouped cone matrices the tuples
$(\#\mathrm{variables},\#\mathrm{relations},\operatorname{rank},
\mathrm{dimension})$ are
\[
(28,17,17,11),\quad(29,17,17,12),\quad
(30,14,14,16),\quad(32,8,7,25).
\]
Relations with the same target are one summed equation, never separate ideal
generators.  The implementation details and reviewed data are available at
\url{https://github.com/mocham/AlgEG/tree/main/code/v4}.

\section{Root-poset cycle algorithms}
\label{app:root-poset-cycles}

We describe separately the absolute and non-Borel computations.  In both, a
source-leaf contraction removes a source of valency one together with its
unique target; isolated sources are discarded.  A surviving Jacobian with $r$
target rows and $c\ge r$ source columns is tested by enumerating the
$\binom cr$ source-column subsets.  Success means that one corresponding
determinant is a nonzero monomial in the simple-root variables.

\subsection{Absolute root posets}

The absolute routine applies the following finite test independently at each
height.

\begin{tcolorbox}
\begin{algorithmic}[1]
\Require An exceptional root system $\Phi$ and its Chevalley constants.
\State Construct the directed root poset and its bipartite height layers
$\fB(\Phi^+,\Delta,h)$.
\For{each height $h$ and each target selection $T$}
  \State Retain the full source layer and only edges ending in $T$.
  \State Repeatedly contract source leaves.
  \If{the graph contracts to the empty graph}
    \State Record a trivial certificate and continue.
  \EndIf
  \State Form the Jacobian
  $J_{\alpha,\beta}=N_{\beta,\alpha-\beta}x_{\alpha-\beta}$.
  \For{each set $C$ of $\#T$ source columns}
    \State Compute $d_C\gets\det J_C$.
    \If{$d_C$ is a nonzero monomial}
      \State Record $C$ and $d_C$ and continue with the next target selection.
    \EndIf
  \EndFor
  \State If no witness was found, report failure.
\EndFor
\end{algorithmic}
\end{tcolorbox}
\renewcommand{\figurename}{Algorithm}
\captionof{figure}{Absolute admissibility test}
\label{alg:absolute-cycles}

For the exceptional types the numerical summary is:
\[
\begin{array}{c|rr}
\toprule
\text{Type}&\text{target selections}&\text{failures}\\
\midrule
F_4&48&0\\
E_6&138&0\\
E_7&384&0\\
E_8&1196&0\\
\bottomrule
\end{array}
\]
The nontrivial determinants, up to orientation-dependent signs, are:
\[
\begin{array}{c|c|l}
\toprule
\text{Type}&\text{height}&\text{Jacobian minor}\\
\midrule
F_4&3&3x_{0001}x_{0010}x_{1000}\\
E_6&2&2x_{000010}x_{001000}x_{010000}\\
E_6&3&3x_{000001}x_{000010}x_{001000}x_{010000}x_{100000}\\
E_7&2&-2x_{0000100}x_{0010000}x_{0100000}\\
E_7&3&-3x_{0000010}x_{0000100}x_{0010000}x_{0100000}x_{1000000}\\
E_7&4&2x_{0000001}x_{0001000}x_{0010000}x_{0100000}\\
E_7&8&-2x_{0000001}x_{0000010}x_{0000100}x_{0010000}\\
E_8&2&-2x_{00001000}x_{00100000}x_{01000000}\\
E_8&3&-3x_{00000100}x_{00001000}x_{00100000}x_{01000000}x_{10000000}\\
E_8&4&2x_{00000010}x_{00010000}x_{00100000}x_{01000000}\\
E_8&5&-5x_{00000001}x_{00000010}x_{00000100}x_{00010000}x_{00100000}x_{01000000}x_{10000000}\\
E_8&8&-2x_{00000010}x_{00000100}x_{00001000}x_{00100000}\\
E_8&9&3x_{00000001}x_{00000100}x_{00001000}x_{00010000}x_{00100000}x_{10000000}\\
E_8&14&2x_{00000010}x_{00001000}x_{00010000}x_{01000000}\\
\bottomrule
\end{array}
\]

For $D_n$, all cycles are isomorphic to $\vec C_6$.  With the standard
ordering $\Delta_{D_n}=\{\delta_1,\ldots,\delta_n\}$, they are given by
\[
\begin{aligned}
\beta_1&=\sum_{i=n-h-1}^{n-2}\delta_i,&
\beta_2&=\sum_{i=n-h}^{n-2}\delta_i+\delta_{n-1},&
\beta_3&=\sum_{i=n-h}^{n-2}\delta_i+\delta_n,
\end{aligned}
\]
and their Jacobian coefficients are $2$ or $-2$.

\subsection{Non-Borel root posets}

Theorem~\ref{thm:twist-C6} already has the correct disjunction: a relative
cycle is either non-obstructing or its associated absolute graph is
admissible.  Thus the non-obstructing condition is applied before any absolute
Jacobian test.

\begin{tcolorbox}
\begin{algorithmic}[1]
\Require A Levi $\bM$, an elliptic Weyl element $w$, its relative root orbits,
and the non-obstructing singleton and pair relation.
\State Construct the $w$-relative root poset and its bipartite height layers.
\State Remove individually non-obstructing targets.
\For{each relative height layer}
  \State Contract source leaves.
  \For{each target subset $T$ containing no non-obstructing pair}
    \State Restrict to $T$ and contract source leaves again.
    \If{the graph contracts to the empty graph}
      \State Record the contraction and continue.
    \EndIf
    \State Replace every relative source and target by all absolute roots in
    its $w$-orbit.
    \State Add every absolute edge whose target-source difference is a
    positive root, and contract source leaves.
    \If{the absolute graph is nonempty}
      \State Apply Algorithm~\ref{alg:absolute-cycles}, testing combinations
      of absolute source columns.
    \EndIf
  \EndFor
\EndFor
\State Report failure if any surviving absolute graph has no monomial maximal
minor.
\end{algorithmic}
\end{tcolorbox}
\renewcommand{\figurename}{Algorithm}
\captionof{figure}{Non-Borel cycle test}
\label{alg:twisted-cycles}

The computation generates all exceptional Levi/Weyl configurations from the
Cartan matrices and covers:
\[
\begin{array}{c|rr}
\toprule
\text{Type}&\text{configurations}&\text{failures}\\
\midrule
F_4&19&0\\
E_6&68&0\\
E_7&145&0\\
E_8&304&0\\
\bottomrule
\end{array}
\]
Consequently all $F_4$ relative cycles are non-obstructing or contractible,
and every surviving E-type associated absolute graph is admissible.  The
reviewed data are available at
\url{https://github.com/mocham/AlgEG/tree/main/code/v4}.

\section{Characteristic-uniform rotation certificates}
\label{app:rotation-algorithm}

The computation is over $\Q$ and records enough integral data to spread each
smooth point to every characteristic $p>12$.

For a rotation root $\delta_{\rot}$ and $\alpha\in\Psi_2$, the code follows
the complete string $\alpha+n\delta_{\rot}$.  At step $n$ it multiplies the
current coefficient by
$N_{-\delta_{\rot},\alpha+n\delta_{\rot}}/n$.  Thus, for Configuration I,
\[
c_{0100}\mapsto c_{0100}+2\mu c_{0110}-\mu^2c_{0120},\qquad
c_{0001}\mapsto c_{0001}-\mu c_{0011}.
\]

\begin{tcolorbox}
\begin{algorithmic}[1]
\Require $(K,\Psi_0,\Psi_2)$ and a rotation root $\delta_{\rot}$.
\State Generate all Chevalley constants and root-string substitutions.
\State Form the divided $\varepsilon$-deformation equations over $\Q$.
\For{the deterministic sequence of small rational specializations and points}
  \State Evaluate every equation exactly.
  \State Evaluate the square Jacobian determinant exactly.
  \If{the equations vanish and the determinant is nonzero}
    \State Factor all denominators, required nonzero coordinates, and the
    determinant; record their prime support and stop.
  \EndIf
\EndFor
\State Fail unless the recorded bad primes are all at most $12$.
\end{algorithmic}
\end{tcolorbox}
\renewcommand{\figurename}{Algorithm}
\captionof{figure}{Rational rotation witness}
\label{alg:rotation-certificate}

The exact output is:
\[
\begin{array}{c|c|r|c}
\toprule
\text{Configuration}&\delta_{\rot}&\det J&\text{bad primes}\\
\midrule
\mathrm{I}&0010&432&2,3\\
\mathrm{II}&0010&704&2,3,5,11\\
\mathrm{III}&1000&15360&2,3,5\\
\mathrm{IV}&0010&895795200&2,3,5\\
\bottomrule
\end{array}
\]
Every equation evaluates to zero at its recorded rational point.  Inverting
the displayed finite set makes that point a smooth section, so reduction gives
a smooth point in every characteristic $p>12$.  The reviewed data are
available at
\url{https://github.com/mocham/AlgEG/tree/main/code/v4}.

\section{Computational source}
\label{app:computational-certification}

The implementation and reviewed data are available at
\begin{center}
\url{https://github.com/mocham/AlgEG/tree/main/code/v4}.
\end{center}
The code generates the root systems, Chevalley constants, order ideals,
elliptic Weyl orbits, and twisted configurations used in the preceding
appendices.  The fine-pair list is checked as in
Algorithm~\ref{alg:F4-bounded-bfs}; it is reconstructed from first principles
only when the reviewed data are absent.

The precise source map, commands, runtimes, certificate formats, checksums,
and artifact inventory are recorded in
\begin{center}
\url{https://github.com/mocham/AlgEG/blob/main/code/README.md}.
\end{center}

\printbibliography

\end{document}